\documentclass[a4paper, 11pt, reqno]{amsart}
\usepackage{etex}
\usepackage{xypic}
\usepackage{amssymb}
\usepackage{amsthm}
\usepackage{bm}
\usepackage{latexsym}
\usepackage{amsmath}
\usepackage{eufrak}
\usepackage{mathrsfs}
\usepackage{amscd}
\usepackage[all]{xy}
\usepackage[usenames]{color}
\usepackage{graphicx}
\usepackage{color}
\usepackage{caption}
\usepackage{amscd}
\theoremstyle{plain}
\usepackage[usenames,dvipsnames]{pstricks}
\usepackage{epsfig}
\usepackage{pst-grad} 
\usepackage{pst-plot} 
\usepackage[space]{grffile} 
\usepackage{etoolbox} 
\makeatletter 
\patchcmd\Gread@eps{\@inputcheck#1 }{\@inputcheck"#1"\relax}{}{}
\makeatother
\usepackage{longtable}
\usepackage{tikz}
\usetikzlibrary{arrows}
\usetikzlibrary{decorations.pathreplacing}

\usepackage{longtable}

\numberwithin{equation}{section}

\newtheorem{theorem}{Theorem}[section]
\newtheorem{proposition}[theorem]{Proposition}
\newtheorem{lemma}[theorem]{Lemma}
\newtheorem{corollary}[theorem]{Corollary}

\theoremstyle{definition}

\newcommand{\appsection}[1]{\let\oldthesection\thesection
\renewcommand{\thesection}{Appendix \oldthesection}
\section{#1}\let\thesection\oldthesection}

\newtheorem{definition}[theorem]{Definition}

\theoremstyle{remark}

\newtheorem{remark}[theorem]{Remark}
\newtheorem{example}[theorem]{Example}

\DeclareMathOperator{\spec}{Spec}

\def\Z{{\mathbb{Z}}}
\def\Q{{\mathbb{Q}}}
\def\C{{\mathbb{C}}}
\def\P{{\mathbb{P}}}
\def\A{{\mathbb{A}}}

\def\O{{\mathcal{O}}}

\def\L{{\mathcal{L}}}

\def\SS{{\mathcal{S}}}

\begin{document}
\bibliographystyle{amsplain}
\title{Characterization of Koll\'ar surfaces}
\author{\textrm{Giancarlo Urz\'ua and Jos\'e Ignacio Y\'a\~nez}}

\address{Facultad de Matem\'aticas\\ Pontificia Universidad Cat\'olica de Chile\\ Campus San Joaqu\'in\\ Avenida Vicu\~na Mackenna 4860\\ Santiago\\ Chile.}


\email{urzua@mat.puc.cl}
\email{jgyanez@uc.cl}

%
%
%

\begin{abstract}
Koll\'ar introduced in \cite{Ko08} the surfaces $$(x_1^{a_1}x_2+x_2^{a_2}x_3+x_3^{a_3}x_4+x_4^{a_4}x_1=0)\subset \P(w_1,w_2,w_3,w_4)$$ where $w_i=W_i/w^*$, $W_i=a_{i+1}a_{i+2}a_{i+3}-a_{i+2}a_{i+3}+a_{i+3}-1$, and $w^*=$gcd$(W_1,\ldots,W_4)$. The aim was to give many interesting examples of $\Q$-homology projective planes. They occur when $w^*=1$. For that case, we prove that Koll\'ar surfaces are Hwang-Keum \cite{HK12} surfaces. For $w^*>1$, we construct a geometrically explicit birational map between Koll\'ar surfaces and cyclic covers $z^{w^*}=l_1^{a_2 a_3 a_4} l_2^{-a_3 a_4} l_3^{a_4} l_4^{-1}$, where $\{l_1,l_2,l_3,l_4\}$ are four general lines in $\P^2$. In addition, by using various properties on classical Dedekind sums, we prove that:

\begin{itemize}


\item[(a)] For any $w^*>1$, we have $p_g=0$ iff the Koll\'ar surface is rational. This happens when $a_{i+1} \equiv 1$ or $a_{i}a_{i+1} \equiv -1 ($mod $w^*)$ for some $i$.

\item[(b)] For any $w^*>1$, we have $p_g=1$ iff the Koll\'ar surface is birational to a K3 surface. We classify this situation.

\item[(c)] For $w^*>>0$, we have that the smooth minimal model $S$ of a generic Koll\'ar surface is of general type with $K_{S}^2/e(S) \to 1$.
\end{itemize}
\end{abstract}
\maketitle
\tableofcontents
\section{Introduction} \label{s0}

The ground field is $\C$. Let $n \geq 3$ be an integer, and let $a_1,\ldots, a_n$ be positive integers such that there is no $(a_i,a_{i+2},\ldots,a_{i+n-2}) = (1,\ldots,1)$ when $n$ is even. The indices are and will be taken modulo $n$. For every $1 \leq i \leq n$, we define the positive integers $$W_i := \sum_{j=1}^n (-1)^{j-1} \prod_{l=i+j}^{i+n-1} a_l \ \ \ \text{and} \ \ \ D:= \prod_{l=1}^n a_l + (-1)^{n-1}.$$ For example, for $n=4$ we have $$W_i=a_{i+1}a_{i+2}a_{i+3}-a_{i+2}a_{i+3}+a_{i+3}-1 \ \ \ \text{and} \ \ \ D= a_1 a_2 a_3 a_4-1.$$ We also define $$w^*:= gcd(W_1,\ldots,W_n).$$ Then $w^*=gcd(W_i,W_{i+1})=gcd(W_i,D)$ since $a_i W_i + W_{i+1}=D$ for all $i$.

Set $$w_i:= \frac{W_i}{w^*} \ \ \ \text{and} \ \ \ d:= \frac{D}{w^*}.$$ Notice that $gcd(a_i,w^*)=1$ for all $i$.

The \textit{Koll\'ar hypersurface} \cite{Ko08} of type $(a_1,\ldots,a_n)$ is $$X(a_1,\ldots,a_n):=(x_1^{a_1}x_2+x_2^{a_2}x_3+\ldots+x_n^{a_n}x_1=0)\subset \P(w_1,\ldots,w_n)$$

Let $0<\mu_i<w^*$ be such that $\mu_i \equiv (-1)^{i+1} \prod_{l=i+1}^{i+n-1} a_l \, ($mod $w^*)$. We consider the normal projective variety $Y'$ given by the $w^*$-th root cover $Y' \to \P^{n-2}=\{y_1+\ldots+y_n=0\} \subset \P^{n-1}$ branch along $\{y_1^{\mu_1} \cdots y_n^{\mu_n}=0 \}$; see Section \ref{s1} for precise definitions. The map $\psi$ associated to the linear system $|x_1^{a_1}x_2,\ldots,x_n^{a_n}x_1|$ in the Koll\'ar hypersurface shows that the varieties $X(a_1,\ldots,a_n)$ and $Y'$ are birational; this is worked out in Section \ref{s1}.

In this paper we consider in detail the case $n=4$; the surface $X=X(a_1,\ldots,a_4)$ will be called \textit{Koll\'ar surface}. First, we note that Koll\'ar surfaces are birational to infinitely many Koll\'ar surfaces with gcd$(w_i,w_{i+2})=1$ and $a_i>1$ (see Theorem \ref{gcd}), and so we assume these numerical conditions to simplify the exposition. Section \ref{s2} is devoted to prove the following.

\begin{theorem}
There is a configuration $\Gamma$ of 6 rational curves in $X$ such that if $\hat{X} \to X$ is a log resolution of $(X,\Gamma)$, then $\hat{X} \to X \stackrel{\psi}{\dashrightarrow} \P^2$ is a morphism which factors through $Y' \to \P^2$ via a birational morphism $\hat{X} \to Y'$.
\end{theorem}

The aim of Koll\'ar surfaces \cite{Ko08} was to give examples of \textit{$\Q$-homology projective planes} ($\Q$HPP) with ample canonical class. This occurs for $w^*=1$ after contracting $(x_1=x_3=0)$ and $(x_2=x_4=0)$ in $X$, when possible. This contraction gives a $\Q$HPP with two cyclic quotient singularities, and, when $a_i \geq 4$ for all $i$, the canonical class is ample. On the other hand, Hwang and Keum constructed in \cite{HK12} a series of examples of $\Q$HPP with ample canonical class and same singularities as Koll\'ar examples. In Section \ref{s3} we prove the following.

\begin{theorem}
Koll\'ar $\Q$-homology projective planes are Hwang-Keum surfaces.
\end{theorem}

As an intriguing problem, we point out that $\Q$HPP with ample canonical class and cyclic quotient singularities have not yet been classified. The number of possible singularities is at most four, and examples with one, two, and three singularities have been constructed. It is conjectured that the case of four singularities is impossible; see \cite{Ko08,HK12}.

In Section \ref{s4} we write down formulas for the invariants of Koll\'ar surfaces via $Y'$ when $w^*>1$. Particularly interesting is the geometric genus, which depends on classical Dedekind sums on the exponents $a_i$'s. For example, by comparing the two models $X$ and $Y'$, we write down an identity for Dedekind sums in Corollary \ref{identity}. More importantly, in Section \ref{s5} we use new bounds on their values, essentially due to Girstmair \cite{Girs16}, to prove the following (see Theorem \ref{pg0}, Theorem \ref{pg1}, and Theorem \ref{generic}).

\begin{theorem}
For $w^*>1$, we have that

\begin{itemize}


\item[(a)] $p_g=0$ if and only if the Koll\'ar surface is rational. This happens when $a_{i} \equiv 1$ or $a_{i}a_{i+1} \equiv -1$ modulo $w^*$ for some $i$.

\item[(b)] $p_g=1$ if and only if the Koll\'ar surface is birational to a K3 surface. We classify this situation in $8$ cases (see Table 1).

\item[(c)] For $w^*>>0$, the smooth minimal model $S$ of a generic Koll\'ar surface is of general type with $K_{S}^2/e(S) \to 1$, where $K_{S}$ is the canonical class, and $e(S)$ is the topological Euler characteristic.
\end{itemize}
\end{theorem}

Moreover we note that any $p_g$ is realizable by some Koll\'ar surface (Proposition \ref{anypg}), and that given $m>0$ there exists an $N$ such that $p_g >m$ if $w^*>N$ (Lemma \ref{pgBound}). At the end, we give explicit examples of Kodaira dimension $1$ elliptic fibrations (Example \ref{ex1}) and surfaces of general type (Example \ref{ex2}), arising as Koll\'ar surfaces for $w^*$ arbitrarily large.

\subsection*{Acknowledgements}
We would like to thank Kurt Girstmair for useful discussions on Dedekind sums. This is part of the Master's thesis of the second author at the Pontificia Universidad Cat\'olica de Chile. A large part of this paper was written while the authors were visiting the Department of Mathematics of the University of Massachusetts at Amherst. We are thankful for the hospitality. The authors were supported by the FONDECYT regular grant 1150068.

\section{Koll\'ar hypersurfaces} \label{s1}

Koll\'ar proves in \cite[Thm.39]{Ko08} the following.

\begin{theorem}
\mbox{}
\begin{itemize}
\item[(1)] The weighted projective space $\P(w_1,\ldots,w_n)$ is well formed, and its singular set has dimension $\leq [n/2]-1$.

\item[(2)] The hypersurface $X(a_1,\ldots,a_n)$ is quasi-smooth, and $\P(w_1,\ldots,w_n) \setminus X(a_1,\ldots,a_n)$ is smooth.

\item[(3)] If $w^*=1$, then $X(a_1,\ldots,a_n)$ is birational to $\P^{n-2}$.

\end{itemize}
\label{kollarTheorem}
\end{theorem}

To prove (3) above, Koll\'ar uses the linear system $|x_1^{a_1}x_2,x_2^{a_2}x_3,\ldots,x_{n}^{a_n}x_1|$. In general, this linear system defines a rational map $$\psi \colon \P(w_1,\ldots,w_n) \dashrightarrow \P_{y_1,\ldots,y_n}^{n-1}$$ given by $y_i=x_i^{a_i} x_{i+1}$.

\begin{proposition}
The rational map $\psi$ defines the field extension $$\C(y_1/y_n,\ldots,y_{n-1}/y_n) \subset \C(y_1/y_n,\ldots,y_{n-1}/y_n)[z]/(z^{w^*}-f/y_n^{W_1})$$ where $z=x_1^d/y_n^{w_1}$ and $f=y_1^{a_2 a_3 \cdots a_n} y_2^{-a_3 \cdots a_n} y_3^{a_4 \cdots a_n} \cdots y_{n-1}^{(-1)^{n-2} a_n} y_n^{(-1)^{n-1}}$.
\label{cyclic}
\end{proposition}

\begin{proof}
At the affine cover level, the field extension induced by $\psi$ is $$ \C(y_1,\ldots,y_n) \subset \C(y_1,\ldots,y_n)[x_1]/(x_1^D-f)$$ where the other variables $x_2, \ldots, x_n$ can be written using $y_1,\ldots,y_n,x_1$. The action of $\C^*$ compatible with the map is: Given $\lambda \in \C^*$, $y_i \mapsto \lambda^d y_i$ and $x_i \mapsto \lambda^{w_i} x_i$. Then the rational map $\psi$ is determined by $$\big(\C(y_1,\ldots,y_n)\big)^{\C^*} \subset \big(\C(y_1,\ldots,y_n)[x_1]/(x_1^D-f) \big)^{\C^*}.$$ Notice that $\big(\C(y_1,\ldots,y_n)\big)^{\C^*}=\C(y_1/y_n,\ldots,y_{n-1}/y_n)$, and that $z=x_1^d/y_n^{w_1}$ is a $\C^*$-invariant element such that $z^{w^*}-f/y_n^{W_1}=0$. Since geometrically the map $\psi$ has degree $w^*$, then $$\big(\C(y_1,\ldots,y_n)[x_1]/(x_1^D-f) \big)^{\C^*} = \C(y_1/y_n,\ldots,y_{n-1}/y_n)[z]/(z^{w^*}-f/y_n^{W_1}).$$
\end{proof}

\begin{corollary}
The corresponding restriction map $$\psi|_X \colon X(a_1,\ldots,a_n) \dashrightarrow \P^{n-2} =\{y_1+\ldots+y_n=0 \}$$ is cyclic of degree $w^*$ totally branch along $(y_1\cdots y_n=0) \subset \P^{n-2}$.
\end{corollary}

In this way, we can write down another normal projective model $Y'$ of $X(a_1,\ldots,a_n)$ using a $w^*$-th root cover as described in \cite{EV92}.

As in the introduction, let $0<\mu_i<w^*$ be such that $$\mu_i \equiv (-1)^{i+1} \prod_{l=i+1}^{i+n-1} a_l \, (\text{mod} \ \ w^*).$$ In $\P^{n-2} =\{y_1+\ldots+y_n=0 \}$, we write $L_i:= \{y_i=0\}$, and so $$ \O_{\P^{n-2}}(t)^{\otimes w^*} \simeq \O_{\P^{n-2}}( \mu_1 L_1 + \ldots+ \mu_n L_n),$$ where $t w^* = \sum_{i=1}^n \mu_i$. Then $$Y_0 :=\spec_{\P^{n-2}} \Big( \bigoplus_{i=0} ^{w^*-1} \O_{\P^{n-2}}(-ti) \Big) \to \P^{n-2}$$ is the cyclic cover given by $z^{w^*}-f/y_n^{W_1}$ above. We want to consider the normalization of $Y_0$. As in \cite{EV92}, we define the line bundles $\L^{(i)}$ on $\P^{n-2}$ as $$ \L^{(i)}:= \O_{\P^{n-2}}(ti) \otimes \O_{\P^{n-2}}\Bigl( - \sum_{j=1}^n \Bigl[\frac{\mu_j \ i}{w^*}\Bigr] L_j \Bigr)$$ for $i \in{\{0,1,...,w^*-1 \}}$, where $[x]$ is the integer part of $x$. Then, the normalization of $Y_0$ is $Y':= \spec_{\P^{n-2}} \Big( \bigoplus_{i=0}^{w^*-1} {\L^{(i)}}^{-1} \Big)$; see \cite[Cor. 3.11]{EV92}. Notice that gcd$(\mu_i,w^*)=1$, and so this cyclic morphism is totally branch at the $L_i$'s.

\begin{corollary}
There is a birational map $X(a_1,\ldots,a_n) \dashrightarrow Y'$.
\label{bir}
\end{corollary}

In the next section we describe explicitly this birational map for $n=4$.

\section{Explicit birational map for Koll\'ar surfaces} \label{s2}

From now on we concentrate in the case of Koll\'ar surfaces, where $n=4$. We will be working with cyclic quotient surface singularities, which we now review. A cyclic quotient singularity $S$, denoted by $\frac{1}{m}(a,b)$, is a germ at the origin of the quotient of $\C^2$ by the action $(x,y)\mapsto (\zeta^a x, \zeta^b y)$, where $\zeta$ is a primitive $m$-th root of $1$, and $a,b$ are integers coprime to $m$; cf. \cite[III \S5]{BHPV04}. Let $0<q<m$ be such that $aq-b\equiv 0$ modulo $m$. Then, $\frac{1}{m}(a,b) = \frac{1}{m}(1,q)$. Let $\sigma \colon \tilde{S} \to S$ be the minimal resolution of $S$. Figure \ref{exdiv} shows the exceptional curves $E_i=\P^1$ of $\sigma$, for $1 \leq i \leq s$, and the strict transforms $E_0$ and $E_{s+1}$ of $(y=0)$ and $(x=0)$ respectively.

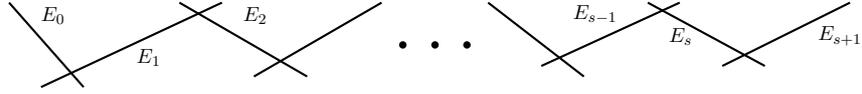
\begin{figure}[htbp]
\psscalebox{.7 .7} 
{
\begin{pspicture}(0,-0.81812537)(15.975043,0.81812537)
\psline[linecolor=black, linewidth=0.04](0.015043564,0.79996765)(1.4150436,-0.8000324)
\psline[linecolor=black, linewidth=0.04](0.6150436,-0.8000324)(4.0150437,0.79996765)
\psline[linecolor=black, linewidth=0.04](3.2150435,0.79996765)(5.6150436,-0.6000323)
\psline[linecolor=black, linewidth=0.04](4.6150436,-0.6000323)(7.0150437,0.79996765)
\psdots[linecolor=black, dotsize=0.14](7.4150434,0.0)
\psdots[linecolor=black, dotsize=0.14](8.015043,0.0)
\psdots[linecolor=black, dotsize=0.14](8.615044,0.0)
\rput[bl](0.6150436,0.39996764){$E_0$}
\rput[bl](2.4150436,-0.40003234){$E_1$}
\rput[bl](4.4150434,0.39996764){$E_2$}
\psline[linecolor=black, linewidth=0.04](9.015043,0.79996765)(10.815043,-0.6000323)
\psline[linecolor=black, linewidth=0.04](10.015043,-0.40003234)(12.615044,0.79996765)
\psline[linecolor=black, linewidth=0.04](12.015043,0.79996765)(14.215044,-0.40003234)
\psline[linecolor=black, linewidth=0.04](13.415044,-0.40003234)(15.815043,0.79996765)
\rput[bl](10.615044,0.39996764){$E_{s-1}$}
\rput[bl](12.415044,0.0){$E_{s}$}
\rput[bl](15.215044,0.0){$E_{s+1}$}
\end{pspicture}
}
\caption{Exceptional divisors over $\frac{1}{m}(1,q)$, $E_0$ and
$E_{s+1}$} \label{exdiv}
\end{figure}

The numbers $E_i^2=-b_i$ are computed using the {\em Hirzebruch-Jung continued fraction}
$$ \frac{m}{q} = b_1 - \frac{1}{b_2 - \frac{1}{\ddots - \frac{1}{b_s}}} =: [b_1, \ldots ,b_s].$$
We denote $|[b_1, \ldots ,b_s]| := m$.  This continued fraction defines the sequence of integers $$ 0=\beta_{s+1} < 1=\beta_s < \ldots < q=\beta_1 < m= \beta_0 $$ where $\beta_{i+1}= b_{i}\beta_i - \beta_{i-1}$. In this way, $\frac{\beta_{i-1}}{\beta_{i}}=[b_i,\ldots,b_s]$. Partial fractions $\frac{\alpha_i}{\gamma_i} =[b_1,\ldots,b_{i-1}]$ are computed through the sequences $$ 0=\alpha_0 < 1=\alpha_1 < \ldots < q^{-1}=\alpha_s < m= \alpha_{s+1},$$ where $\alpha_{i+1}=b_i\alpha_{i} - \alpha_{i-1}$ ($q^{-1}$ is the integer such that $0<q^{-1}<m$ and $q q^{-1} \equiv 1 ($mod $m)$), and $\gamma_0=-1$, $\gamma_1=0$, $\gamma_{i+1}=b_i \gamma_i - \gamma_{i-1}$. We have $\alpha_{i+1}\gamma_i - \alpha_i \gamma_{i+1}=-1$, $\beta_i = q \alpha_i - m \gamma_i$, and $\frac{m}{q^{-1}}=[b_s,\ldots,b_1]$. These numbers appear in the pull-back formulas \begin{equation} \sigma^*\big((y=0)\big) = \sum_{i=0}^{s+1} \frac{\beta_i}{m} E_i, \ \ \ \text{and} \ \ \ \sigma^*\big((x=0)\big)= \sum_{i=0}^{s+1} \frac{\alpha_i}{m} E_i,\label{pullbackFormula}\end{equation} and $K_{\tilde{S}}
\equiv \sigma^*(K_{S}) + \sum_{i=1}^s (-1 +\frac{\beta_i+\alpha_i}{m}) E_i$.

\bigskip

\bigskip

Let $X(a_1,a_2,a_3,a_4)$ be a Koll\'ar surface. Let $$p_1=(1:0:0:0), \ p_2=(0:1:0:0), \ p_3=(0:0:1:0), \ p_4=(0:0:0:1).$$

\begin{proposition}
The surface $X(a_1,a_2,a_3,a_4)$ is normal, and it has only singularities of type $\frac{1}{w_i}(w_{i+2},w_{i+3})$ at the points $p_i$ when gcd$(w_i,w_{i+2}) = 1$, and of type $\frac{1}{t_i}(t_{i+2},w_{i+3})$ when gcd$(w_i,w_{i+2}) = h>1$, where $w_j = ht_j$.
\label{singularityType}
\end{proposition}

\begin{proof}
Here we follow the idea in \cite[\S 10.1]{Ian00}. Without loss of generality, it is enough to check the singularity at $p_1$. Consider the affine cone $C_X \subset \C^4$ of $X(a_1,a_2,a_3,a_4)$ and the corresponding action of $\C^*$ given by $$ \lambda \in \C^*, \ \ \ \ \lambda \cdot (x_1,x_2,x_3,x_4) = (\lambda^{w_1}x_1,\lambda^{w_2}x_2,\lambda^{w_3}x_3,\lambda^{w_4}x_4).$$ Then to study the singularities around $p_1$, we check how the action behaves when we restrict to $(x_1 = 1)$. Notice that, when $x_1 \neq 0$, \[ \frac{\partial}{\partial x_2}(x_1^{a_1}x_2 + x_2^{a_2}x_3 + x_3^{a_3}x_4 + x_4^{a_4}x_1) = x_1^{a_1} + a_2x_2^{a_2-1}x_3 \neq 0,\] so locally, by the Implicit Function Theorem, we can write $x_2$ as a function of $x_3$ and $x_4$, which become local parameters. Then the action of $\C^*$ restricted to $(x_1 = 1)$ is \[\zeta_1\cdot (1,x_2,x_3,x_4) = (1,\zeta_1^{w_2}x_2,\zeta_1^{w_3}x_3,\zeta_1^{w_4}x_4),\] where $\zeta_1$ is a $w_1$-th primitive root of $1$. Therefore, after taking the quotient, the singularity is a cyclic singularity of type $\frac{1}{w_1}(w_3,w_4)$, if $gcd(w_i,w_{i+2}) =1$. If $gcd(w_i,w_{i+2}) = h> 1$, then there are elements which fix the axis $(x_3=0)$, so they are quasi-reflections. We eliminate them by dividing $w_i = ht_i$ and $w_{i+2} = ht_{i+2}$ by $h$, obtaining that the singularity is $\frac{1}{t_i}(t_{i+2},w_{i+3})$.
\end{proof}

Assume that $a_i \geq 2$ for all $i$ \footnote{This is to have the key configuration of curves as shown. By Theorem \ref{gcd}, Koll\'ar surfaces with $a_i=1$ are birationally included in our analysis. Also, check Corollary \ref{a=1} when $w^*=1$.}. We have the following key configuration of curves on $X(a_1,a_2,a_3,a_4)$: \[\begin{array}{lcl}
C_1 & := & (x_1 = x_3 = 0)\\
C_2 & := & (x_2 = x_4 = 0)\\
\Gamma_{1,2} & := & (x_3 = x_4^{a_4} + x_1^{a_1-1}x_2 = 0)\\
\Gamma_{2,3} & := & (x_4 = x_1^{a_1} + x_2^{a_2-1}x_3 = 0)\\
\Gamma_{3,4} & := & (x_1 = x_2^{a_2} + x_3^{a_3-1}x_4 = 0)\\
\Gamma_{4,1} & := & (x_2 = x_3^{a_3} + x_4^{a_4-1}x_1 = 0)
\end{array}\]

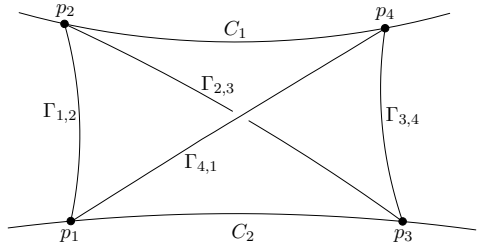
\begin{figure}[h]
\centering
\psscalebox{.7 .7}
{
\begin{tikzpicture}[line cap=round,line join=round,>=triangle 45,x=0.7163616792249728cm,y=0.5910301953818828cm]
\clip(-2.728219383921863,-4.365965439519161) rectangle (11.231209616829457,4.093839218632604);
\draw [shift={(3.999553863722877,20.027933003425957)}] plot[domain=4.384346663084768:5.036891478372022,variable=\t]({1.*17.627944860866872*cos(\t r)+0.*17.627944860866872*sin(\t r)},{0.*17.627944860866872*cos(\t r)+1.*17.627944860866872*sin(\t r)});
\draw [shift={(3.973233397593565,-41.81489795918368)}] plot[domain=1.4066425627285466:1.7251781816514435,variable=\t]({1.*38.715348451922516*cos(\t r)+0.*38.715348451922516*sin(\t r)},{0.*38.715348451922516*cos(\t r)+1.*38.715348451922516*sin(\t r)});
\draw [shift={(-16.079421781026042,-0.5870466510328181)}] plot[domain=-0.1729193261297013:0.22499344547697542,variable=\t]({1.*15.99333835353281*cos(\t r)+0.*15.99333835353281*sin(\t r)},{0.*15.99333835353281*cos(\t r)+1.*15.99333835353281*sin(\t r)});
\draw [shift={(-26.24342260132135,-42.46275538822232)}] plot[domain=0.8462119471603695:0.9452017313997237,variable=\t]({1.*52.23429570366096*cos(\t r)+0.*52.23429570366096*sin(\t r)},{0.*52.23429570366096*cos(\t r)+1.*52.23429570366096*sin(\t r)});
\draw [shift={(-26.24342260132135,-42.46275538822232)}] plot[domain=0.9552017313997237:1.0552017313997237,variable=\t]({1.*52.23429570366096*cos(\t r)+0.*52.23429570366096*sin(\t r)},{0.*52.23429570366096*cos(\t r)+1.*52.23429570366096*sin(\t r)});
\draw [shift={(213.98340313955862,-280.63123190500966)}] plot[domain=2.199366023226847:2.2287675896035837,variable=\t]({1.*350.45540143703397*cos(\t r)+0.*350.45540143703397*sin(\t r)},{0.*350.45540143703397*cos(\t r)+1.*350.45540143703397*sin(\t r)});
\draw [shift={(23.636083459666064,0.893259300972552)}] plot[domain=3.0182660548521616:3.4129124498965346,variable=\t]({1.*15.836683027325408*cos(\t r)+0.*15.836683027325408*sin(\t r)},{0.*15.836683027325408*cos(\t r)+1.*15.836683027325408*sin(\t r)});
\draw[color=black] (4,2.8) node {$C_1$};
\draw[color=black] (4.198903080390686,-3.7) node {$C_2$};
\draw [fill=black] (-0.4891864140265314,2.9810667008859824) circle (2.0pt);
\draw[color=black] (-0.45924868519909745,3.417655897821184) node {$p_2$};
\draw [fill=black] (-0.324597370462421,-3.338842334380736) circle (2.0pt);
\draw[color=black] (-0.339038317054845,-3.8) node {$p_1$};
\draw [fill=black] (7.919681484603742,2.8413964281362816) circle (2.0pt);
\draw[color=black] (7.925424492862513,3.3) node {$p_4$};
\draw [fill=black] (8.378737560278907,-3.3510221946908203) circle (2.0pt);
\draw[color=black] (8.406265965439523,-3.8) node {$p_3$};
\draw[color=black] (-0.6,0.11187077385424224) node {$\Gamma_{1,2}$};
\draw[color=black] (3.5377460555972973,0.9833959429000724) node {$\Gamma_{2,3}$};
\draw[color=black] (3.1,-1.5) node {$\Gamma_{4,1}$};
\draw[color=black] (8.38605559729527,0.0066867017280213825) node {$\Gamma_{3,4}$};
\end{tikzpicture}
}
\caption{Key configuration of curves on a Koll\'ar surface.} \label{curveConfiguration}
\end{figure}

\begin{proposition}
The curves $C_1,C_2$ are smooth and rational. The curve $\Gamma_{i,j}$ is rational, and it may only have a unibranch singularity at $p_j$.
\end{proposition}

\begin{proof}
The curves $C_1, C_2$ are obviously isomorphic to $\P^1$. To prove the assertion about $\Gamma_{i,j}$, it is enough to do it for $\Gamma_{2,3}$. Notice that this curve lives in $(x_4=0)=\P(w_1,w_2,w_3)$, and that it is possibly singular only at $(0:0:1)$. Let us consider the $\Z/w_1 \oplus \Z/w_2 \oplus \Z/w_3$ quotient map $$\P^2 \to \P(w_1,w_2,w_3)$$ given by $(x:y:z) \mapsto (x^{w_1}:y^{w_2}:z^{w_3})$. Then the preimage of $\Gamma_{2,3}$ is $$\Gamma'_{2,3}=(x^{w_1 a_1}+y^{w_2(a_2-1)}z^{w_3}=0),$$ and so $\Gamma_{2,3}$ is rational since all irreducible components (branches at $(0:0:1)$) of $\Gamma'_{2,3}$ are rational curves.

To see that $\Gamma_{2,3}$ is unibranch at $(0:0:1)$, we will show that the (possible) branches of $\Gamma'_{2,3}$ form one orbit under the $\Z/w_1 \oplus \Z/w_2 \oplus \Z/w_3$ action. We take the canonical affine chart at $(0:0:1)$, where $\Gamma'_{2,3}=(x^{w_1 a_1}+y^{w_2(a_2-1)}=0)$. We consider the action of $\Z/w_3$ given by $(x,y) \mapsto (\zeta_3^k x, \zeta_3^k y)$ where $k \in \Z$ and $\zeta_3=e^{\frac{2\pi i}{w_3}}$. Notice that gcd$(w_2,w_1)=1$ and gcd$(w_2,a_1)=1$ by definition, and so we write $a_2-1=rb$ and $w_1 a_1 =ra$ where gcd$(a,b)=1$, to factor in branches $$x^{w_1 a_1}+y^{w_2(a_2-1)}= \prod_{c=0}^{r-1} (y^{w_2 b} - \zeta_{2r}^{2c+1} x^{a})$$ where $\zeta_{2r}=e^{\frac{\pi i}{r}}$. Then we take $ y^{w_2 b} - \zeta_{2r} x^{a}$ and apply $(x,y) \mapsto (\zeta_3^k x, \zeta_3^k y)$ to obtain the branch $ y^{w_2 b} - \zeta_{2r} \zeta_{3}^{k(a-w_2b)}  x^{a},$ but $a-w_2b = \frac{w_3}{r}$, and so it goes to $ y^{w_2 b} - \zeta_{2r}^{2k+1} x^{a}$. Therefore branches form one orbit, and the curve $\Gamma_{2,3}$ is unibranch at $(0:0:1)$.
\end{proof}

\begin{proposition}
Assume that $a_i > w^*$ for some $i$. Then $\Gamma_{i+2,i+3}$ is nonsingular.
\label{paGeneral}
\end{proposition}

\begin{proof}
We take $a_1 > w^*$ to prove that $\Gamma_{3,4}$ is nonsingular. For this we will compute the arithmetic genus of $\Gamma_{3,4}$. Let $\P = \P(w_2,w_3,w_4)$, and consider the exact sequence of sheaves $ 0 \to \O_\P(-a_2w_2) \to \O_\P \to \O_{\Gamma_{3,4}} \to 0$. From it we have that $\chi(\O_{\Gamma_{3,4}}) = \chi(\O_\P) - \chi(\O_\P(-a_2w_2))$. If $\text{gcd}(w_2,w_4) = 1$, then by \cite[Section 1.4]{Dolg82} we have that $\chi(\O_\P) - \chi(\O_\P(-a_2w_2)) = 1 - h^0(\P,\O_\P(a_2w_2-w_2-w_3-w_4))$. Then \[ p_a(\Gamma_{3,4}) = 1 - \chi(\O_{\Gamma_{3,4}}) = h^0(\P,\O_\P(a_2w_2-w_2-w_3-w_4)), \] so we have to compute the number of nonnegative integer solutions of the equation $w_2x + w_3y + w_4z = a_2w_2 - w_2-w_3-w_4$. As $a_2w_2 + w_3 = a_3w_3 + w_4$, then our equation can be written as $$w_2(x+a_2z) + w_3(y+(1-a_3)z) = (a_3-2)w_3 - w_2$$ and its solutions are \begin{equation}\label{solutionPa}
x = -1 -tw_3 - a_2z\quad , \quad y  =  a_3-2+tw_2+(a_3-1)z \quad , \quad z = z
\end{equation}
If $x$, $y$ and $z$ are nonnegative, then $t < 0$, so we will change the sign of $t$ and assume that $t > 0$. Then from Equations \eqref{solutionPa} we obtain that \[ a_2z \leq tw_3 - 1 \] and $(a_3-1)z\geq tw_2-a_3+2$. Hence we have that \begin{equation}\label{systemPa}
\frac{tw_3-1}{a_2} \geq z \geq \frac{tw_2 + 2-a_3}{a_3-1}
\end{equation} Replacing with $w_2=\frac{1}{w^*}(a_3a_4a_1 - a_4a_1 + a_1 -1)$ and $w_3 =\frac{1}{w^*}(a_4a_1a_2-a_1a_2+a_2 -1)$ we obtain \[ ta_4a_1 -t(a_1 -1) - \frac{t+w^*}{a_2} \geq w^*z \geq ta_4a_1 - w^* +\frac{t(a_1-1)+w^*}{a_3-1}.\] Because $a_1 > w^*$ and $t\geq 1$, then $t(a_1-1) \geq w^*$, so $ta_4a_1 - w^* \geq ta_4a_1 -t(a_1 -1)$. We have that both $\frac{t+w^*}{a_2}$ and $\frac{t(a_1-1)+w^*}{a_3-1}$ are positive, therefore the RHS of the system \eqref{systemPa} is greater than the LHS, so the system has no solution. Hence the arithmetic genus of $\Gamma_{3,4}$ is zero and therefore nonsingular.

If $\text{gcd}(w_2,w_4) = h>1$, then $p_a(\Gamma_{3,4}) = h^1(\P,\O_\P(-a_2w_2))$. To compute it, we first have to consider the well formed weighted projective plane $\P' = \P(t_2,w_3,t_4) \simeq \P$, where $t_2 = w_2/h$ and $t_4 = w_4/h$, and following \cite[Remarks 1.3.2]{Dolg82}, we have that $\O_\P(-a_2w_2) \simeq \O_{\P'}(-a_2t_2)$. Then $p_a(\Gamma_{3,4}) = h^0(\P',\O_{\P'}(a_2t_2-t_2-w_3-t_4))$, which is equivalent to the number of nonnegative integer solutions of the equation \[ t_2x + w_3y + t_4z = a_2t_2 - t_2 - w_3 - t_4. \] The general solution of this equation is \begin{equation*}
x = -1 -tw_3 - a_2z\quad , \quad y  =  \frac{a_3-1}{h}-1+t_2t+\frac{a_3-1}{h}z \quad , \quad z = z,
\end{equation*} with $t\in \Z$. Then $t<0$, and changing the sign of $t$ as above, we have that the arithmetic genus is equal to the number of solutions of the system \[ a_1a_4t - t(a_1-1) - \frac{t+w^*}{a_2}\geq w^*z \geq a_1a_4t - w^* + \frac{hw^*+(a_1-1)t}{a_3-1}, \] but again, as $a_i > w^*$, then the RHS is greater than the LHS, so the arithmetic genus is 0.
\end{proof}

\begin{proposition}
The map $\psi$ is defined precisely in $X(a_1,a_2,a_3,a_4) \setminus \{p_1,p_2,p_3,p_4\}$, and it contracts
$$ \psi(C_1\setminus \{p_2,p_4\}) = (0:1:0:-1) \ \ \ \ \ \ \ \psi(C_2\setminus \{p_1,p_3\}) = (1:0:-1:0)$$
$$ \psi(\Gamma_{1,2}\setminus \{p_1,p_2\}) = (-1:0:0:1) \ \ \ \ \ \ \psi(\Gamma_{2,3}\setminus \{p_2,p_3\}) = (1:-1:0:0) $$
$$ \psi(\Gamma_{3,4}\setminus \{p_3,p_4\}) = (0:1:-1:0) \ \ \ \ \ \ \psi(\Gamma_{4,1}\setminus \{p_4,p_1\}) = (0:0:1:-1) $$
\label{imageCurves}
\end{proposition}

\begin{proof}
We have that $\psi|_{\Gamma_{1,2}\setminus \{p_1,p_2\}} = (x_1^{a_1-1}x_2 : 0 : 0 : x_4^{a_4})$, and because $x_1^{a_1-1}x_2 = -x_4^{a_4}$ over $\Gamma_{1,2}$, then $\psi|_{\Gamma_{1,2}\setminus \{p_1,p_2\}} = (-1:0:0:1)$. This gives the result for all curves $\Gamma_{i,i+1}$.

For $C_1$, let $x_4=1$ and $x_2=b\neq 0$. Then the equation of the surface with these restrictions is \[ bx_1^{a_1} + b^{a_2}x_3 + x_3^{a_3} + x_1 = x_1(1+bx_1^{a_1-1}) + x_3(b^{a_2} + x_3^{a_3-1}) = 0. \] The map is $\psi(x_1:b:x_3:1) = (bx_1^{a_1} : b^{a_2}x_3 : x_3^{a_3} : x_1)$. We multiply every coordinate by $(1+bx_1^{a_1-1})$, and use the relation $x_1(1+bx_1^{a_1-1}) = -x_3(b^{a_2} + x_3^{a_3-1})$, to write down $\psi(x_1:b:x_3:1)$ as
$$ (bx_1^{a_1}(1+bx_1^{a_1-1}) : b^{a_2}x_3(1+bx_1^{a_1-1}) : x_3^{a_3}(1+bx_1^{a_1-1}): x_1(1+bx_1^{a_1-1}))=$$
$$ (-x_3bx_1^{a_1-1}(b^{a_2} + x_3^{a_3-1}) : b^{a_2}x_3(1+bx_1^{a_1-1}): x_3^{a_3}(1+bx_1^{a_1-1}) : -x_3(b^{a_2} + x_3^{a_3-1}))$$
$$= (-bx_1^{a_1-1}(b^{a_2} + x_3^{a_3-1}) : b^{a_2}(1+bx_1^{a_1-1}):x_3^{a_3-1}(1+bx_1^{a_1-1}) : -(b^{a_2} + x_3^{a_3-1})).$$
Hence $\psi(0:b:0:1) = (0:b^{a_2}:0:-b^{a_2}) = (0:1:0:-1)$. A similar argument works for $C_2$.
\end{proof}

\begin{remark}
By Theorem \ref{gcd}, we know that any $X(a_1,a_2,a_3,a_4)$ has a birational model $X(a'_1,a'_2,a'_3,a'_4)$ with gcd$(w'_i,w'_{i+2})=1$. \textbf{From now on, we assume that} gcd$(w_1,w_3)=$gcd$(w_2,w_4)=1$.
\end{remark}

Now we want to study the behavior of $\psi$ on a resolution of the singularities in $X(a_1,a_2,a_3,a_4)$. To do so, we need to write this map in terms of local coordinates in the resolution, which are described in the following theorem.

\begin{theorem}[\cite{R}, Theorem 3.2]
Let $X = \mathbb{C}^2/\mathbb{Z}/m$ be a cyclic singularity of type $\frac{1}{m}(a,b)$, and let $\frac{1}{m}(a,b)=\frac{1}{m}(1,q)$ as explained at the beginning of Section \ref{s2}. Let $N$ be the lattice $N = \mathbb{Z}^2 + \mathbb{Z}\cdot \frac{1}{m}(1,q)$, and \[ M = \{ (r,s) : r + qs \equiv 0 \mod m \} \subset \mathbb{Z}^2 \] the dual lattice of invariant monomials under the action $(x,y) \mapsto (\zeta_m x,\zeta_m^q y)$ with $\zeta_m$ an $m$-th primitive root of unity.

Let $\frac{m}{q} = [b_1,\ldots,b_s]$ and let $z_0,z_1,\ldots , z_{s+1}$ vectors in $N$ defined as \[ z_i = \frac{1}{m}(\alpha_i,\beta_i), \] where $\alpha_i$ and $\beta_i$ are as defined at the beginning of Section \ref{s2}. Then for each $i=0,\ldots ,s$, let $u_i,v_i$ be monomials forming the dual basis of $M$ to $z_i,z_{i+1}$; that is, $ u_i = (\beta_i,-\alpha_i); v_i=(-\beta_{i+1},\alpha_{i+1}).$

Then $X$ has a resolution of singularities $Y\to X$ constructed as follows: \[ Y = U_0\cup U_1 \cup \cdots \cup U_s,\] where $U_i \simeq \mathbb{C}^2$ with coordinates $u_i,v_i$.

The glueing $U_i\cup U_{i+1}$ and the morphism $Y\to X$ are both determined by the definition of $u_i,v_i$ and they consist of \[ U_i\setminus (v_i = 0) \stackrel{\simeq}{\longrightarrow} U_{i+1} \setminus (u_{i+1} = 0) \ \text{ given by }\  u_{i+1} = v_i^{-1},\, v_{i+1} = u_iv_i^{b_i}. \] It follows from the definition of the numbers $\alpha_i$ and $\beta_i$ that $u_0 = x^m$ and $v_s = y^m$, and they satisfy the relations \[ x^m = u_i^{\alpha_{i+1}}v_i^{\alpha_i} \ \text{ and }\  y^m = u_i^{\beta_{i+1}}v_i^{\beta_i}.\]

\label{reidResolution}
\end{theorem}

\begin{theorem}
Let $\sigma\colon \tilde{X} \to X(a_1,a_2,a_3,a_4)$ be the minimal resolution, and let $$\hat{X} \stackrel{\varphi}{\longrightarrow} \tilde{X} \stackrel{\sigma}{\longrightarrow}X(a_1,a_2,a_3,a_4)$$ be the minimal log resolution of $X$ together with the key configuration of curves. Then $\psi\circ \sigma \circ \varphi$ is a birational morphism.
\label{morphism}
\end{theorem}

To prove the Theorem \ref{morphism} we have to compute the strict transform of the curves $\Gamma_{i,i+1}$ on $\tilde{X}$. Let $E_{i,j}$ be the components of the exceptional divisor over the point $p_i$, let $\frac{1}{w_i}(w_{i+2},w_{i+3}) = \frac{1}{w_i}(1,q_i)$, and let $\alpha_{i,j}$, $\beta_{i,j}$ and $\gamma_{i,j}$ the integers defined for the continued fraction of $\frac{w_i}{q_i}$. Recall from the proof of Proposition \ref{singularityType} that $x_{i+2}$ and $x_{i+3}$ are toric local coordinates at $p_i$, so we have that $E_{i,0}$ and $E_{i,s_i+1}$ are the strict transform of $(x_{i+3} = 0)$ and $(x_{i+2} = 0)$ at the open set $(x_i \neq 0)$. This means that $E_{1,0} = E_{3,0}$ and $E_{2,0} = E_{4,0}$ and correspond to the strict transform of $C_2$ and $C_1$ respectively. On the other hand, $E_{i,s_i+1}$ corresponds to the strict transform of the curve $\Gamma_{i,i+1}$. Then it remains to compute the strict transform of $\Gamma_{i,i+1}$ around the point $p_{i+1}$, and without loss of generality, we will compute the strict transform $\Gamma_{3,4}$ at the point $p_4$. As all the results will hold locally for $\Gamma_{3,4}$, we can modify the following proofs for every $\Gamma_{i,i+1}$.


\begin{figure}[h]
\psscalebox{.6 .6} 
{
\begin{pspicture}(0,-7.461949)(18.83,7.461949)
\psline[linecolor=black, linewidth=0.038](0.7882351,6.992537)(2.2,5.5807724)(2.2,5.5807724)
\rput[bl](0.7999999,7.1219487){$E_{2,s_2}$}
\psline[linecolor=black, linewidth=0.038](5.4823527,5.5807724)(6.894117,6.992537)(6.894117,6.992537)
\psline[linecolor=black, linewidth=0.038](6.423529,6.992537)(7.8352942,5.5807724)(7.8352942,5.5807724)
\psline[linecolor=black, linewidth=0.038](7.152941,5.9454784)(11.3620615,5.9454784)(11.752941,5.9454784)
\psline[linecolor=black, linewidth=0.038](10.988235,5.5807724)(12.4,6.992537)(11.458823,6.0513606)(11.952941,6.5454783)(12.152941,6.7454786)
\psline[linecolor=black, linewidth=0.038](11.929412,6.992537)(13.341176,5.5807724)
\psline[linecolor=black, linewidth=0.038](13.752941,5.5807724)(15.164705,6.992537)(15.164705,6.992537)
\psdots[linecolor=black, dotsize=0.08](13.3529415,5.945478)
\psdots[linecolor=black, dotsize=0.08](13.552941,5.945478)
\psdots[linecolor=black, dotsize=0.08](13.752941,5.945478)
\psdots[linecolor=black, dotsize=0.08](4.952941,5.945478)
\psdots[linecolor=black, dotsize=0.08](5.152941,5.945478)
\psdots[linecolor=black, dotsize=0.08](5.352941,5.945478)
\psline[linecolor=black, linewidth=0.038](14.729412,6.992537)(16.141176,5.5807724)
\psdots[linecolor=black, dotsize=0.08](16.15294,5.945478)
\psdots[linecolor=black, dotsize=0.08](16.352942,5.945478)
\psdots[linecolor=black, dotsize=0.08](16.55294,5.945478)
\psline[linecolor=black, linewidth=0.038](16.55294,5.5807724)(17.964706,6.992537)(17.964706,6.992537)
\psline[linecolor=black, linewidth=0.038](2.552941,5.5807724)(3.9647062,6.992537)(3.9647062,6.992537)
\psdots[linecolor=black, dotsize=0.08](2.152941,5.945478)
\psdots[linecolor=black, dotsize=0.08](2.352941,5.945478)
\psdots[linecolor=black, dotsize=0.08](2.552941,5.945478)
\psline[linecolor=black, linewidth=0.038](3.5294118,6.992537)(4.9411764,5.5807724)
\rput[bl](2.6,6.4519486){$E_{2,j_2}$}
\rput[bl](4.2,6.321949){$E_{2,j_2-1}$}
\rput[bl](5.66,6.421949){$E_{2,2}$}
\rput[bl](7.1,6.321949){$E_{2,1}$}
\rput[bl](8.042941,5.4854784){$E_{2,0} = C_1' = E_{4,0}$}
\rput[bl](10.9,6.321949){$E_{4,1}$}
\rput[bl](12.34,6.521949){$E_{4,2}$}
\rput[bl](13.52,6.521949){$E_{4,j_4-1}$}
\rput[bl](15.600001,6.2019486){$E_{4,j_4}$}
\rput[bl](17.06,6.961949){$E_{4,s_4}$}
\psline[linecolor=black, linewidth=0.038](0.7882351,2.992537)(2.2,1.5807723)(2.2,1.5807723)
\rput[bl](0.0,2.5219488){$E_{1,s_1}$}
\psline[linecolor=black, linewidth=0.038](5.4823527,1.5807723)(6.894117,2.992537)(6.894117,2.992537)
\psline[linecolor=black, linewidth=0.038](6.423529,2.992537)(7.8352942,1.5807723)(7.8352942,1.5807723)
\psline[linecolor=black, linewidth=0.038](7.152941,1.9454783)(11.3620615,1.9454783)(11.752941,1.9454783)
\psline[linecolor=black, linewidth=0.038](11.929412,2.992537)(13.341176,1.5807723)
\psline[linecolor=black, linewidth=0.038](10.988235,1.5807723)(12.4,2.992537)(11.458823,2.0513606)(11.952941,2.5454783)(12.152941,2.7454782)
\psline[linecolor=black, linewidth=0.038](13.752941,1.5807723)(15.164705,2.992537)(15.164705,2.992537)
\psdots[linecolor=black, dotsize=0.08](13.3529415,2.1454782)
\psdots[linecolor=black, dotsize=0.08](13.552941,2.1454782)
\psdots[linecolor=black, dotsize=0.08](13.752941,2.1454782)
\psdots[linecolor=black, dotsize=0.08](4.952941,2.1454782)
\psdots[linecolor=black, dotsize=0.08](5.152941,2.1454782)
\psdots[linecolor=black, dotsize=0.08](5.352941,2.1454782)
\psline[linecolor=black, linewidth=0.038](14.729412,2.992537)(16.141176,1.5807723)
\psdots[linecolor=black, dotsize=0.08](16.15294,2.1454782)
\psdots[linecolor=black, dotsize=0.08](16.352942,2.1454782)
\psdots[linecolor=black, dotsize=0.08](16.55294,2.1454782)
\psline[linecolor=black, linewidth=0.038](16.55294,1.5807723)(17.964706,2.992537)(17.964706,2.992537)
\psline[linecolor=black, linewidth=0.038](2.552941,1.5807723)(3.9647062,2.992537)(3.9647062,2.992537)
\psdots[linecolor=black, dotsize=0.08](2.152941,2.1454782)
\psdots[linecolor=black, dotsize=0.08](2.352941,2.1454782)
\psdots[linecolor=black, dotsize=0.08](2.552941,2.1454782)
\psline[linecolor=black, linewidth=0.038](3.5294118,2.992537)(4.9411764,1.5807723)
\rput[bl](3.0,1.6219488){$E_{1,j_1}$}
\rput[bl](4.2,2.3219488){$E_{1,j_1-1}$}
\rput[bl](5.89,1.6219488){$E_{1,2}$}
\rput[bl](7.1,2.3219488){$E_{1,1}$}
\rput[bl](8.042941,1.4854783){$E_{1,0} = C_2' = E_{3,0}$}
\rput[bl](10.9,2.3219488){$E_{3,1}$}
\rput[bl](12.25,1.7219487){$E_{3,2}$}
\rput[bl](13.52,2.5219488){$E_{3,j_3-1}$}
\rput[bl](15.100001,1.6019489){$E_{3,j_3}$}
\rput[bl](18.06,2.5619488){$E_{3,s_3}$}
\rput[bl](3.7996078,4.4254785){$E_{2,s_2+1} = \Gamma'_{2,3}$}
\rput[bl](0.2596078,3.7354782){$\Gamma'_{1,2} = E_{1,s_1+1}$}
\rput[bl](16.369608,3.6254783){$E_{3,s_3+1} = \Gamma'_{3,4}$}
\rput[bl](7.299608,3.4454782){$E_{4,s_4+1} = \Gamma'_{4,1}$}
\rput[bl](9.689783,-0.32820603){$\sigma$}
\psbezier[linecolor=black, linewidth=0.038](5.335525,-2.440988)(6.656396,-3.0430932)(13.90076,-3.315474)(15.139927,-2.72770471422942)
\psbezier[linecolor=black, linewidth=0.038](13.802172,-7.171202)(11.313487,-6.3116646)(7.432577,-6.0392833)(4.2766485,-6.799696187816097)
\psbezier[linecolor=black, linewidth=0.038](6.177771,-1.9174432)(6.552248,-2.9990857)(6.0893087,-5.3860884)(5.010733,-7.293956141584525)
\psbezier[linecolor=black, linewidth=0.04](5.895798,-2.2545218)(6.352941,-3.1545217)(12.652941,-6.9545217)(12.552941,-6.854521756835937)(12.452941,-6.754522)(11.952941,-6.354522)(12.010084,-5.9116645)
\rput{22.435516}(-1.1863874,-3.899587){\psframe[linecolor=white, linewidth=0.04, fillstyle=solid, dimen=outer](9.55378,-4.830249)(8.9222,-5.051301)}
\psline[linecolor=black, linewidth=0.038, arrowsize=0.05291667cm 2.0,arrowlength=1.4,arrowinset=0.0]{<-}(9.489783,-0.9703112)(9.489783,0.29284665)
\psbezier[linecolor=black, linewidth=0.04](15.752941,7.145478)(15.552941,6.145478)(14.752941,6.945478)(14.952941,6.74547787890625)(15.152941,6.545478)(14.752941,4.7454777)(17.910835,1.8717941)
\psframe[linecolor=white, linewidth=0.04, fillstyle=solid, dimen=outer](15.695045,4.819163)(15.48452,4.6086364)
\psbezier[linecolor=black, linewidth=0.038](3.489783,2.1875834)(2.7266254,2.1296887)(2.5695207,3.3641381)(3.552941,3.545478243164063)(4.5363617,3.7268183)(16.183655,4.69954)(17.15294,4.9454784)(18.122227,5.1914163)(17.663467,6.4717937)(16.858204,6.6086364)
\rput{10.674068}(1.0506088,-2.3416893){\psframe[linecolor=white, linewidth=0.04, fillstyle=solid, dimen=outer](13.526965,4.704877)(12.590123,4.1996136)}
\psbezier[linecolor=black, linewidth=0.038](1.4897833,6.6086364)(0.5213915,6.8580704)(-0.33519426,5.3192406)(0.647678,5.134951927374589)(1.6305503,4.9506636)(13.3529415,4.9454784)(14.552941,4.1454782)(15.752941,3.3454783)(14.542415,2.1560047)(14.552941,2.1454782)
\psframe[linecolor=white, linewidth=0.04, fillstyle=solid, dimen=outer](3.2792568,5.134952)(3.0687308,4.819163)
\psbezier[linecolor=black, linewidth=0.038](2.9634676,6.39811)(3.352941,5.9454784)(3.373832,5.4152946)(2.9634676,4.5033729800061675)(2.5531032,3.5914514)(1.5529411,2.1454782)(1.0687306,2.1875834)
\rput[bl](5.58452,-3.038732){$p_2$}
\rput[bl](14.531888,-2.491364){$p_4$}
\rput[bl](12.779257,-6.7682056){$p_3$}
\rput[bl](4.968731,-6.5229425){$p_1$}
\rput[bl](9.695046,-2.8545218){$C_1$}
\rput[bl](8.958204,-6.854522){$C_2$}
\rput[bl](5.279257,-4.549258){$\Gamma_{1,2}$}
\rput[bl](13.595046,-4.7597847){$\Gamma_{3,4}$}
\rput[bl](6.6687307,-5.5439954){$\Gamma_{4,1}$}
\rput[bl](10.610836,-5.607153){$\Gamma_{2,3}$}
\psline[linecolor=black, linewidth=0.04, arrowsize=0.05291667cm 2.0,arrowlength=1.4,arrowinset=0.0]{->}(6.5424147,-2.5492587)(7.4897833,-2.6545217)
\rput[bl](6.658204,-2.433469){$x_1=0$}
\psline[linecolor=black, linewidth=0.04, arrowsize=0.05291667cm 2.0,arrowlength=1.4,arrowinset=0.0]{->}(13.910836,-2.8650482)(12.858204,-2.8650482)
\psline[linecolor=black, linewidth=0.04, arrowsize=0.05291667cm 2.0,arrowlength=1.4,arrowinset=0.0]{->}(13.70031,-3.2861006)(12.752941,-3.7071536)
\rput[bl](6.58452,-4.091364){$x_4=0$}
\rput[bl](12.884521,-2.7492588){$x_3=0$}
\psline[linecolor=black, linewidth=0.04, arrowsize=0.05291667cm 2.0,arrowlength=1.4,arrowinset=0.0]{->}(5.4897833,-6.6545215)(6.647678,-6.549258)
\psline[linecolor=black, linewidth=0.04, arrowsize=0.05291667cm 2.0,arrowlength=1.4,arrowinset=0.0]{->}(5.38452,-6.3387322)(5.7003098,-5.3913636)
\psline[linecolor=black, linewidth=0.04, arrowsize=0.05291667cm 2.0,arrowlength=1.4,arrowinset=0.0]{->}(12.647678,-6.7597847)(12.963467,-5.7071533)
\psline[linecolor=black, linewidth=0.04, arrowsize=0.05291667cm 2.0,arrowlength=1.4,arrowinset=0.0]{->}(12.437152,-6.8703113)(11.06873,-6.6597853)
\psline[linecolor=black, linewidth=0.04, arrowsize=0.05291667cm 2.0,arrowlength=1.4,arrowinset=0.0]{->}(6.4371514,-2.9703114)(7.2792573,-3.7071536)
\rput[bl](11.18452,-7.291364){$x_2=0$}
\rput[bl](4.358204,-6.033469){$x_3=0$}
\rput[bl](12.98452,-6.1492586){$x_1=0$}
\rput[bl](12.300309,-3.986101){$x_2=0$}
\rput[bl](5.9003096,-6.9861){$x_4=0$}
\psbezier[linecolor=black, linewidth=0.04](4.5529413,-7.054522)(5.352941,-6.2545223)(14.3529415,-3.054522)(14.552941,-2.6545221210937497)
\psbezier[linecolor=black, linewidth=0.04](12.552941,-7.454522)(12.152941,-6.454522)(13.552941,-3.6545222)(13.952941,-3.0545221210937497)(14.3529415,-2.4545221)(13.152941,-4.2545223)(14.3529415,-3.8545222)
\end{pspicture}
}
\caption{Key configuration of curves on $X(a_1,a_2,a_3,a_4)$ and the curve configuration of the minimal resolution $\tilde{X}$. }
\label{keyConfigurationResolution}
\end{figure}
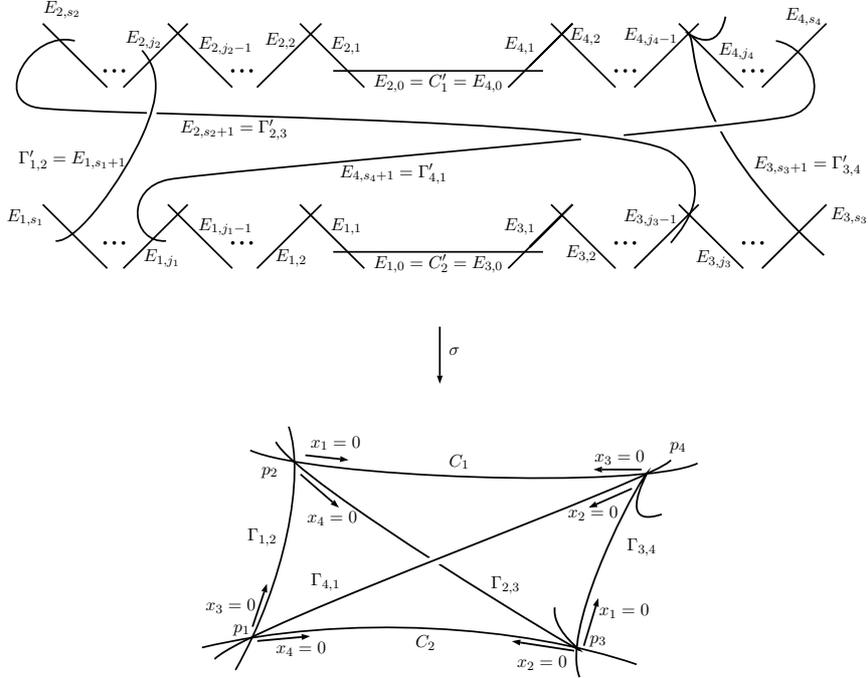


\begin{proposition}
Let $U_{4,j}$ the open sets of the resolution of $\frac{1}{w_4}(1,q_4)$ as defined in Theorem \ref{reidResolution}. Then the local equation of the strict transform of the curve $\Gamma_{3,4}$ restricted to the open set $U_{4,j}$ is \[\Gamma_{34}' = \left\{ \begin{array}{r}
1 + u_j^{((a_3-1)\beta_{4,j+1} - a_2\alpha_{4,j+1})/w_4}v_j^{((a_3-1)\beta_{4,j} - a_2\alpha_{4,j})/w_4} = 0\\
u_j^{(a_2\alpha_{4,j+1} - (a_3-1)\beta_{4,j+1})/w_4}v_j^{(a_2\alpha_{4,j} - (a_3-1)\beta_{4,j})/w_4} + 1 = 0\\
u_j^{(a_2\alpha_{4,j+1} - (a_3-1)\beta_{4,j+1})/w_4} + v_j^{((a_3-1)\beta_{4,j} - a_2\alpha_{4,j})/w_4} = 0
\end{array} \right.,\] if \begin{itemize}
\item[] $a_2\alpha_{4,j} - (a_3 - 1)\beta_{4,j} < a_2\alpha_{4,j+1} - (a_3-1)\beta_{4,j+1} \leq 0$,
\item[] $0\leq a_2\alpha_{4,j} - (a_3 - 1)\beta_{4,j} < a_2\alpha_{4,j+1} - (a_3-1)\beta_{4,j+1}$,
\item[] $a_2\alpha_{4,j} - (a_3 - 1)\beta_{4,j} \leq 0 \leq a_2\alpha_{4,j+1} - (a_3-1)\beta_{4,j+1}$,
\end{itemize} respectively. \label{strictTransformEquation}
\end{proposition}

\begin{proof}
We can assume that $x_4 = 1$ and $x_1 = 0$, so we must study the curve $(x_2^{a_2} + x_3^{a_3-1} =0) \subset (x_4 \not = 0) \subset \P(w_2,w_3,w_4)$. By Theorem \ref{reidResolution}, to find the total transform of $\Gamma_{3,4}$ in $U_i$ we replace $x_2$ and $x_3$ with $u_i^{\alpha_{4,i+1}/w_4}v_i^{\alpha_{4,i}/w_4}$ and $u_i^{\beta_{4,i+1}/w_4}v_i^{\beta_{4,i}/w_4}$ respectively, and so the total transform is \[ (u_i^{\alpha_{4,i+1}/w_4}v_i^{\alpha_{4,i}/w_4})^{a_2} + (u_i^{\beta_{4,i+1}/w_4}v_i^{\beta_{4,i}/w_4})^{a_3-1} = 0. \]

Recall that $\alpha_{4,i} < \alpha_{4,i+1}$ and $\beta_{4,i+1} < \beta_{4,i}$, so \[a_2\alpha_{4,i} - (a_3-1)\beta_{4,i} < a_2\alpha_{4,i+1} - (a_3-1)\beta_{4,i+1}.\] Thus if $a_2\alpha_{4,i} - (a_3 - 1)\beta_{4,i} < a_2\alpha_{4,i+1} - (a_3-1)\beta_{4,i+1} \leq 0$, we factor out $(u_i^{\alpha_{4,i+1}/w_4}v_i^{\alpha_{4,i}/w_4})^{a_2}$. If $0\leq a_2\alpha_{4,i} - (a_3 - 1)\beta_{4,i} < a_2\alpha_{4,i+1} - (a_3-1)\beta_{4,i+1}$, we factor out $(u_i^{\beta_{4,i+1}/w_4}v_i^{\beta_{4,i}/w_4})^{a_3-1}$. If $a_2\alpha_{4,i} - (a_3 - 1)\beta_{4,i} \leq 0 \leq a_2\alpha_{4,i+1} - (a_3-1)\beta_{4,i+1}$, we factor out $u_i^{((a_3-1)\beta_{4,i+1})/w_4}$ and $v_i^{a_2\alpha_{4,i}/w_4}$, obtaining what we wanted to prove.
\end{proof}

Notice that $\Gamma'_{3,4}$ intersects the exceptional divisor if and only if $$ a_2\alpha_{4,i} - (a_3 - 1)\beta_{4,i} \leq 0 \leq a_2\alpha_{4,i+1} - (a_3-1)\beta_{4,i+1}.$$ If $a_2\alpha_{4,i} - (a_3 - 1)\beta_{4,i} < 0 < a_2\alpha_{4,i+1} - (a_3-1)\beta_{4,i+1}$, then the curve intersects two components of the exceptional divisor, and if $a_2\alpha_{4,i} - (a_3 - 1)\beta_{4,i}= 0$ or $a_2\alpha_{4,i+1} - (a_3-1)\beta_{4,i+1}=0$, then it intersects only one component.

\begin{proposition}
Let us say that $\Gamma'_{3,4}$ intersects the exceptional divisor over $p_4$ at the components $E_{4,j}$ and $E_{4,j+1}$ with multiplicity $m_j$ and $m_{j+1}$ respectively (possibly $m_{j+1} = 0$). Then $a_3 - 1 = \alpha_{4,j}m_j + \alpha_{4,j+1}m_{j+1}$ and $a_2 = \beta_{4,j}m_j + \beta_{4,j+1}m_{j+1}$.
\label{intersectionCurves}
\end{proposition}

\begin{proof}

Let $H$ be the restriction to $X(a_1,a_2,a_3,a_4)$ of a generator of the class group of $\P(w_1,w_2,w_3,w_4)$. We have that $$w_1H\cdot w_2H = \frac{w_1w_2(a_3w_3 + w_4)}{w_1w_2w_3w_4} = \frac{1}{w_3} + \frac{a_3}{w_4}.$$ On the other hand, $w_1H\cdot w_2H = \sigma^*(w_1H)\cdot \sigma^*(w_2H)$, where $\sigma^*(w_1H) = \sigma^*(\Gamma_{3,4} + C_1)$, and $\sigma^*(w_2H) = \sigma^*(\Gamma_{4,1} + C_2)$. Because the pull-back of a divisor has intersection zero with any component of the exceptional divisor, and using the pull-back formulas in \eqref{pullbackFormula} we have that $\sigma^*(w_1H)\cdot \sigma^*(w_2H)$  $$ =  (\Gamma_{3,4}' + C_1')\cdot (\sum_{i=0}^{s_3+1} \frac{\beta_{3,i}}{w_3}E_{3,i} + \sum_{i=0}^{s_4+1} \frac{\alpha_{4,i}}{w_4}E_{4,i})$$
 $$ =  \Gamma_{3,4}'\cdot \sum_{i=0}^{s_3+1} \frac{\beta_{3,i}}{w_3}E_{3,i} + C_1'\cdot \sum_{i=0}^{s_4+1} \frac{\alpha_{4,i}}{w_4}E_{4,i} + \Gamma_{3,4}'\cdot \sum_{i=0}^{s_4+1} \frac{\alpha_{4,i}}{w_4}E_{4,i}$$
 $$= \frac{1}{w_3} +  \frac{1}{w_4}+ \sum_{i=0}^{s_4+1} \frac{\alpha_{4,i}}{w_4}\Gamma_{3,4}'\cdot E_{4,i}.$$ Then $a_3 - 1 = \alpha_{4,j}\Gamma_{3,4}'\cdot E_{4,j}+ \alpha_{4,j+1}\Gamma_{3,4}'\cdot E_{4,j+1}= \alpha_{4,j}m_j + \alpha_{4,j+1}m_{j+1}$. To simplify the computation of the second equality, we will restrict to the plane $\P(w_2,w_3,w_4)$, with $L$ a generator of the class group. We can do this because at the point $p_4$ the singularity is the same as the one at the point $(0:0:1) \in \P(w_2,w_3,w_4)$, so locally $\sigma$ does not change.

Then $w_3L\cdot a_2w_2L = \frac{a_2w_2w_3}{w_2w_3w_4} = \frac{a_2}{w_4}$ and also \[\sigma^*(w_3L)\cdot \sigma^*(a_2w_2L) = \Gamma_{3,4}'\cdot \sum_{i=0}^{s_4+1} \frac{\beta_{4,i}}{w_4}E_{4,i},\] where $\sigma^*(w_3L) = \sigma^*(C_1)$ and $\sigma^*(a_2w_2L) = \sigma^*(\Gamma_{3,4})$. Then $a_2 = \beta_{4,j}m_j + \beta_{4,j+1}m_{j+1}$.
\end{proof}

\begin{corollary}
If $\Gamma'_{3,4}$ intersects the exceptional divisor in one component, then it does it transversally.
\label{intersectionTransversally}
\end{corollary}

\begin{proof}
Recall that in the open subset $U_{4,i}$, the exponents of the variables $u_i$ and $v_i$ of the strict transform of $\Gamma_{3,4}$ are $\pm(a_2\alpha_{4,i+1} - (a_3-1)\beta_{4,i+1})/w_4$ and $\pm(a_2\alpha_{4,i} - (a_3-1)\beta_{4,i})/w_4$.

Suppose that $\Gamma'_{3,4}$ intersects $E_j$ with multiplicity $m_j$. Then, using Proposition \ref{intersectionCurves}, we have that for all $i$ \[ \frac{a_2\alpha_{4,i} - (a_3 - 1)\beta_{4,i}}{w_4} = m_j\frac{\beta_{4,j}\alpha_{4,i} - \alpha_{4,j}\beta_{4,i}}{w_4}, \] but the singularity at $p_4$ was unibranch, so it is locally irreducible. Therefore the exponents on the resolution must be relatively prime. Thus $m_j=1$.
\end{proof}

\begin{theorem}
The curve $\Gamma'_{3,4}$ intersects the exceptional divisor in one component if and only if $\psi\circ\sigma$ is defined on the whole exceptional divisor over $p_4$.
\label{definedTransversally}
\end{theorem}

\begin{proof}
The equation of our surface is $x_1^{a_1}x_2 + x_2^{a_2}x_3 + x_3^{a_3}x_4+x_4^{a_4}x_1 = 0$, so locally at $p_4$ our surface is $(x_1^{a_1}x_2 + x_2^{a_2}x_3 + x_3^{a_3}+x_1 = 0)$. Then analytically the power series expansion of $x_1$ in terms of $x_2$ and $x_3$ is \[  x_1 = -x_2^{a_2}x_3 - x_3^{a_3} + (\text{higher order terms in }x_2\text{ and }x_3). \]

Therefore, at the open set $U_i$ \begin{eqnarray*}
\sigma^*(x_1) & = & -(u_i^{\alpha_{4,i+1}/w_4}v_i^{\alpha_{4,i}/w_4})^{a_2}(u_i^{\beta_{4,i+1}/w_4}v_i^{\beta_{4,i}/w_4}) - (u_i^{\beta_{4,i+1}/w_4}v_i^{\beta_{4,i}/w_4})^{a_3}\\ & & + (\text{higher order terms}).
\end{eqnarray*} and so \[
\psi \circ \sigma|_{U_i} = (({}^\ast):u_i^{(a_2\alpha_{4,i+1}+\beta_{4,i+1})/w_4}v_i^{(a_2\alpha_{4,i}+\beta_{4,i+1})/w_4}: u_i^{a_3\beta_{4,i+1}/w_4}v_i^{a_3\beta_{4,i}/w_4}:\]
 \[\ \ \ \ \ \ \ \ \ -u_i^{(a_2\alpha_{4,i+1}+\beta_{4,i+1})/w_4}v_i^{(a_2\alpha_{4,i}+\beta_{4,i+1})/w_4} - u_i^{a_3\beta_{4,i+1}/w_4}v_i^{a_3\beta_{4,i}/w_4} + ({}^*)),
\] where $({}^*)$ are terms in $u_i$ and $v_i$ of degree higher than $(a_2\alpha_{4,i+1}+\beta_{4,i+1}+a_2\alpha_{4,i}+\beta_{4,i+1})/w_4$ and $(a_3\beta_{4,i+1}+a_3\beta_{4,i})/w_4$.

Assume now that $u_i$ and $v_i$ are both nonzero. If $a_2\alpha_{4,i} - (a_3 - 1)\beta_{4,i} < a_2\alpha_{4,i+1} - (a_3-1)\beta_{4,i+1} < 0$, then we can factor out \[ (u_i^{\alpha_{4,i+1}/w_4}v_i^{\alpha_{4,i}/w_4})^{a_2}(u_i^{\beta_{4,i+1}/w_4}v_i^{\beta_{4,i}/w_4}) \] from $\psi\circ \sigma$ to obtain \[ \psi\circ \sigma|_{U_i}  =  (({}^*) : 1 : u_i^{(a_2\alpha_{4,i+1} - (a_3-1)\beta_{4,i+1})/w_4}v_i^{(a_2\alpha_{4,i} - (a_3 - 1)\beta_{4,i})/w_4}:-1 +({}^*) ) \] Then $(\psi\circ \sigma|_{U_i})(u_i,0) = (\psi\circ \sigma|_{U_i})(0,v_i) = (0:1:0:-1)$. Repeating the same procedure for $0< a_2\alpha_{4,i} - (a_3 - 1)\beta_{4,i} < a_2\alpha_{4,i+1} - (a_3-1)\beta_{4,i+1}$, we obtain that restricted to that open set $U_i$, \[(\psi\circ \sigma|_{U_i})(u_i,0) = (\psi\circ \sigma|_{U_i})(0,v_i) = (0:0:1:-1).\]

Now we are left with the case $a_2\alpha_{4,i} - (a_3 - 1)\beta_{4,i} \leq 0 \leq a_2\alpha_{4,i+1} - (a_3-1)\beta_{4,i+1}$. Suppose first that the curve $\Gamma'_{3,4}$ intersect transversally the exceptional divisor, so we know that there is some $j$ such that $a_2\alpha_{4,j} - (a_3 - 1)\beta_{4,j}=0$, and by Corollary \ref{intersectionTransversally}, $a_2\alpha_{4,j+1} - (a_3-1)\beta_{4,j+1}=1$, and $a_2\alpha_{4,j-1} - (a_3-1)\beta_{4,j-1}=-1$. Then in $U_{j-1}$ we can still factor out \[ (u_i^{\alpha_{4,i+1}/w_4}v_i^{\alpha_{4,i}/w_4})^{a_2}(u_i^{\beta_{4,i+1}/w_4}v_i^{\beta_{4,i}/w_4}), \] so assuming that $u_{j-1}$ and $v_{j-1}$ are not zero, the maps looks like \[\psi\circ \sigma|_{U_{j-1}} = (({}^*) : 1 : v_{j-1}:-1 - v_{j-1} + ({}^*) ).\] Therefore $(\psi\circ \sigma|_{U_{j-1}})(u_{j-1},0) = (0:1:0:-1)$ and $(\psi \circ \sigma|_{U_{j-1}})(0,v_{j-1}) = (0:1:v_{j-1}:-1-v_{j-1})$. Doing the same for $U_j$ we find that $(\psi\circ \sigma|_{U_j})(0,v_j) = (0:0:1:-1)$ and $(\psi\circ \sigma|_{U_j})(u_j,0) = (0:u_j:1:-u_j-1)$.
Then we see that $\psi\circ \sigma(\bigcup_{i=0}^{j-1} E_{4,i}) = (0:1:0:-1)$, $\psi\circ \sigma(\bigcup_{i=j+1}^{s_{4}+1} E_{4,i}) = (0:0:1:-1)$. Notice that $v_{j-1}$ and $u_j$ are the coordinates of the charts of $E_j\simeq \P^1$ and that \[(\psi \circ \sigma|_{U_{j-1}})(0,v_{j-1}) = (0:1:v_{j-1}:-1-v_{j-1})\] and \[ (\psi\circ \sigma|_{U_j})(u_j,0) = (0:u_j:1:-u_j-1). \] So $\psi\circ \sigma$ is an isomorphism from $E_j$ onto the line $(y_1 = 0)\subset (y_1+y_2+y_3+y_4 = 0) \subset \P_{y_1,y_2,y_3,y_4}^3$. Therefore $\psi\circ\sigma$ is defined at the exceptional divisor over $p_4$, and it is totally branch over the line $L_1=(y_1 = 0)\subset (y_1+y_2+y_3+y_4 = 0)$.

Now, if $\Gamma'_{3,4}$ does not intersect transversally the exceptional divisor, then $a_2\alpha_{4,i} - (a_3 - 1)\beta_{4,i}\neq 0$ for all $i$, so we will have some $j$ such that \[a_2\alpha_{4,j} - (a_3 - 1)\beta_{4,j} < 0 < a_2\alpha_{4,j+1} - (a_3-1)\beta_{4,j+1},\] and we will not be able to define the map on the open set $U_j$. This because we can factor out $u_j^{a_3\beta_{4,j+1}}v_j^{a_2\alpha_{4,j}+\beta_{4,j}}$ from $\psi\circ\sigma|_{U_j}$, so the map will be \begin{eqnarray*}\psi\circ\sigma|_{U_j} & = & (({}^*): u_j^{(a_2\alpha_{4,j+1} - (a_3-1)\beta_{4,j+1})/w_4}: v_j^{((a_3 - 1)\beta_{4,j}-a_2\alpha_{4,j})/w_4}:\\ & & -u_j^{(a_2\alpha_{4,j+1} - (a_3-1)\beta_{4,j+1})/w_4} - v_j^{((a_3 - 1)\beta_{4,j}-a_2\alpha_{4,j})/w_4} + ({}^*)) \end{eqnarray*} Then if $v_j\neq 0$, $(\psi\circ\sigma|_{U_j})(0,v_j) = (0:0:1:-1)$, and if $u_j\neq 0$, we have $(\psi\circ\sigma|_{U_j})(u_j,0) = (0:1:0:-1)$, and so it is not well-defined when $u_j = v_j = 0$. \end{proof}

\begin{proposition}
Assume that $\Gamma'_{3,4}$ does not intersect transversally the exceptional divisor, so it intersect it at the point $(0,0)$ of some affine open set $U_j$. Let $\varphi_1\colon X_1 \to \tilde{X}$ be the blowup over that point, let $E^{(1)}_{4,j}$ the new component of the exceptional divisor, and let $u_j,v'_{j,1}$ and $u'_{j,1}, v_j$ be the affine coordinates of $U^{(1,1)}_j$ and $U^{(1,2)}_j$, the two affine charts over $U_j$. Then they satisfy the relation $x_2^{w_4} = u_j^{\alpha_{4,j}+\alpha_{4,j+1}}v'^{\alpha_{4,j}}_{j,1}=u'^{\alpha_{4,j+1}}_{j,1}v_j^{\alpha_{4,j}+\alpha_{4,j+1}}$ and $x_3^{w_4} = u_j^{\beta_{4,j}+\beta_{4,j+1}}v'^{\beta_{4,j}}_{j,1}=u'^{\beta_{4,j+1}}_{j,1}v_j^{\beta_{4,j}+\beta_{4,j+1}}$.
\label{relationBlowup}
\end{proposition}

\begin{proof}
This follows from the fact that the resolution was constructed as a toric variety, and the blowup of an affine variety defined by vectors $v_1$ and $v_2$, is the variety associated to the fan generated by the vectors $v_1$, $v_1 + v_2$ and $v_2$.
\end{proof}


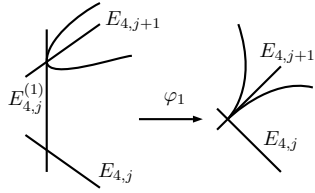
\begin{figure}[h]
\psscalebox{.7 .7} 
{
\begin{pspicture}(0,-1.7670864)(5.738193,1.7670864)
\psline[linecolor=black, linewidth=0.04](5.133333,-1.4508055)(3.9333334,-0.2508056)
\psline[linecolor=black, linewidth=0.04](3.9333334,-0.6508056)(5.133333,0.5491944)
\psbezier[linecolor=black, linewidth=0.04](4.3333335,1.3491944)(4.733333,0.14919442)(4.133333,-0.45080557)(4.133333,-0.45080558776855467)(4.133333,-0.45080557)(4.9333334,0.3491944)(5.733333,0.14919442)
\rput[bl](4.8,-1.0508056){$E_{4,j}$}
\rput[bl](4.733333,0.61586106){$E_{4,j+1}$}
\psline[linecolor=black, linewidth=0.04, arrowsize=0.05291667cm 2.0,arrowlength=1.4,arrowinset=0.0]{<-}(3.6666667,-0.45080557)(2.4666667,-0.45080557)
\rput[bl](2.9333334,-0.18413892){$\varphi_1$}
\psline[linecolor=black, linewidth=0.04](0.33333334,0.3491944)(1.7333333,1.3491944)
\psline[linecolor=black, linewidth=0.04](1.7333333,-1.7508056)(0.33333334,-0.7508056)
\psline[linecolor=black, linewidth=0.04](0.73333335,-1.4508055)(0.73333335,1.2491944)
\psbezier[linecolor=black, linewidth=0.04](1.7333333,1.7491944)(1.3333334,1.5491945)(0.73333335,1.0491945)(0.73333335,0.6491944122314454)(0.73333335,0.24919441)(2.1333334,0.7491944)(2.3333333,0.7491944)
\rput[bl](1.8333334,1.2158611){$E_{4,j+1}$}
\rput[bl](1.7,-1.6508056){$E_{4,j}$}
\rput[bl](0.0,-0.35080558){$E^{(1)}_{4,j}$}
\end{pspicture}
}
\caption{An example of the situation in Proposition \ref{relationBlowup}.}
\end{figure}


Notice that if $a_2\alpha_{4,j} - (a_3 - 1)\beta_{4,j} < 0 < a_2\alpha_{4,j+1} - (a_3-1)\beta_{4,j+1}$, then \[a_2\alpha_{4,j} - (a_3 - 1)\beta_{4,j} < a_2(\alpha_{4,j} + \alpha_{4,j+1}) - (a_3-1)(\beta_{4,j}+\beta_{4,j+1})\] and \[a_2(\alpha_{4,j} + \alpha_{4,j+1}) - (a_3-1)(\beta_{4,j}+\beta_{4,j+1}) < a_2\alpha_{4,j+1} - (a_3-1)\beta_{4,j+1},\] so we can use Proposition \ref{strictTransformEquation} to see that the strict transform of $\Gamma'_{3,4}$ in the blowup intersects at most two components of the exceptional divisor, and that the singularity of the curve is ``better''. Therefore the map $\psi\circ \sigma\circ \varphi_1$ is defined in one of the charts $U^{(1,i)}_j$, and if $a_2(\alpha_{4,j} + \alpha_{4,j+1}) - (a_3-1)(\beta_{4,j}+\beta_{4,j+1}) = 0$, then it is defined in all the exceptional divisor on $X_1$ over $p_4$.

\begin{proof}[Proof of Theorem \ref{morphism}]
If all the curves $\Gamma'_{i,i+1}$ intersect transversally the exceptional divisor on $\tilde{X}$, then the result follows from Theorem \ref{definedTransversally}. If not, then consider the log resolution $\varphi\colon \hat{X}\to X$ of all the curves $\Gamma'_{i,i+1}$. Proposition \ref{relationBlowup} shows that the relations of the new local coordinates are compatible with the previous ones, and as the strict transform of the curves $\Gamma'_{i,i+1}$ intersect transversally the exceptional divisor, we can use the proof of Theorem \ref{definedTransversally} to show that the composition $\psi\circ \sigma \circ \varphi$ is defined over $\hat{X}$.
\end{proof}

\begin{corollary}
The morphisms $\psi\circ \sigma \circ \varphi \colon \hat{X} \to \P^2$ and $Y' \to \P^2$ (defined at the end of Section \ref{s1}) factor through a birational morphism $\hat{X} \to Y'$ which contracts precisely six chains of smooth rational curves in $$(\sigma \circ \varphi) ^*(C_1+C_2+\Gamma_{1,2}+\Gamma_{2,3}+\Gamma_{3,4}+\Gamma_{4,1}),$$ each containing one of the proper transforms of $C_1, C_2,\Gamma_{1,2},\Gamma_{2,3},\Gamma_{3,4},\Gamma_{4,1}$, and each contracting to the six cyclic quotient singularities in $Y'$.
\label{factors}
\end{corollary}

\begin{proof}
First, by Theorem \ref{morphism}, we note that  $\psi\circ \sigma \circ \varphi \colon \hat{X} \to \P^2$ contracts precisely six chains of smooth rational curves in $(\sigma \circ \varphi) ^*(C_1+C_2+\Gamma_{1,2}+\Gamma_{2,3}+\Gamma_{3,4}+\Gamma_{4,1}),$ each containing one of the proper transforms of $C_1$, $C_2$, $\Gamma_{1,2}$, $\Gamma_{2,3}$, $\Gamma_{3,4}$, $\Gamma_{4,1}$. This was done locally when we proved definition of the map in Theorem \ref{definedTransversally} at a certain exceptional component over the $p_i$. Each of these components maps to each of the $4$ lines in $\P^2$. Therefore, the birational map $\hat{X} \dashrightarrow Y'$ is defined over these components except possibly over the six singularities of $Y'$. Because there is a unique minimal resolution for normal two dimensional singularities, the $6$ chains of curves in $\hat{X}$ mapping to the $6$ nodes of the four lines in $\P^2$ must contract to the $6$ singularities of $Y'$.
\end{proof}

\section{Koll\'ar surfaces are Hwang-Keum surfaces} \label{s3}

We now study the case $w^* = 1$. In this section, we allow $\text{gcd}(w_1,w_3)$ and $\text{gcd}(w_2,w_4)$ to be greater than $1$.

In \cite[p. 231]{Ko08}, it is shown that the curves $C_1$ and $C_2$ are extremal rays of the $K_{X(a_1,a_2,a_3,a_4)} + (1-\epsilon)(C_1 + C_2)$ minimal model program if $C_1^2 <0$ and $C_2^2<0$. They are both contractible to quotient singularities. In \cite{HK12} they computed explicitly the type of these singularities.

\begin{theorem}[\cite{HK12}, Theorem 1.1]
The contraction of the curve $C_1$ forms a singularity of type $\frac{1}{s_1}(w_2,w_4)$, with $s_1 = a_4w_4-w_3$, and the contraction of the curve $C_2$ forms a singularity of type $\frac{1}{s_2}(w_1,w_3)$, with $s_2 = a_3w_3 - w_2$. If $w^* = 1$, then their Hirzebruch-Jung continued fractions are \[ [\underbrace{2,\ldots, 2}_{a_4-1} ,a_3,a_1,\underbrace{2,\ldots, 2}_{a_2-1}]\quad and\quad [\underbrace{2,\ldots, 2}_{a_3-1}, a_2,a_4,\underbrace{2,\ldots, 2}_{a_1-1}], \]respectively.
\label{singularitiesContraction}
\end{theorem}

Let $\eta\colon X(a_1,a_2,a_3,a_4) \to X'(a_1,a_2,a_3,a_4)$ be the contraction of $C_1$ and $C_2$. In \cite[\S 4]{HK12} they construct several examples of rational $\Q$-homology projective planes with two cyclic singularities. In certain cases the singularities are the same as for $X'(a_1,a_2,a_3,a_4)$ when $w^* = 1$.

The construction of Hwang-Keum is as follows. Let $L_1,L_2,L_3,L_4$ be four general lines in $\P^2$ and choose four points from the six intersection points, such that every $L_i$ passes through two of them. After blowing up each of these four points twice, we obtain the curve configuration

\bigskip
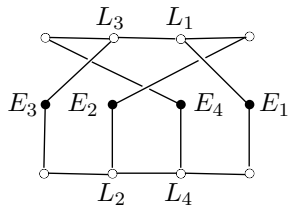
\psscalebox{.9 .9} 
{
\begin{pspicture}(-4.5, -1.5)(5,1.5)
\psdots[dotsize=0.14,fillstyle=solid,dotstyle=o](0.7650486,1.0255765)
\psdots[dotsize=0.14,fillstyle=solid,dotstyle=o](1.7650486,1.0204202)
\psdots[dotsize=0.14,fillstyle=solid,dotstyle=o](2.7450485,1.0055765)
\psdots[dotsize=0.14,fillstyle=solid,dotstyle=o](3.7450485,1.0455763)
\psdots[dotsize=0.14,fillstyle=solid,dotstyle=o](0.7450486,-0.9795798)
\psdots[dotsize=0.14,fillstyle=solid,dotstyle=o](1.7450486,-0.9795797)
\psdots[dotsize=0.14,fillstyle=solid,dotstyle=o](2.7450485,-0.9795797)
\psdots[dotsize=0.14,fillstyle=solid,dotstyle=o](3.7450485,-0.9744236)
\psdots[dotsize=0.14](0.7650486,0.040420208)
\psdots[dotsize=0.14](1.7450486,0.040420208)
\psdots[dotsize=0.14](2.7450485,0.040420208)
\psdots[dotsize=0.14](3.7450485,0.040420208)
\psline[linewidth=0.02cm](1.7409375,0.0296875)(1.7409375,-0.8903125)
\psline[linewidth=0.02cm](2.7409377,0.0496875)(2.7409377,-0.8903125)
\psline[linewidth=0.02cm](3.7409377,0.0496875)(3.7409377,-0.8703125)
\psline[linewidth=0.02cm](0.7609375,0.0496875)(0.7409375,-0.8903125)
\psline[linewidth=0.02cm](2.8009374,0.96968746)(3.7409377,0.0296875)
\psline[linewidth=0.02cm](0.7409375,0.0496875)(1.7209375,0.96968746)
\psline[linewidth=0.02cm](1.7409375,0.0296875)(3.0009375,0.6896875)
\psline[linewidth=0.02cm](3.1009374,0.7296875)(3.6809375,1.0296875)
\psline[linewidth=0.02cm](0.8209375,0.9896875)(1.3809375,0.7096875)
\psline[linewidth=0.02cm](1.4409375,0.6696875)(2.1809375,0.3096875)
\psline[linewidth=0.02cm](2.2809374,0.2696875)(2.7209377,0.0496875)
\usefont{T1}{ptm}{m}{n}
\rput(2.7323437,1.3196875){$L_1$}
\usefont{T1}{ptm}{m}{n}
\rput(1.7123437,1.3196875){$L_3$}
\usefont{T1}{ptm}{m}{n}
\rput(1.7323438,-1.2803125){$L_2$}
\usefont{T1}{ptm}{m}{n}
\rput(2.7123437,-1.2803125){$L_4$}
\usefont{T1}{ptm}{m}{n}
\rput(4.112344,0.040420208){$E_1$}
\usefont{T1}{ptm}{m}{n}
\rput(1.3123437,0.040420208){$E_2$}
\usefont{T1}{ptm}{m}{n}
\rput(0.43234375,0.040420208){$E_3$}
\usefont{T1}{ptm}{m}{n}
\rput(3.1523438,0.040420208){$E_4$}
\psline[linewidth=0.02cm](0.8209375,1.0296875)(1.7009375,1.0096875)
\psline[linewidth=0.02cm](1.8209375,1.0296875)(2.6609375,1.0096875)
\psline[linewidth=0.02cm](2.8009374,1.0096875)(3.7009375,1.0296875)
\psline[linewidth=0.02cm](0.8009375,-0.9703125)(1.6809375,-0.9903125)
\psline[linewidth=0.02cm](1.8009375,-0.9703125)(2.6809375,-0.9703125)
\psline[linewidth=0.02cm](2.8009374,-0.9703125)(3.6809375,-0.9903125)
\end{pspicture}
}

\bigskip

\noindent where $\bullet$ is a $(-1)$-curve and $\circ$ is a $(-2)$-curve. We now blow up $r_i$ times the point $E_i\cap L_i$ to obtain the surface $Z(a_1,a_2,a_3,a_4)$, where $a_i = 2 + r_i$. The curve configuration on $Z(a_1,a_2,a_3,a_4)$ is shown in Figure \ref{curveConfigurationHK}.

\bigskip
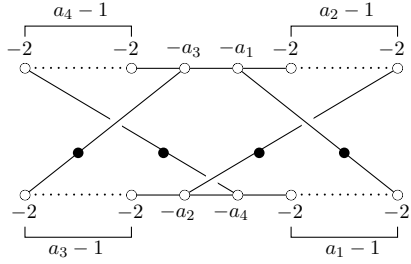
\begin{figure}[htb]
\psscalebox{.7 .7} 
{
\begin{pspicture}(0,-2.3905556)(7.5288887,2.3905556)
\psline[linecolor=black, linewidth=0.02](4.3555555,-1.2127777)(0.35555556,1.1872222)
\psframe[linecolor=white, linewidth=0.04, fillstyle=solid, dimen=outer](3.9777777,-0.635)(3.7555556,-1.0794444)
\psline[linecolor=black, linewidth=0.02](3.3555555,-1.2127777)(7.3555555,1.1872222)
\psframe[linecolor=white, linewidth=0.04, fillstyle=solid, dimen=outer](5.7777777,0.365)(5.5555553,-0.079444446)
\psframe[linecolor=white, linewidth=0.04, fillstyle=solid, dimen=outer](2.1777778,0.365)(1.9555556,-0.079444446)
\psline[linecolor=black, linewidth=0.02](3.3555555,1.1872222)(0.35555556,-1.2127777)
\psline[linecolor=black, linewidth=0.02](4.3555555,1.1872222)(7.3555555,-1.2127777)
\psline[linecolor=black, linewidth=0.04, linestyle=dotted, dotsep=0.10583334cm](0.39033815,1.1959178)(2.3033817,1.1959178)
\psline[linecolor=black, linewidth=0.04, linestyle=dotted, dotsep=0.10583334cm](5.390338,1.1959178)(7.3033814,1.1959178)
\psline[linecolor=black, linewidth=0.02](2.3555555,1.1872222)(5.3555555,1.1872222)
\psdots[linecolor=black, dotstyle=o, dotsize=0.2, fillcolor=white](2.3555555,1.1872222)
\psdots[linecolor=black, dotstyle=o, dotsize=0.2, fillcolor=white](0.35555556,1.1872222)
\psdots[linecolor=black, dotstyle=o, dotsize=0.2, fillcolor=white](3.3555555,1.1872222)
\psdots[linecolor=black, dotstyle=o, dotsize=0.2, fillcolor=white](4.3555555,1.1872222)
\psdots[linecolor=black, dotstyle=o, dotsize=0.2, fillcolor=white](5.3555555,1.1872222)
\psdots[linecolor=black, dotstyle=o, dotsize=0.2, fillcolor=white](7.3555555,1.1872222)
\psline[linecolor=black, linewidth=0.04, linestyle=dotted, dotsep=0.10583334cm](0.39033815,-1.2040821)(2.3033817,-1.2040821)
\psline[linecolor=black, linewidth=0.04, linestyle=dotted, dotsep=0.10583334cm](5.390338,-1.2040821)(7.3033814,-1.2040821)
\psline[linecolor=black, linewidth=0.02](2.3555555,-1.2127777)(5.3555555,-1.2127777)
\psdots[linecolor=black, dotstyle=o, dotsize=0.2, fillcolor=white](2.3555555,-1.2127777)
\psdots[linecolor=black, dotstyle=o, dotsize=0.2, fillcolor=white](0.35555556,-1.2127777)
\psdots[linecolor=black, dotstyle=o, dotsize=0.2, fillcolor=white](3.3555555,-1.2127777)
\psdots[linecolor=black, dotstyle=o, dotsize=0.2, fillcolor=white](4.3555555,-1.2127777)
\psdots[linecolor=black, dotstyle=o, dotsize=0.2, fillcolor=white](5.3555555,-1.2127777)
\psdots[linecolor=black, dotstyle=o, dotsize=0.2, fillcolor=white](7.3555555,-1.2127777)
\psdots[linecolor=black, dotsize=0.2](1.3555555,-0.41277778)
\psdots[linecolor=black, dotsize=0.2](2.9555554,-0.41277778)
\psdots[linecolor=black, dotsize=0.2](4.7555556,-0.41277778)
\psdots[linecolor=black, dotsize=0.2](6.3555555,-0.41277778)
\rput[bl](0.0,1.4094445){$-2$}
\rput[bl](2.0,1.4094445){$-2$}
\rput[bl](5.0,1.4094445){$-2$}
\rput[bl](7.0,1.4094445){$-2$}
\rput[bl](3.0,1.4094445){$-a_3$}
\rput[bl](4.0,1.4094445){$-a_1$}
\rput[bl](0.08888889,-1.6794444){$-2$}
\rput[bl](2.088889,-1.6794444){$-2$}
\rput[bl](5.088889,-1.6794444){$-2$}
\rput[bl](7.088889,-1.6794444){$-2$}
\rput[bl](2.8888888,-1.6794444){$-a_2$}
\rput[bl](3.8888888,-1.6794444){$-a_4$}
\psline[linecolor=black, linewidth=0.02](0.35555556,1.7872223)(0.35555556,1.9872222)(2.3555555,1.9872222)(2.3555555,1.7872223)
\psline[linecolor=black, linewidth=0.02](5.3555555,1.7872223)(5.3555555,1.9872222)(7.3555555,1.9872222)(7.3555555,1.7872223)
\psline[linecolor=black, linewidth=0.02](0.35555556,-1.8127778)(0.35555556,-2.0127778)(2.3555555,-2.0127778)(2.3555555,-1.8127778)
\psline[linecolor=black, linewidth=0.02](5.3555555,-1.8127778)(5.3555555,-2.0127778)(7.3555555,-2.0127778)(7.3555555,-1.8127778)
\rput[bl](0.8888889,2.1205556){$a_4-1$}
\rput[bl](5.888889,2.1205556){$a_2-1$}
\rput[bl](0.8,-2.3905556){$a_3-1$}
\rput[bl](6.0,-2.3905556){$a_1-1$}
\end{pspicture}
}
\caption{Curve configuration over $Z(a_1,a_2,a_3,a_4)$.}
\label{curveConfigurationHK}
\end{figure}

Let $T(a_1,a_2,a_3,a_4)$ be the surface obtained by contracting the two chains of rational curves corresponding to the white vertices. Then this surface is a rational $\Q$-homology projective plane with two cyclic singularities. By Theorem \ref{singularitiesContraction}, it has the same singularities as $X'(a_1,a_2,a_3,a_4)$ when $w^*=1$.

\begin{theorem}
Let $X(a_1,a_2,a_3,a_4)$ be a Koll\'ar surface with $w^*=1$, and assume that $a_i\geq 2$ for all $i$. Then $X'(a_1,a_2,a_3,a_4)$ is the Hwang-Keum surface $T(a_1,a_2,a_3,a_4)$.
\label{kollarEqualHK}
\end{theorem}

To prove Theorem \ref{kollarEqualHK} we will show that we can find the same curve configuration of $Z(a_1,a_2,a_3,a_4)$ (Figure \ref{curveConfigurationHK}) in $\tilde{X'}$ the minimal resolution of $X'(a_1,a_2,a_3,a_4)$.

First of all, we prove that the rational map $\psi$ is defined in the minimal resolution of $X$. For this we will use the following proposition.

\begin{proposition}
Let $X$ be a surface with a cyclic quotient singularity at the point $p$, and let $C\subset X$ be a curve passing through $p$. Then $C$ is nonsingular at $p$ if and only if the strict transform of $C$ intersects transversally at one point only one component of the exceptional divisor of the minimal resolution of $X$.
\label{intersectTransversally}
\end{proposition}

\begin{proof}
See \cite{GL97}.
\end{proof}

By Proposition \ref{paGeneral} we have that the curves $\Gamma_{i,i+1}$ are smooth, so Proposition \ref{intersectTransversally} says that the curves $\Gamma'_{i,i+1}$ intersect transversally the exceptional divisor over $p_{i+1}$. If $\text{gcd}(w_1,w_3) = \text{gcd}(w_2,w_4) = 1$, then we already know that the map $\psi$ is defined on the minimal resolution of $X$. Therefore we only need to check the same assertion when $\text{gcd}(w_1,w_3) > 1$ or $\text{gcd}(w_2,w_4) > 1$.

\begin{proposition}
The map $\psi\circ \sigma \colon \tilde{X}\to \P^2$ is a morphism.
\end{proposition}

\begin{proof}
We study the case over the point $p_4$, with $\text{gcd}(w_2,w_4) = h >1$. The singularity at $p_4$ is $1/w_4(w_2,w_3)$ with toric coordinates $x_2$ and $x_3$. From Proposition \ref{singularityType} we have that $1/w_4(w_2,w_3) \simeq 1/t_4(t_2,w_3)$, with toric coordinates $x'_2$ and $x'_3$, and the relation $x'_2 = x_2$ and $x'_3 = x^h_3$. Then from Theorem \ref{reidResolution} we have $Y=U_1\cup \cdots U_{s_4}$ in the resolution of $p_4$, with $u_i,v_i$ the local coordinates in $U_i$, and the relation $x'^{t_4}_2 = u_i^{\alpha_{4,i}}v_i^{\alpha_{4,{i+1}}}$ and $x'^{t_4}_3 = u_i^{\beta_i}v_i^{\beta_{i+1}}$. The curve $\Gamma_{3,4}\subset \P(t_2,w_3,t_4)$, restricted to the open set $(x_4=1)$, has equation $x'^{a_2}_2 + x'^{(a_3-1)/h}_3 = 0$, and we can use Proposition \ref{strictTransformEquation} to find the equation of the curve in every $U_i$.

Following the proof of Proposition \ref{intersectionCurves}, we have that the intersection number \[ \Gamma_{3,4}'\cdot \sum_{i=0}^{s_4+1} \frac{\beta_{4,i}}{t_4}E_{4,i} = \frac{a_2}{t_4}, \] and using the fact that the curve $\Gamma'_{3,4}$ intersects tranversally one component, we have that there exists $\beta_{4,j} = a_2$ and $\alpha_{4,j} = (a_3-1)/h$. Therefore
\begin{eqnarray*}
a_2\alpha_{4,j-1} - \frac{a_3-1}{h}\beta_{4,j-1} & = & -1\\
a_2\alpha_{4,j}  - \frac{a_3-1}{h}\beta_{4,j} & = & 0\\
a_2\alpha_{4,j+1}  - \frac{a_3-1}{h}\beta_{4,j+1} & = & 1\\
\end{eqnarray*}
Hence considering the composition
\[ \xymatrix@1{ \tilde{X} \ar[r]^-{\sigma} & \frac{1}{t_4}(t_2,w_3) \ar[r]^-{\simeq} & \frac{1}{w_4}(w_2,w_3) \ar@{-->}[r]^-{\psi} & X(a_1,a_2,a_3,a_4) } \] we have the hypothesis of Theorem \ref{definedTransversally}, therefore the map is defined on the whole exceptional divisor.
\end{proof}

\begin{proposition}
The curves $C'_1$ and $C'_2$ in $\tilde{X}$ are $(-1)$-curves. To obtain the chain of curves \[ K_1:= E_{2,s_2}\cup \cdots \cup E_{2,1} \cup C'_1 \cup E_{4,1} \cup \cdots \cup E_{4,s_4}\] and \[ K_2:= E_{1,s_1}\cup \cdots \cup E_{1,1} \cup C'_2 \cup E_{3,1} \cup \cdots \cup E_{3,s_3}\] we blowup $\tilde{X'}$ on the intersection points of the curves with self-intersections $-a_3$ and $-a_1$, and $-a_2$ and $-a_4$ respectively.
\label{C(-1)}
\end{proposition}

\begin{proof}
We have the following commutative diagram

$$ \xymatrix{   \tilde{X} \ar[r]^-{\sigma} \ar[d] & X(a_1,a_2,a_3,a_4)  \ar[d]^{\eta}  \\  \tilde{X'} \ar[r]^-{\sigma'}  & X'(a_1,a_2,a_3,a_4)}$$

Then, to obtain the chain of curves $K_1$ we have to blowup on the exceptional divisor over the singularity $\frac{1}{s_1}(w_2,w_4)$. This is because if no blowup were needed, then $C'_1$ would be some of the curves in the exceptional divisor over the singularity $\frac{1}{s_1}(w_2,w_4)$, so we would have that $w_2\leq a_4-1$ or $w_4\leq a_2 - 1$, which can happen only if one of the $a_i$ is $1$. Recall from Theorem \ref{singularitiesContraction} that the Hirzebruch-Jung continued fraction of the singularity $\frac{1}{s_1}(w_2,w_4)$ is $[2,\ldots,2,a_3,a_1,2,\ldots,2]$. Then we want to show that the blowups needed must be done between the curves with self-intersection $-a_3$ and $-a_1$. For this, we will rule out every other possibility. Suppose first that the blowups are done on the point \begin{center}
 	\begin{tikzpicture}[scale=.2,node distance = 1cm,main node/.style={circle,minimum width=2mm,draw,inner sep=0.1mm}]
    	\node[main node] (a){};
    	\node (b) [right of=a] {$\cdots$};
    	\node[main node] (k) [right of=b] {};
    	\node[main node] (c) [right of=k,label = {[shift={(0,0)}]$-a_3$}] {};
    	\node[main node] (d) [right of=c,label = {[shift={(0,0)}]$-a_1$}] {};
    	\node[main node] (e) [right of=d] {};
    	\node (f) [right of=e] {$\cdots$};
    	\node[main node] (g) [right of=f] {};
    	\draw[-] (a) to node {} (b);
    	\draw[-] (b) to node {} (k);
    	\draw[-] (k) to node {} (c);
    	\draw[-] (c) to node {} (d);
    	\draw[-] (d) to node {$\times$} (e);
    	\draw[-] (e) to node {} (f);
    	\draw[-] (f) to node {} (g);
	\end{tikzpicture}
\end{center} then we would obtain that the continued fraction associated to the singularity at $p_2$ would have an $\beta_i$ such that \[\beta_i \geq |[\underbrace{2,\ldots, 2}_{a_4-1},a_3,a_1+1]|,\] but $|[\underbrace{2,\ldots, 2}_{a_4-1},a_3,a_1+1]| = w_2 + 2 + a_3a_4 - 2a_4 > w_2$, which is a contradiction. If the blowups are done on the point 
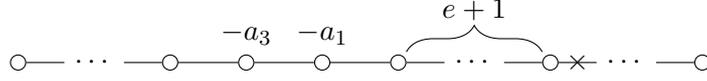
\begin{center}
 	\begin{tikzpicture}[scale=.2,node distance = 1cm,main node/.style={circle,minimum width=2mm,draw,inner sep=0.1mm}]
    	\node[main node] (a){};
    	\node (b) [right of=a] {$\cdots$};
    	\node[main node] (c) [right of=b] {};
    	\node[main node] (d) [right of=c,label = {[shift={(0,0)}]$-a_3$}] {};
    	\node[main node] (e) [right of=d,label = {[shift={(0,0)}]$-a_1$}] {};
    	\node[main node] (f) [right of=e] {};
    	\node (g) [right of=f] {$\cdots$};
    	\node[main node] (h) [right of=g] {};
    	\node (i) [right of=h] {$\cdots$};
    	\node[main node] (j) [right of=i] {};
    	\draw[-] (a) to node {} (b);
    	\draw[-] (b) to node {} (c);
    	\draw[-] (c) to node {} (d);
    	\draw[-] (d) to node {} (e);
    	\draw[-] (e) to node {} (f);
    	\draw[-] (f) to node {} (g);
    	\draw[-] (g) to node {} (h);
    	\draw[-] (h) to node {$\times$} (i);
    	\draw[-] (i) to node {} (j);
    	\draw [decorate,decoration={brace,amplitude=10pt,raise=4pt}]
(f) -- (h) node [black,midway,xshift=0cm,yshift=0.7cm] {$e+1$};
	\end{tikzpicture}
\end{center} with $e\geq 0$, we would have \[ \beta_i \geq |[\underbrace{2,\ldots, 2}_{a_4-1},a_3,a_1,\underbrace{2,\ldots, 2}_{e},3]|, \] but $|[\underbrace{2,\ldots, 2}_{a_4-1},a_3,a_1,\underbrace{2,\ldots, 2}_{e},3]| = (2e+3)w_2 - (2e+1)a_3a_4 - 2a_4 + 1 > w_2$.

Therefore, the blowups to obtain the chain of curves $K_1$ desired have to be done at the point 
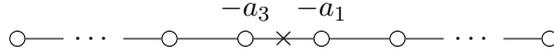
\begin{center}
 	\begin{tikzpicture}[scale=.2,node distance = 1cm,main node/.style={circle,minimum width=2mm,draw,inner sep=0.1mm}]
    	\node[main node] (a){};
    	\node (b) [right of=a] {$\cdots$};
    	\node[main node] (k) [right of=b] {};
    	\node[main node] (c) [right of=k,label = {[shift={(0,0)}]$-a_3$}] {};
    	\node[main node] (d) [right of=c,label = {[shift={(0,0)}]$-a_1$}] {};
    	\node[main node] (e) [right of=d] {};
    	\node (f) [right of=e] {$\cdots$};
    	\node[main node] (g) [right of=f] {};
    	\draw[-] (a) to node {} (b);
    	\draw[-] (b) to node {} (k);
    	\draw[-] (k) to node {} (c);
    	\draw[-] (c) to node {$\times$} (d);
    	\draw[-] (d) to node {} (e);
    	\draw[-] (e) to node {} (f);
    	\draw[-] (f) to node {} (g);
	\end{tikzpicture}
\end{center}

\end{proof}


From the proof of Prop. \ref{C(-1)}, we have that the singularity at $p_i$ of the Koll\'ar surface has Hirzebruch-Jung continued fraction \[ [\ldots, c_i,\underbrace{2,\ldots, 2}_{a_{i+2}-1}] \] with $c_i>2$. The intersection of $\Gamma'_{i-1,i}$ with the exceptional divisor over $p_i$ is $ \beta_{i,j}/w_i = a_{i+2}/w_i$, so the curve $\Gamma'_{i-1,i}$ intersects the exceptional divisor over $p_i$ at the mentioned component with self-intersection $-c_i$. This because $\beta_{i,s_i+1}=0$ and $\beta_{i,s_i} = 1$, and $\beta_{i,k-1} = b_k\beta_{i,k} -\beta_{i,k+1}$. This implies that $\beta_{i,s_i - (a_2-1)} = a_2 = \beta_j$. Therefore we have the curve configuration shown in Figure \ref{curveConfigurationXtilde}.

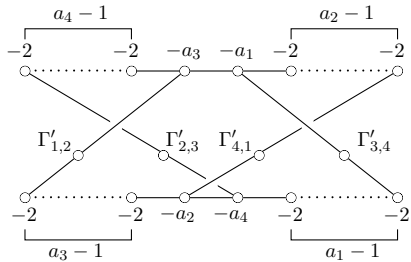
\begin{figure}[ht]
\psscalebox{.7 .7} 
{
\begin{pspicture}(0,-2.3905556)(7.5288887,2.3905556)
\psline[linecolor=black, linewidth=0.02](4.3555555,-1.2127777)(0.35555556,1.1872222)
\psframe[linecolor=white, linewidth=0.04, fillstyle=solid, dimen=outer](3.9777777,-0.635)(3.7555556,-1.0794444)
\psline[linecolor=black, linewidth=0.02](3.3555555,-1.2127777)(7.3555555,1.1872222)
\psframe[linecolor=white, linewidth=0.04, fillstyle=solid, dimen=outer](5.7777777,0.365)(5.5555553,-0.079444446)
\psframe[linecolor=white, linewidth=0.04, fillstyle=solid, dimen=outer](2.1777778,0.365)(1.9555556,-0.079444446)
\psline[linecolor=black, linewidth=0.02](3.3555555,1.1872222)(0.35555556,-1.2127777)
\psline[linecolor=black, linewidth=0.02](4.3555555,1.1872222)(7.3555555,-1.2127777)
\psline[linecolor=black, linewidth=0.04, linestyle=dotted, dotsep=0.10583334cm](0.39033815,1.1959178)(2.3033817,1.1959178)
\psline[linecolor=black, linewidth=0.04, linestyle=dotted, dotsep=0.10583334cm](5.390338,1.1959178)(7.3033814,1.1959178)
\psline[linecolor=black, linewidth=0.02](2.3555555,1.1872222)(5.3555555,1.1872222)
\psdots[linecolor=black, dotstyle=o, dotsize=0.2, fillcolor=white](2.3555555,1.1872222)
\psdots[linecolor=black, dotstyle=o, dotsize=0.2, fillcolor=white](0.35555556,1.1872222)
\psdots[linecolor=black, dotstyle=o, dotsize=0.2, fillcolor=white](3.3555555,1.1872222)
\psdots[linecolor=black, dotstyle=o, dotsize=0.2, fillcolor=white](4.3555555,1.1872222)
\psdots[linecolor=black, dotstyle=o, dotsize=0.2, fillcolor=white](5.3555555,1.1872222)
\psdots[linecolor=black, dotstyle=o, dotsize=0.2, fillcolor=white](7.3555555,1.1872222)
\psline[linecolor=black, linewidth=0.04, linestyle=dotted, dotsep=0.10583334cm](0.39033815,-1.2040821)(2.3033817,-1.2040821)
\psline[linecolor=black, linewidth=0.04, linestyle=dotted, dotsep=0.10583334cm](5.390338,-1.2040821)(7.3033814,-1.2040821)
\psline[linecolor=black, linewidth=0.02](2.3555555,-1.2127777)(5.3555555,-1.2127777)
\psdots[linecolor=black, dotstyle=o, dotsize=0.2, fillcolor=white](2.3555555,-1.2127777)
\psdots[linecolor=black, dotstyle=o, dotsize=0.2, fillcolor=white](0.35555556,-1.2127777)
\psdots[linecolor=black, dotstyle=o, dotsize=0.2, fillcolor=white](3.3555555,-1.2127777)
\psdots[linecolor=black, dotstyle=o, dotsize=0.2, fillcolor=white](4.3555555,-1.2127777)
\psdots[linecolor=black, dotstyle=o, dotsize=0.2, fillcolor=white](5.3555555,-1.2127777)
\psdots[linecolor=black, dotstyle=o, dotsize=0.2, fillcolor=white](7.3555555,-1.2127777)
\psdots[linecolor=black, dotstyle=o, dotsize=0.2, fillcolor=white](1.3555555,-0.41277778)
\psdots[linecolor=black, dotstyle=o, dotsize=0.2, fillcolor=white](2.9555554,-0.41277778)
\psdots[linecolor=black, dotstyle=o, dotsize=0.2, fillcolor=white](4.7555556,-0.41277778)
\psdots[linecolor=black, dotstyle=o, dotsize=0.2, fillcolor=white](6.3555555,-0.41277778)
\rput[bl](0.0,1.4094445){$-2$}
\rput[bl](2.0,1.4094445){$-2$}
\rput[bl](5.0,1.4094445){$-2$}
\rput[bl](7.0,1.4094445){$-2$}
\rput[bl](3.0,1.4094445){$-a_3$}
\rput[bl](4.0,1.4094445){$-a_1$}
\rput[bl](0.08888889,-1.6794444){$-2$}
\rput[bl](2.088889,-1.6794444){$-2$}
\rput[bl](5.088889,-1.6794444){$-2$}
\rput[bl](7.088889,-1.6794444){$-2$}
\rput[bl](2.8888888,-1.6794444){$-a_2$}
\rput[bl](3.8888888,-1.6794444){$-a_4$}
\psline[linecolor=black, linewidth=0.02](0.35555556,1.7872223)(0.35555556,1.9872222)(2.3555555,1.9872222)(2.3555555,1.7872223)
\psline[linecolor=black, linewidth=0.02](5.3555555,1.7872223)(5.3555555,1.9872222)(7.3555555,1.9872222)(7.3555555,1.7872223)
\psline[linecolor=black, linewidth=0.02](0.35555556,-1.8127778)(0.35555556,-2.0127778)(2.3555555,-2.0127778)(2.3555555,-1.8127778)
\psline[linecolor=black, linewidth=0.02](5.3555555,-1.8127778)(5.3555555,-2.0127778)(7.3555555,-2.0127778)(7.3555555,-1.8127778)
\rput[bl](0.8888889,2.1205556){$a_4-1$}
\rput[bl](5.888889,2.1205556){$a_2-1$}
\rput[bl](0.8,-2.3905556){$a_3-1$}
\rput[bl](6.0,-2.3905556){$a_1-1$}
\rput[bl](0.6,-0.39055556){$\Gamma_{1,2}'$}
\rput[bl](3.0,-0.39055556){$\Gamma_{2,3}'$}
\rput[bl](4.0,-0.39055556){$\Gamma_{4,1}'$}
\rput[bl](6.6,-0.39055556){$\Gamma_{3,4}'$}
\end{pspicture}
}
\caption{Curve configuration on $\tilde{X'}$.}
\label{curveConfigurationXtilde}
\end{figure}

\begin{proposition}
The curves $\Gamma'_{i,i+1}$ are $(-1)$-curves.
\label{Gamma(-1)}
\end{proposition}

\begin{proof}
We have a birational morphism $\psi\circ \sigma \colon \tilde{X} \to \P^2$, so it is a composition of blowups, which contracts $(-1)$-curves to reach $\P^2$. We start by contracting the curves from the proof of Proposition \ref{C(-1)} to obtain $\tilde{X'}$ with the curve configuration of Figure \ref{curveConfigurationXtilde}. Recall from Theorem \ref{definedTransversally} that the image of the curves with self-intersection $-a_i$ are the four lines in general position in $\P^2$, so they cannot be contracted. Then, one of the $\Gamma_{i,i+1}'$ is a $(-1)$-curve, say that it is $\Gamma_{1,2}'$. We contract $\Gamma_{1,2}'$ and the chain of $(-2)$-curves connected to it, to obtain the diagram in Figure \ref{XtildeContraction}.

\begin{figure}[ht]
\begin{center}
\psscalebox{.7 .7} 
{
\begin{pspicture}(0,-2.3905556)(7.5288887,2.3905556)
\psline[linecolor=black, linewidth=0.02](1.7555555,-1.2127777)(7.3555555,1.1872222)
\psline[linecolor=black, linewidth=0.02](3.3555555,1.1872222)(1.7555555,-1.2127777)
\psline[linecolor=black, linewidth=0.04, linestyle=dotted, dotsep=0.10583334cm](5.390338,-1.2040821)(7.3033814,-1.2040821)
\psline[linecolor=black, linewidth=0.02](1.7555555,-1.2127777)(5.3555555,-1.2127777)
\psframe[linecolor=white, linewidth=0.04, fillstyle=solid, dimen=outer](3.3777778,-0.235)(3.1555555,-0.67944443)
\psframe[linecolor=white, linewidth=0.04, fillstyle=solid, dimen=outer](5.577778,0.565)(5.1555557,0.12055556)
\psframe[linecolor=white, linewidth=0.04, fillstyle=solid, dimen=outer](2.5777779,0.165)(2.3555555,-0.27944446)
\psline[linecolor=black, linewidth=0.02](4.3555555,1.1872222)(7.3555555,-1.2127777)
\psline[linecolor=black, linewidth=0.04, linestyle=dotted, dotsep=0.10583334cm](0.39033815,1.1959178)(2.3033817,1.1959178)
\psline[linecolor=black, linewidth=0.04, linestyle=dotted, dotsep=0.10583334cm](5.390338,1.1959178)(7.3033814,1.1959178)
\psline[linecolor=black, linewidth=0.02](2.3555555,1.1872222)(5.3555555,1.1872222)
\psdots[linecolor=black, dotstyle=o, dotsize=0.2, fillcolor=white](2.3555555,1.1872222)
\psdots[linecolor=black, dotstyle=o, dotsize=0.2, fillcolor=white](3.3555555,1.1872222)
\psdots[linecolor=black, dotstyle=o, dotsize=0.2, fillcolor=white](4.3555555,1.1872222)
\psdots[linecolor=black, dotstyle=o, dotsize=0.2, fillcolor=white](5.3555555,1.1872222)
\psdots[linecolor=black, dotstyle=o, dotsize=0.2, fillcolor=white](7.3555555,1.1872222)
\psdots[linecolor=black, dotstyle=o, dotsize=0.2, fillcolor=white](7.3555555,-1.2127777)
\psdots[linecolor=black, dotstyle=o, dotsize=0.2, fillcolor=white](5.3555555,-1.2127777)
\psdots[linecolor=black, dotstyle=o, dotsize=0.2, fillcolor=white](1.7555555,-1.2127777)
\psdots[linecolor=black, dotstyle=o, dotsize=0.2, fillcolor=white](4.5555553,-0.012777777)
\psdots[linecolor=black, dotstyle=o, dotsize=0.2, fillcolor=white](6.3555555,-0.41277778)
\rput[bl](0.0,1.4094445){$-2$}
\rput[bl](2.0,1.4094445){$-2$}
\rput[bl](5.0,1.4094445){$-2$}
\rput[bl](7.0,1.4094445){$-2$}
\rput[bl](3.2,1.4094445){$0$}
\rput[bl](4.0,1.4094445){$-a_1$}
\rput[bl](5.088889,-1.6794444){$-2$}
\rput[bl](7.088889,-1.6794444){$-2$}
\rput[bl](1.6888889,-1.6794444){$-a_2+1$}
\rput[bl](3.8888888,-1.6794444){$-a_4$}
\psline[linecolor=black, linewidth=0.02](0.35555556,1.7872223)(0.35555556,1.9872222)(2.3555555,1.9872222)(2.3555555,1.7872223)
\psline[linecolor=black, linewidth=0.02](5.3555555,1.7872223)(5.3555555,1.9872222)(7.3555555,1.9872222)(7.3555555,1.7872223)
\psline[linecolor=black, linewidth=0.02](5.3555555,-1.8127778)(5.3555555,-2.0127778)(7.3555555,-2.0127778)(7.3555555,-1.8127778)
\rput[bl](0.8888889,2.1205556){$a_4-1$}
\rput[bl](5.888889,2.1205556){$a_2-1$}
\rput[bl](6.0,-2.3905556){$a_1-1$}
\rput[bl](3.0,-0.39055556){$\Gamma_{2,3}'$}
\rput[bl](4.8,-0.39055556){$\Gamma_{4,1}'$}
\rput[bl](6.6,-0.39055556){$\Gamma_{3,4}'$}
\psline[linecolor=black, linewidth=0.02](4.3555555,-1.2127777)(0.35555556,1.1872222)
\psdots[linecolor=black, dotstyle=o, dotsize=0.2, fillcolor=white](0.35555556,1.1872222)
\psdots[linecolor=black, dotstyle=o, dotsize=0.2, fillcolor=white](2.9555554,-0.41277778)
\psdots[linecolor=black, dotstyle=o, dotsize=0.2, fillcolor=white](4.3555555,-1.2127777)
\end{pspicture}
}
\end{center}
\caption{Contraction of $\Gamma_{1,2}'$ and the chain of $(-2)$-curves.}
\label{XtildeContraction}
\end{figure}
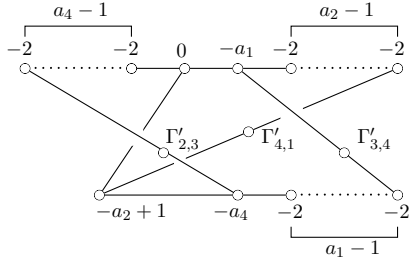

By repeating the procedure, we obtain that all curves $\Gamma_{i,i+1}'$ are $(-1)$-curves.
\end{proof}

\begin{proof}[Proof of Theorem \ref{kollarEqualHK}]
From Proposition \ref{C(-1)} and Proposition \ref{Gamma(-1)}, we conclude that $\tilde{X'}$ and $Z(a_1,a_2,a_3,a_4)$ are obtained from the same sequence of blowups of $\P^2$. Therefore $\tilde{X'} \simeq Z(a_1,a_2,a_3,a_4)$ and so $X'(a_1,a_2,a_3,a_4) \simeq T(a_1,a_2,a_3,a_4)$.
\end{proof}

\begin{remark}
Notice that if $w^* \neq 1$, then the surface $T(a_1,a_2,a_3,a_4)$ does not correspond to a Koll\'ar surface, so Koll\'ar surfaces with $w^*=1$ and $a_i \geq 2$ are strictly contained in Hwang-Keum surfaces.
\end{remark}

Finally, we check what happens when some $a_i = 1$, say $a_1 = 1$.

\begin{corollary}
Let $a_1 = 1$. Then the point $p_4$ is smooth, and the map $\psi$ is defined in the log resolution $\hat{X}$ of the key curves. The curve $\Gamma_{3,4}$ is smooth, and $\psi$ does not contract $C_1$. The surface $\hat{X}$ is obtained by doing blowups from $Z(1,a_2,a_3,a_4)$. The curve $C_1 \subset X(1,a_2,a_3,a_4)$ is contractible if and only if $a_3>a_2$.

\label{a=1}
\end{corollary}

\begin{proof}
If $a_1 = 1$, then $w_2 = a_4(a_3-1)$ and $w_4 = a_3-1$. Then by Proposition \ref{singularityType} we have that the point $p_4$ is smooth, and at the point $p_2$ the singularity is of type $\frac{1}{a_4}(1,a_2a_3a_4-a_3a_4+a_4-1) = \frac{1}{a_4}(1,a_4-1)$. The curve $\Gamma_{1,2}$ intersects transversally the curve $C_1$ at the point $(0:-1:0:1)$, and following the proof of Proposition \ref{imageCurves} we have that $\psi(0:1:0:b) = (b:-1-b:0:1)$, so the curve $\psi$ does not contract $C_1$. The curve $\Gamma_{3,4}$ restricted to the weighted projective plane $(x_1 = 0)$ and to the open set $(x_4 \neq 1)$ is $(x_2^{a_2} + x_3 = 0)\subset \A^2$, so it is smooth and to obtain the log resolution $\hat{X}$ is necessary to do $a_2$ blowups.

Now assume that all the other $a_i\geq 2$. Therefore $C_2$ is contractible, and by contracting it and all the other $(-1)$-curves in $\hat{X}$ we obtain the surface $\hat{X'}$ with the curve configuration shown in Figure \ref{curvesOnXnI}.
\begin{figure}[htb]
\psscalebox{.7 .7} 
{
\begin{pspicture}(0,-2.3905556)(7.45625,2.3905556)
\psline[linecolor=black, linewidth=0.02](4.3555555,-1.2127777)(0.35555556,1.1872222)
\psframe[linecolor=white, linewidth=0.04, fillstyle=solid, dimen=outer](3.9777777,-0.635)(3.7555556,-1.0794444)
\psline[linecolor=black, linewidth=0.02](3.3555555,-1.2127777)(7.3555555,1.1872222)
\psframe[linecolor=white, linewidth=0.04, fillstyle=solid, dimen=outer](5.077778,-0.035)(4.8555555,-0.47944444)
\psframe[linecolor=white, linewidth=0.04, fillstyle=solid, dimen=outer](2.1777778,0.365)(1.9555556,-0.079444446)
\psline[linecolor=black, linewidth=0.02](3.3555555,1.1872222)(0.35555556,-1.2127777)
\psline[linecolor=black, linewidth=0.02](4.3555555,1.1872222)(5.3555555,-1.2127777)
\psline[linecolor=black, linewidth=0.04, linestyle=dotted, dotsep=0.10583334cm](0.39033815,1.1959178)(2.3033817,1.1959178)
\psline[linecolor=black, linewidth=0.04, linestyle=dotted, dotsep=0.10583334cm](5.390338,1.1959178)(7.3033814,1.1959178)
\psline[linecolor=black, linewidth=0.02](2.3555555,1.1872222)(5.3555555,1.1872222)
\psdots[linecolor=black, dotstyle=o, dotsize=0.2, fillcolor=white](2.3555555,1.1872222)
\psdots[linecolor=black, dotstyle=o, dotsize=0.2, fillcolor=white](0.35555556,1.1872222)
\psdots[linecolor=black, dotstyle=o, dotsize=0.2, fillcolor=white](3.3555555,1.1872222)
\psdots[linecolor=black, dotsize=0.2](4.3555555,1.1872222)
\psdots[linecolor=black, dotstyle=o, dotsize=0.2, fillcolor=white](5.3555555,1.1872222)
\psdots[linecolor=black, dotstyle=o, dotsize=0.2, fillcolor=white](7.3555555,1.1872222)
\psline[linecolor=black, linewidth=0.04, linestyle=dotted, dotsep=0.10583334cm](0.39033815,-1.2040821)(2.3033817,-1.2040821)
\psline[linecolor=black, linewidth=0.02](2.3555555,-1.2127777)(5.3555555,-1.2127777)
\psdots[linecolor=black, dotstyle=o, dotsize=0.2, fillcolor=white](2.3555555,-1.2127777)
\psdots[linecolor=black, dotstyle=o, dotsize=0.2, fillcolor=white](0.35555556,-1.2127777)
\psdots[linecolor=black, dotstyle=o, dotsize=0.2, fillcolor=white](3.3555555,-1.2127777)
\psdots[linecolor=black, dotstyle=o, dotsize=0.2, fillcolor=white](4.3555555,-1.2127777)
\psdots[linecolor=black, dotstyle=o, dotsize=0.2, fillcolor=white](5.3555555,-1.2127777)
\psdots[linecolor=black, dotstyle=o, dotsize=0.2, fillcolor=white](1.3555555,-0.41277778)
\psdots[linecolor=black, dotstyle=o, dotsize=0.2, fillcolor=white](2.9555554,-0.41277778)
\psdots[linecolor=black, dotstyle=o, dotsize=0.2, fillcolor=white](5.555556,0.08722222)
\rput[bl](0.0,1.4094445){$-2$}
\rput[bl](2.0,1.4094445){$-2$}
\rput[bl](5.0,1.4094445){$-2$}
\rput[bl](7.0,1.4094445){$-2$}
\rput[bl](3.0,1.4094445){$C'^2_1$}
\rput[bl](4.0,1.4094445){$-1$}
\rput[bl](0.08888889,-1.6794444){$-2$}
\rput[bl](2.088889,-1.6794444){$-2$}
\rput[bl](2.8888888,-1.6794444){$-a_2$}
\rput[bl](3.8888888,-1.6794444){$-a_4$}
\psline[linecolor=black, linewidth=0.02](0.35555556,1.7872223)(0.35555556,1.9872222)(2.3555555,1.9872222)(2.3555555,1.7872223)
\psline[linecolor=black, linewidth=0.02](5.3555555,1.7872223)(5.3555555,1.9872222)(7.3555555,1.9872222)(7.3555555,1.7872223)
\psline[linecolor=black, linewidth=0.02](0.35555556,-1.8127778)(0.35555556,-2.0127778)(2.3555555,-2.0127778)(2.3555555,-1.8127778)
\rput[bl](0.8888889,2.1205556){$a_4-1$}
\rput[bl](5.888889,2.1205556){$a_2-1$}
\rput[bl](0.8,-2.3905556){$a_3-1$}
\rput[bl](0.6,-0.39055556){$\Gamma_{1,2}'$}
\rput[bl](3.0,-0.39055556){$\Gamma_{2,3}'$}
\rput[bl](5.8,-0.19055556){$\Gamma_{4,1}'$}
\rput[bl](5.0,-1.7905556){$\Gamma_{3,4}'$}
\end{pspicture}
}
\caption{Curve configuration on $\hat{X'}$.}
\label{curvesOnXnI}
\end{figure}
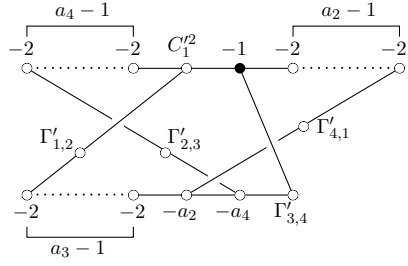
If also $a_2 = 1$, then all the points are smooth but point $p_2$ with a singularity of type $\frac{1}{a_4}(1,a_4-1)$, and we obtain the curve configuration on $\hat{X}$ shown in Figure \ref{curvesOnXnII}.
\begin{figure}[htb]
\psscalebox{.7 .7} 
{
\begin{pspicture}(0,-2.3905556)(5.7,2.3905556)
\psline[linecolor=black, linewidth=0.02](4.3555555,-1.2127777)(0.35555556,1.1872222)
\psframe[linecolor=white, linewidth=0.04, fillstyle=solid, dimen=outer](3.7777777,-0.635)(3.5555556,-1.0794444)
\psline[linecolor=black, linewidth=0.02](3.3555555,-1.2127777)(5.3,1.2205555)
\psframe[linecolor=white, linewidth=0.04, fillstyle=solid, dimen=outer](4.7777777,0.665)(4.5555553,0.22055556)
\psframe[linecolor=white, linewidth=0.04, fillstyle=solid, dimen=outer](2.1777778,0.365)(1.9555556,-0.079444446)
\psline[linecolor=black, linewidth=0.02](3.3555555,1.1872222)(0.35555556,-1.2127777)
\psline[linecolor=black, linewidth=0.02](4.3555555,1.1872222)(5.3555555,-1.2127777)
\psline[linecolor=black, linewidth=0.04, linestyle=dotted, dotsep=0.10583334cm](0.39033815,1.1959178)(2.3033817,1.1959178)
\psline[linecolor=black, linewidth=0.02](2.3555555,1.1872222)(5.3555555,1.1872222)
\psdots[linecolor=black, dotstyle=o, dotsize=0.2, fillcolor=white](2.3555555,1.1872222)
\psdots[linecolor=black, dotstyle=o, dotsize=0.2, fillcolor=white](0.35555556,1.1872222)
\psdots[linecolor=black, dotstyle=o, dotsize=0.2, fillcolor=white](3.3555555,1.1872222)
\psdots[linecolor=black, dotsize=0.2](4.3555555,1.1872222)
\psdots[linecolor=black, dotstyle=o, dotsize=0.2, fillcolor=white](5.3555555,1.1872222)
\psline[linecolor=black, linewidth=0.04, linestyle=dotted, dotsep=0.10583334cm](0.39033815,-1.2040821)(2.3033817,-1.2040821)
\psline[linecolor=black, linewidth=0.02](2.3555555,-1.2127777)(5.3555555,-1.2127777)
\psdots[linecolor=black, dotstyle=o, dotsize=0.2, fillcolor=white](2.3555555,-1.2127777)
\psdots[linecolor=black, dotstyle=o, dotsize=0.2, fillcolor=white](0.35555556,-1.2127777)
\psdots[linecolor=black, dotsize=0.2](3.3555555,-1.2127777)
\psdots[linecolor=black, dotstyle=o, dotsize=0.2, fillcolor=white](4.3555555,-1.2127777)
\psdots[linecolor=black, dotstyle=o, dotsize=0.2, fillcolor=white](5.3555555,-1.2127777)
\psdots[linecolor=black, dotstyle=o, dotsize=0.2, fillcolor=white](1.3555555,-0.41277778)
\psdots[linecolor=black, dotstyle=o, dotsize=0.2, fillcolor=white](2.9555554,-0.41277778)
\rput[bl](0.0,1.4094445){$-2$}
\rput[bl](2.0,1.4094445){$-2$}
\rput[bl](3.0,1.4094445){$C'^2_1$}
\rput[bl](4.0,1.4094445){$-1$}
\rput[bl](0.08888889,-1.6794444){$-2$}
\rput[bl](2.088889,-1.6794444){$-2$}
\rput[bl](3.088889,-1.6794444){$-1$}
\rput[bl](4.0888886,-1.7794445){$C'^2_2$}
\psline[linecolor=black, linewidth=0.02](0.35555556,1.7872223)(0.35555556,1.9872222)(2.3555555,1.9872222)(2.3555555,1.7872223)
\psline[linecolor=black, linewidth=0.02](0.35555556,-1.8127778)(0.35555556,-2.0127778)(2.3555555,-2.0127778)(2.3555555,-1.8127778)
\rput[bl](0.8888889,2.1205556){$a_4-1$}
\rput[bl](0.8,-2.3905556){$a_3-1$}
\rput[bl](0.6,-0.39055556){$\Gamma_{1,2}'$}
\rput[bl](3.0,-0.39055556){$\Gamma_{2,3}'$}
\rput[bl](5.1,1.4094445){$\Gamma_{4,1}'$}
\rput[bl](5.0,-1.7905556){$\Gamma_{3,4}'$}
\end{pspicture}
}
\caption{Curve configuration on $X'_n$ when $a_2=1$.}
\label{curvesOnXnII}
\end{figure}
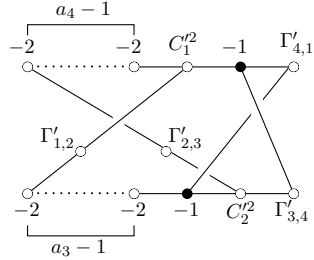

Following the proof of Proposition \ref{Gamma(-1)} we have that the curves $\Gamma_{i,i+1}'$ are $(-1)$-curves, $C'^2_1 = -a_3$ and $C'^2_2 = -a_4$. Therefore $\hat{X'} \simeq Z(1,a_2,a_3,a_4)$, and by contracting the $(-1)$-curve in the top chain along with the $(-2)$-curves to the right, we obtain that $C'^2_1 = -a_3+a_2$. Therefore $C_2$ is contractible if and only if $C'^2_1 <0$, and this is equivalent to $a_3>a_2$. \end{proof}

\section{Koll\'ar surfaces as branch covers of $\P^2$} \label{s4}

We now consider the birational model $Y':= \spec_{\P^{2}} \Big( \bigoplus_{i=0}^{w^*-1} {\L^{(i)}}^{-1} \Big)$ of $X(a_1,a_2,a_3,a_4)$, which was defined at the end of Section \ref{s1} as the $w^*$-th root cover of $(L_1^{\mu_1} L_2^{\mu_2} L_3^{\mu_3} L_4^{\mu_4}=0) \subset \P^2$. We recall that $0<\mu_i<w^*$ are $$\mu_1 \equiv a_2 a_3 a_4, \ \ \ \mu_2 \equiv -a_3 a_4, \ \ \ \mu_3 \equiv a_4, \ \ \ \mu_4 \equiv -1$$ modulo $w^*$, and that by definition gcd$(\mu_i,w^*)=1$. The lines $L_1,L_2,L_3,L_4$ form a plane curve with six nodes. We also recall that $$ \L^{(i)}:= \O_{\P^{2}}(ti) \otimes \O_{\P^{2}}\Bigl( - \sum_{j=1}^4 \Bigl[\frac{\mu_j \ i}{w^*}\Bigr] L_j \Bigr)$$ for $i \in{\{0,1,...,w^*-1 \}}$, where $[x]$ is the integer part of $x$, and $t w^* = \sum_{i=1}^4 \mu_i$. Let $Y$ be the minimal resolution of all singularities in $Y'$.

\begin{theorem}
A $X(a_1,a_2,a_3,a_4)$ is birational to $$X(a'_1,a'_2,a'_3,a'_4) \subset \P(w'_1,w'_2,w'_3,w'_4)$$ with gcd$(w'_1,w'_3)=$gcd$(w'_2,w'_4)=1$, for infinitely many $4$-tuples $(a'_1,a'_2,a'_3,a'_4)$.
\label{gcd}
\end{theorem}

\begin{proof}
By Corollary \ref{bir}, the surface $X(a_1,a_2,a_3,a_4)$ is birational to $Y'$, and so for any  $t_i \in \Z_{\geq 0}$ we have that $X(a_1,a_2,a_3,a_4)$ is birational to $$X(a_1+t_1w^*,a_2+t_2w^*,a_3+t_3w^*,a_4+t_4w^*),$$ as soon as $w^*=$gcd$(W'_1,\ldots,W'_4)$ for the corresponding $W'_i$. This is because, for a fixed $w^*$, the isomorphism type of $Y'$ depends only on the multiplicities $\mu_i$ modulo $w^*$. In this way, we must find $t_i \in \Z_{\geq 0}$ such that $\text{gcd}(w'_1,w'_3)=\text{gcd}(w'_2,w'_4) = 1$, and $w^*=$gcd$(W'_1,\ldots,W'_4)$.

First, choose $t_3$ such that $\text{gcd}(a_3+t_3w^*,6(a_4-1)) = 1$, and let $a'_3 := a_3+ t_3w^*$ and $W'_1 := a_2a'_3a_4-a'_3a_4 + a_4 -1  = w'_1 w^*$. Next take $t_2$ such that $\text{gcd}(w'_1 + t_2a'_3a_4,6(a_4-1)) = 1$, and then define $a'_2:= a_2 + t_2w^*$. Now we will choose $t_1$ such that the final weights $(w''_1,w''_2,w''_3,w''_4)$ satisfy $\text{gcd}(w''_1,w''_3) = \text{gcd}(w''_2,w''_4) = 1$, and $w^*=\text{gcd}(W''_1,\ldots,W''_4)$.

Let $W'_2:= a'_3a_4a_1 - a_4a_1 + a_1 - 1= w'_2w^*$, $W'_3 := a_4a_1a'_2 - a_1a'_2 + a'_2 -1= w'_3w^*$, and $W'_4 := a_1a'_2a'_3 - a'_2a'_3 + a'_3 -1= w'_4w^*$,
and define $$ W''_1 := w''_1w^*, \ \ \ \ \ \ \ W''_2 := w''_2w^* = (w'_2 + t(a'_3a_4-a_4+1)) w^*,$$ $$W''_3  :=  w''_3w^* = (w'_3 + t(a_4a'_2-a'_2)) w^*, \ \ \ \ \ \ W''_4 := w''_4w^* = (w'_4 + ta'_2a'_3) w^*,$$ where $t$ will be found.

Let $w''_1 = \prod q_{1,j}^{\lambda_{1,j}}$ be its prime factorization. Then we have to find a solution $t$ for $w'_4 + ta'_2a'_3 \not\equiv 0 \,(\text{mod } q_{1,j})$, $w'_3 + ta'_2(a_4-1) \not \equiv 0 \,(\text{mod } q_{1,j})$, and $t\not \equiv 0 \,(\text{mod } q_{1,j})$, for all $j$. This $t$ will exist because we have that $\text{gcd}(a_4-1,w''_1) = 1$, and that all $p_{1,j}$ are greater than $3$, by the previous choice of $t_2$ and $t_3$.

By the Chinese Remainder Theorem, we know that the solutions are of the form $t_1 + r\cdot \prod q_{1,j}$, $r\in \Z$. Hence we have that gcd$(w''_1,w''_3) =$gcd$(w''_1,w''_4)=1$, for any choice of $r$. Therefore, considering \[ w''_2 = w'_2 + t_1(a'_3a_4-a_4+1) +r\cdot (a'_3a_4-a_4+1)\cdot \prod q_{1,j} \] and $w''_4 = w'_4 + t_1a'_2a'_3 + r\cdot a'_2a'_3\cdot \prod q_{1,j}$, it is enough to find an $r\in \Z_{\geq 0}$ such that $\text{gcd}(w''_2,w''_4) = 1$. Let
$$ A := w'_2 + t_1(a'_3a_4-a_4+1) \ \ \ B := (a'_3a_4-a_4+1)\cdot \prod q_{1,j}$$ $$C  :=  w'_4 + t_1a'_2a'_3 \ \ \ \ D := a'_2a'_3\cdot \prod q_{1,j}.$$
Notice that $\text{gcd}(A,B)=1$ by the definition of $w'_2$ and the way $t_1$ was obtained. Let $AD-BC = q_{2,1}^{\lambda_{2,1}}q_{2,2}^{\lambda_{2,2}}\cdots q_{2,l}^{\lambda_{2,l}}$; $q_{2,j}$ prime number, and let $r_1$ be a solution of \begin{equation} A+Br \not\equiv 0 \,(\text{mod } q_{2,j}).
\label{eqBir}
\end{equation} Now assume that $\text{gcd}(w''_2,w''_4)=\text{gcd}(A+Br_1,C+Dr_1)>1$. This means that there is a prime $p\not = q_{2,j}$ for all $j$, such that it divides both $A+Br$ and $C+Dr$. Then consider the linear transformation $T \colon (\Z/p\Z)^2 \to (\Z/p\Z)^2$ associated to the matrix $\left( \begin{array}{ll}
A & B\\ C & D
\end{array} \right)$. This matrix maps the vector $(1,r_1)$ to $(0,0)$, so the matrix is singular. But the determinant $AD-BC \not = 0 \,(\text{mod } p)$, which is a contradiction. Therefore $\text{gcd}(A+Br_1,C+Dr_1)=1$. Let $a'_1 := a_1 + (t_1+r_1\cdot \prod p_{1,j})w^*$. This gives us that $X(a'_1,a'_2,a'_3,a_4)\subset \P(w''_1,w''_2,w''_3,w''_4)$ is birational to $X(a_1,a_2,a_3,a_4)$, with $\gcd(w''_1,w''_3) = \gcd(w''_2,w''_4)$ and because $\text{gcd}(w''_1,w''_4)  = 1$, then $w^*=$gcd$(W''_1,\ldots,W''_4)$. Because the equation \eqref{eqBir} has infinite solutions, then we have infinite 4-tuples $(a''_1,a''_2,a''_3,a''_4)$ that satisfy the result.
\end{proof}

\begin{corollary}
Let $Y'$ be a $n$-th root cover of $(L_1^{\mu_1} L_2^{\mu_2} L_3^{\mu_3} L_4^{\mu_4}=0) \subset \P^2$, with $\text{gcd}(\mu_i,n) = 1$ for all $i$. Then $Y'$ is birational to a Koll\'ar surface.
\end{corollary}

\begin{proof}
If we multiply the $\mu_i$ by a unit $\xi$ of $\Z/n\Z$, then the $n$-th root cover does not change. So we take $\xi$ such that $\xi\mu _4 = -1$. In this way, we have to solve the system $a_2a_3a_4 \equiv \xi\mu_1$, $-a_3a_4 \equiv \xi\mu_2$, $a_4 \equiv \xi\mu_3$, and $a_1a_2a_3a_4 \equiv 1$ modulo $n$, which has a solution because $\xi$ and the $\mu_i$ are units in $\Z/n\Z$. Then, with those $a_i$ we can use Theorem \ref{gcd} to find numbers $a'_i$ such that $X(a'_1,a'_2,a'_3,a'_4)$ is a Koll\'ar surface with $w^* = n$, and birational to $Y'$.
\end{proof}

We want to compute the main numerical invariants of $Y$. For that we first define the following numbers.

\begin{definition}
Let $n>1$ be an integer, and let $a,b$ be integers coprime to $n$.

\begin{itemize}
\item[(1)] We define the generalized Dedekind sum \cite[p.94]{HiZa74} as $$s(a,b;n)= \sum_{i=1}^{n-1} \Bigl(\Bigl(\frac{ia}{n}\Bigr)\Bigr)\Bigl(\Bigl(\frac{ib}{n}\Bigr)\Bigr)$$ where $((x))=x-[x]-\frac{1}{2}$ for any rational number $x$.

\item[(2)] Let $0<q<n$ be such that $a+qb\equiv 0$ modulo $n$. We define the HJ length $l=l(a,b;n)$ as the length of the Hirzebruch-Jung continued fraction $$ \frac{n}{q} = [b_1,...,b_{l}].$$
\end{itemize}
\end{definition}

Dedekind sums and Hirzebruch-Jung continued fractions relate as (see e.g. \cite{Ba77}, \cite[Example 3.5]{Urz10}) $$12 s(a,b;n)= \frac{q+q^{-1}}{n} + \sum_{i=1}^{l(a,b;n)} (e_i - 3),$$ where $0<q^{-1}<n$ and $qq^{-1}\equiv1$ modulo $n$.

\begin{proposition}
We have that $\pi_1(Y)=0$, and $$p_g(Y)=  2 s(1,1;w^*) + \sum_{i<j} s(\mu_i,\mu_j;w^*)$$ where $s(1,1;w^*)=\frac{w^*}{12} + \frac{1}{6 w^*} - \frac{1}{4}$.
\label{pg}
\end{proposition}

\begin{proof}
See \cite[Prop.3.2 and Thm.8.5]{Urz10}.
\end{proof}

\begin{remark}
Since the geometric genus $p_g(Y)$ is a nonnegative number, we have $2 s(1,1;w^*) + \sum_{i<j} s(\mu_i,\mu_j;w^*) \geq 0$, which can be rewritten using basic properties of Dedekind sums as $$p_g(Y) = 2s(1,1;w^*) - \sum_{i=1}^4 s(1,a_i;w^*) + s(1,a_1a_4;w^*)+ s(1,a_1 a_2; w^*) \geq 0 .$$ We will tell more on this expression in the next section.
\end{remark}

\begin{proposition}
We have that $e(Y)=w^*+2+\sum_{i<j} l(\mu_i,\mu_j;w^*)$, and $$K_Y^2=w^*+\frac{4}{w^*}+4 + \sum_{i<j} (12s(\mu_i,\mu_j;w^*)-l(\mu_i,\mu_j;w^*))  .$$
\label{ek^2}
\end{proposition}

\begin{proof}
See \cite[Prop. 3.6]{Urz10} and use Noether's formula.
\end{proof}

\begin{corollary}
For $X=X(a_1,a_2,a_3,a_4)$ we have $e(X)=w^*+4$, $\pi_1(X)=0$, and $p_g(X) = 2s(1,1;w^*) - \sum_{i=1}^4 s(1,a_i;w^*) + s(1,a_1a_4;w^*)+ s(1,a_1 a_2; w^*)$.
\label{invariantsFormula}
\end{corollary}

\begin{corollary}
Let gcd$(w_i,w_{i+2}) =1$ for all $i$. Then $$12 \Big( \sum_{i<j} s(\mu_i,\mu_j;w^*) + \sum_i s(w_{i+2},w_{i+3};w_i) \Big) = \ \ \ \ \ \ \ \ \ \ \ \ \ \ \ \ \ \ \ \ \ \ \ \ \ \ \ \ \ \ \ \ \ \ \ \ $$ $$ \ \ \ \ \ \ \ \ \ \ \ \ \ \ \ \ \ \ \ \ \ \ \ \ \ \ \ \ \ \ \ \ \ \ \ \ \ \ \ \ \ \ \ \ \ \ \frac{d(d-\sum_i w_i)^2}{ \prod_i w_i} - \sum_i \frac{2}{w_i} - \frac{{w^*}^2-6 w^* +4}{w^*}.$$
\label{identity}
\end{corollary}

\begin{proof}
Let $X=X(a_1,a_2,a_3,a_4)$. We are going to compute $p_g(X)$ from $X$, and then the equality follows from $p_g(X)=p_g(Y)$. Let $\tilde X \to X$ be the minimal resolution of singularities. As in the proof of Prop. 3.4 in \cite{Urz10}, we have $$K_{\tilde X}^2 - K_X^2= - 12 \sum_i s(w_{i+2},w_{i+3};w_i) - \sum_i l(w_{i+2},w_{i+3};w_i) + \sum_i \frac{2(w_i-1)}{w_i},$$ and $e(\tilde X) - e(X)= \sum_i l(w_{i+2},w_{i+3};w_i)$. Since $K_X^2= \frac{d(d-\sum_i w_i)^2}{ \prod_i w_i}$ and $e(X)=w^*+4$, then the formula $$p_g(X) =  \frac{d(d-\sum_{i=1}^4 w_i)^2}{12 w_1 w_2 w_3 w_4} - \sum_i s(w_{i+2},w_{i+3};w_i) -\frac{1}{6} \sum_i \frac{1}{w_i} +\frac{w^*}{12}$$ is a consequence of the Noether's equality $12 \chi(O_{\tilde X})=K_{\tilde X}^2 + e(\tilde X)$.
\end{proof}

\section{Theorems on geometric genus} \label{s5}

In this section we prove results related to the geometric genus of Koll\'ar surfaces. All our computations will be done in terms of Dedekind sums, and so we state the Reciprocity law.

\begin{theorem}[see e.g. \cite{HiZa74}, p.93]
If $n$ and $k$ are relatively prime, then \begin{equation} s(1,k;n) + s(1,n;k) = \frac{1}{12}\left(\frac{n}{k} + \frac{1}{nk} + \frac{k}{n}\right) - \frac{1}{4}.
\label{reciprocity}
\end{equation}
\end{theorem}

Throughout this section, $w^*$ will be greater than $1$. All equalities involving $\equiv$ will be modulo $w^*$, unless stated otherwise. The symbol $q^{-1}$ will denote the inverse of $q$ modulo $w^*$. To avoid confusions, we will write $\frac{1}{q}$ when it corresponds to a number in $\Q$.

\begin{proposition}
Any $n\geq 0$ is realizable as the geometric genus of a Koll\'ar surface.
\label{anypg}
\end{proposition}

\begin{proof}
We know that $w^* = 1$ implies rational, and so $p_g =0$. Assume that $n>0$, and let $w^* = 3n + 1$ and $a_1 \equiv 3^{-1}$, $a_2\equiv 3$, $a_3\equiv a_4 \equiv w^* - 1$. This gives the $w^*$-th root cover $Y$ with $\mu_1 = 3$, $\mu_2 = \mu_3 = \mu_4 = w^* - 1$. The geometric genus of $Y$ is \begin{eqnarray*}
p_g(Y') & = & 5s(1,1;w^*) - 3s(1,3;w^*)\\
 & = & 5\left( \frac{w^*}{12} + \frac{1}{6w^*} - \frac{1}{4} \right) - 3\left(\frac{w^*}{36} + \frac{1}{4w^*} + \frac{1}{36w^*} -\frac{1}{18} - \frac{1}{4} \right)\\
 & = & n.
\end{eqnarray*} \end{proof}

\subsection{$p_g=0$ surfaces are rational}

\begin{theorem}
Let $X=X(a_1,a_2,a_3,a_4)$ a Koll\'ar surface with $w^* >1$. Then the following are equivalent
\begin{itemize}
\item[(a)] $p_g(X) = 0$.
\item[(b)] $a_i \equiv 1$ or $a_ia_{i+1} \equiv -1$ modulo $w^*$ for some $i$.
\item[(c)] $X$ is rational.
\end{itemize}
\label{pg0}
\end{theorem}

\begin{lemma}
Let $0<a<n$ be relatively prime. Then \begin{enumerate}
\item[(1)] $s(1,1;n) > 2s(1,a;n)$ if $a\not\equiv 1$;
\item[(2)] $s(1,1;n) > 3s(1,a;n)$ if $a\not\equiv 1,2,2^{-1}$;
\item[(3)] $s(1,1;n) > 4s(1,a;n)$ if $a \not\equiv 1,2,2^{-1},3,3^{-1}$.
\end{enumerate}
\label{dedekindBound}
\end{lemma}

\begin{proof}
First of all, using the Reciprocity law we have \begin{eqnarray*}
2s(1,2;n) & = & \frac{n^2-6n+5}{12n} < s(1,1;n)\\
3s(1,3;n) & \leq & \frac{n^2 -7n+10}{12n} < s(1,1;n)\\
4s(1,4;n) & \leq & \frac{n^2-6n+17}{12n} < s(1,1;n)
\end{eqnarray*} with $\text{gcd}(n,2) = 1$, $\text{gcd}(n,3) = 1$ and $\text{gcd}(n,4) = 1$ respectively. Notice that $s(1,1;n) = (n-1)(n-2)/12n$. In \cite[Thm.1]{Girs16}, the author describes how Dedekind sums $s(m,n)$ grow for a fixed $m$, given a positive integer $k$. To do so, Girstmair divides the numbers $1\leq m \leq n-1$ as ordinary and not ordinary, and proves that if $m$ is ordinary, then $s(m,n) \leq \dfrac{n}{12(k+1)} + O(1)$, and if $m$ is not ordinary then there exists $d\in \{1,\ldots ,2k+1\}$ and $c\in \{0,1,\ldots,d\}$, gcd$(c,d)=1$, such that $s(m,n) = \dfrac{n}{12dq} + O(1)$, where $q = md-nc$.

First assume that $k=2$. Notice that $\frac{s(1,1;n)}{2} = \frac{n}{24} + O(1)$, also if $m$ is ordinary, then $s(m,n) \leq \dfrac{n}{36} + O(1)$, and if $m$ is not ordinary and $dq \geq 3$, then $s(m,n) \leq \dfrac{n}{36} + O(1)$. Therefore, we have to find a bound for the three $O(1)$ involved, and find an $N$ such that if $n>N$, then $s(1,1;n)/2 > s(m,n)$ for ordinary numbers and nonordinary numbers with $qd\geq 3$. The procedure to do so is shown by Girstmair in \cite[Thm. 2]{Girs16}, and for the case $k=2$ such $N$ is $132$. The nonordinary numbers with $qd\leq 2$ correspond to $m\equiv 1,2,2^{-1}$, but the first case was ruled out in the proposition, and the inequality for $2$ and $2^{-1}$ was shown at the beginning of the proof. Therefore, we have (1) for $n > 132$, and using a computer we can check that it holds true for every $n$.

For $k=3$ and $k=4$ we obtain similar results, with $N=320$ and $N=630$ respectively. The cases with $qd \leq 3$ and $qd\leq 4$ are the ones ruled out in the proposition, and using a computer we can check that (2) and (3) are true for $n\leq 320$ and $n\leq 630$.
\end{proof}

\begin{corollary}
\mbox{}
\begin{enumerate}
\item[(1)] $2s(1,1;n) - 2s(1,2;n) +s(1,4;n) - s(1,3;n) + s(1,2\cdot 3^{-1};n)$ \\ $- s(1,4\cdot 3^{-1};n) > 0$ for all $n > 5$;
\item[(2)] $2s(1,1;n) - s(1,2;n) - s(1,3;n) - s(1,4;n) + s(1,6;n) - s(1,2\cdot 3^{-1};n)$ \\ $+ s(1,4\cdot 3^{-1};n) > 0$ for all $n>7$;
\item[(3)] $2s(1,1;n) - s(1,2;n) - s(1,3;n) - s(1,5;n) + s(1,6;n) + s(1,2\cdot 5^{-1};n)$ \\ $- s(1,6 \cdot 5^{-1};n) > 0$ for all $n>7$.
\end{enumerate}
\label{coroDedekind}
\end{corollary}

\begin{proof}
Using the inequalities from \ref{dedekindBound} we see that to prove (1) it is enough to prove that $\frac{2}{3}s(1,1;n) + s(1,4;n) + s(1,2\cdot 3^{-1};n) - s(1,4\cdot 3^{-1};n) > 0$. On the other hand, we have that $s(1,4;n) > 0$ if $n\not \in \{7,13,19,25,31\}$, that $s(1,-2\cdot 3^{-1};n) < s(1,1;n)/3$ if $n\not\in \{5,7\}$ and $s(1,4\cdot 3^{-1};n) < s(1,1;n)/3$ if $n\neq 5$. Therefore, if $n$ is not one of those cases, then the inequality holds. We check the remaining cases and find that (1) is false only if $n=5$.
We repeat the same argument and prove that we have to check the cases when $n\in \{7,11,13,19,25,31\}$ for (2), and when $n\in \{7,13,19,31\}$ for (3). Both cases give us that (2) or (3) are false only if $n=7$.
\end{proof}

\begin{proof}[Proof of Theorem \ref{pg0}]
By Corollary \ref{invariantsFormula}, we have that the geometric genus of $X(a_1,a_2,a_3,a_4)$ is \[ p_g(X) = 2s(1,1;w^*) - \sum_{i=1}^4 s(1,a_i;w^*) + s(1,a_1a_4;w^*)+ s(1,a_1 a_2; w^*) \]

\underline{$(c) \Rightarrow (a)$:} This is trivial.
\bigskip

\underline{$(a) \Rightarrow (b)$:} Assume that $a_i \not \equiv 1$ and $a_ia_{i+1} \not \equiv -1$ for all $i$. First, if $a_i\not \equiv 2,2^{-1}$ and $a_ia_{i+1}\not \equiv -2,-2^{-1}$ for all $i$, then by Lemma \ref{dedekindBound},(2) we have that $ p_g >2s(1,1;w^*) - \frac{6}{3}s(1,1;w^*) >0$. Therefore it is enough to rule out the cases when $a_1 \equiv 2$ or $a_1a_2 \equiv -2^{-1}$. First suppose that $a_1 \equiv 2$, so \[ p_g = 2s(1,1;w^*) + s(1,2a_2;w^*) + s(1,2a_4;w^*) - s(1,2;w^*) -\sum_{i=2}^4 s(1,a_i;w^*), \] and we have to check the cases when we cannot use Lemma \ref{dedekindBound},(3).

If $a_3 \equiv 2$ or $a_3 \equiv 2^{-1}$, then $a_1a_2 \equiv -1$ or $a_4 \equiv 1$ respectively, so they satisfy the hypothesis for $p_g = 0$.

If $a_2\equiv 2^{-1}$, $2a_2 \equiv -2$, $2a_4 \equiv -2$, $a_4\equiv 3^{-1}$ or $2a_2\equiv -3$, then one of the terms is equal to $s(1,1;w^*)$ or two of the terms cancel, so by Lemma \ref{dedekindBound},(1) we have that $p_g >0$.

If $a_2\equiv 2$, $2a_2 \equiv -2^{-1}$ or $2a_4\equiv -2^{-1}$, then \[ p_g =2s(1,1;w^*) - 2s(1,2;w^*) +s(1,4;w^*) - s(1,3;w^*) + s(1,2\cdot 3^{-1};w^*)\] \[ - s(1,4\cdot 3^{-1};w^*) \ \ \ \ \ \ \ \ \ \ \ \ \ \ \ \ \ \ \ \ \ \ \ \ \ \ \ \ \ \ \ \ \ \ \ \ \ \ \ \ \ \ \ \ \ \ \ \ \ \ \ \ \] and by Corollary \ref{coroDedekind},(1) $p_g>0$ when $w^* >5$. If $w^* = 5$, then it satisfies the conditions for $p_g=0$.

If $a_2 \equiv 3$ or $2a_4 \equiv -3^{-1}$, then \[ p_g = 2s(1,1;w^*) - s(1,2;w^*) - s(1,3;w^*) - s(1,4;w^*) + s(1,6;w^*)\] \[ - s(1,2\cdot 3^{-1};w^*) + s(1,4\cdot 3^{-1};w^*) \ \ \ \ \ \ \ \ \ \ \ \ \ \ \ \ \ \ \] and by Corollary \ref{coroDedekind},(2) $p_g>0$ when $w^* >7$. If $w^* = 7$, then it satisfies the conditions for $p_g=0$.

If $a_4 \equiv 3$ or $2a_2 \equiv -3^{-1}$, then \[ p_g = 2s(1,1;w^*) - s(1,2;w^*) - s(1,3;w^*) - s(1,5;w^*) + s(1,6;w^*) \] \[+ s(1,2\cdot 5^{-1};w^*) - s(1,6 \cdot 5^{-1};w^*) \ \ \ \ \ \ \ \ \ \ \ \ \ \ \ \ \ \] and by Corollary \ref{coroDedekind},(3) $p_g>0$ when $w^* >7$. If $w^* = 7$, then it satisfies the conditions for $p_g=0$.

These cover all the cases for $a_1 \equiv 2$. Now assume that $a_1a_2 \equiv -2^{-1}$, so \[ p_g = 2s(1,1;w^*) - s(1,2;w^*) + s(1,a_1a_4;w^*) + s(1,2a_2;w^*) - \sum_{i=2}^4 s(1,a_i;w^*),\] and we proceed as the previous case.

If $a_1a_4 \equiv -2$ or $a_1a_4 \equiv -2^{-1}$, then $a_1\equiv 1$ or $a_4 \equiv 1$ respectively, so they satisfy the hypothesis for $p_g = 0$.

If $a_2\equiv 3^{-1}$ or $a_3 \equiv 3$, then two of the terms in the sum cancel, so by Lemma \ref{dedekindBound},(1) we have that $p_g >0$.

If $a_4 \equiv 3^{-1}$ or $2a_2 \equiv -3^{-1}$, then \[ p_g = 2s(1,1;w^*) - s(1,2;w^*) - s(1,3;w^*) - s(1,4;w^*) + s(1,6;w^*)\] \[ - s(1,2\cdot 3^{-1};w^*) + s(1,4\cdot 3^{-1};w^*) \ \ \ \ \ \ \ \ \ \ \ \ \ \ \ \ \ \ \] and by Corollary \ref{coroDedekind},(2) $p_g>0$ when $w^* >7$. If $w^* = 7$, then it satisfies the conditions for $p_g=0$.

If $a_2 \equiv 3$ or $a_3 \equiv 3^{-1}$, then \[ p_g = 2s(1,1;w^*) - s(1,2;w^*) - s(1,3;w^*) - s(1,5;w^*) + s(1,6;w^*) \] \[+ s(1,2\cdot 5^{-1};w^*) - s(1,6 \cdot 5^{-1};w^*) \ \ \ \ \ \ \ \ \ \ \ \ \ \ \ \ \ \] and by Corollary \ref{coroDedekind},(3) $p_g>0$ when $w^* >7$. If $w^* = 7$, then it satisfies the conditions for $p_g=0$.

These cover all the cases for $a_1a_2 \equiv -2^{-1}$.
\bigskip

\underline{$(b) \Rightarrow (c)$:} Notice that $b)$ implies the existence of $\mu_i$ and $\mu_j$ such that $\mu_i+\mu_j \equiv 0 ($mod $w^*)$. Consider the trivial pencil of lines trough $L_i \cap L_j$. Since $\mu_i+\mu_j \equiv 0 ($mod $w^*)$, this pencil defines a pencil of smooth rational curves in $Y$ via pull-back. Therefore $Y$ is rational, and so is $X$.
\end{proof}

\subsection{$p_g=1$ surfaces are K3} In Table 1, we show the total transform of the key configuration of curves after successively blowing down several $(-1)$-curves from the minimal resolution of the indicated surfaces $X(a_1,a_2,a_3,a_4)$.

\begin{theorem}
Let $X=X(a_1,a_2,a_3,a_4)$ a Koll\'ar surface with $w^* >1$. Then the following are equivalent
\begin{itemize}
\item[(a)] $p_g(X) = 1$.
\item[(b)] $X$ is birational to one of the $8$ surfaces in Table 1.
\item[(c)] $X$ is birational to a K3 surface.
\end{itemize}
\label{pg1}
\end{theorem}

\begin{proof}
\underline{$(c) \Rightarrow (a)$:} It is trivial.
\bigskip

\underline{$(a) \Rightarrow (b)$:} First we prove the following lemma.

\begin{lemma}
Let $m$ be a positive integer. Then there is a positive integer $N$ such that if $w^*>N$ and $p_g\neq 0$, then $p_g > m$.
\label{pgBound}
\end{lemma}

\begin{proof}
If all $a_i$, and $-a_1a_2$ and $-a_1a_4$ are not equivalent to $2,2^{-1},3,3^{-1}$, then by Lemma \ref{dedekindBound},(3) we have that \[ p_g > 2s(1,1;w^*) - \frac{6}{4}s(1,1;w^*) = \frac{1}{2}s(1,1;w^*). \] Also we note that if we fix two of these values, say for example $a_1\equiv 2$ and $a_1a_2 \equiv -3$, then the rest of the $a_i$ are completely determined, and they are equivalent to $2,2^{-1},3,3^{-1}$ only for finitely many $w^*$. Therefore if we set that two of the $a_i$, $-a_1a_2$ or $-a_1a_4$ to be equivalent to $3$ or $3^{-1}$, then for $w^*>>0$ we have that \[ p_g > 2s(1,1;w^*) - \frac{2}{3}s(1,1;w^*) - s(1,1;w^*) = \frac{1}{3}s(1,1;w^*). \] If one of the values is $2$ or $2^{-1}$ and the other is $3$ or $3^{-1}$, then for $w^*>>0$ \[ p_g > 2s(1,1;w^*) - \frac{1}{2}s(1,1;w^*) - \frac{1}{3}s(1,1;w^*) - s(1,1;w^*) = \frac{1}{6}s(1,1;w^*). \] Both of these cases happen when $w^* > 28$, hence we have to check the case when two of the values are $2$ or $2^{-1}$. This was done in the proof of \ref{pg0}, and the only relevant case is when $p_g$ is $2s(1,1;w^*) - 2s(1,2;w^*) +s(1,4;w^*) - s(1,3;w^*) + s(1,2\cdot 3^{-1};w^*) - s(1,4\cdot 3^{-1};w^*)$. For $w^* >> 0$ we have that \[ p_g > 2s(1,1;w^*) - s(1,1;w^*) - \frac{1}{3}s(1.1;w^*) - \frac{1}{2}s(1,1;w^*) + s(1,4;w^*), \] and because $s(1,4;w^*) \geq 0$ for $w^* \geq 15$, we have that $p_g > s(1,1;w^*)/6$.

Therefore $N$ is the first integer such that $s(1,1;N) > 6m$.
\end{proof}

To prove this implication, we first use Lemma \ref{pgBound} for $m=1$, which gives us that $N = 75$. We check using a computer all the possible $w^*$-th root covers for $w^* \leq 75$, and find that there are 8 cases with $p_g = 1$, which are represented by a Koll\'ar surface in Table 1.
\bigskip

\underline{$(b) \Rightarrow (c)$:} We prove this implication by means of the following simple lemma.

\begin{lemma}
Let $S$ be a smooth projective surface with $p_g=1$ and $q=0$. Assume it has an effective connected divisor $F$ with $F^2=0$ and $p_a(F)=1$, and a $(-2)$-curve $C$ such that $F \cdot C=1$. Then $S$ is birational to a K3 surface, and $F$ is a fiber of an elliptic fibration $S \to \P^1$, where $C$ is a section.
\end{lemma}

\begin{proof}
Notice that $F$ has the type of a non-multiple fiber of an elliptic fibration. We want to get such a fibration on $S$. By the Riemann-Roch inequality and $F \cdot (F-K_S)=0$, we have $h^0(F)+h^2(F) \geq \chi(\O_S)=2$. Since in addition $h^2(F)=h^0(K_S-F)$ and $C \cdot (K_S-F)=-1$, we have $h^2(F)=0$. Therefore, there is a fibration $S \to \P^1$ with general fiber of genus $1$ and $F$ is a fiber. Let $S'$ be the relative minimal model of this fibration. By the canonical class formula, $K_S \sim (-2+\chi(\O_S))F + \sum_i (m_i-1)G_i + E$ where $G_i$ are the multiple fibers, and $E$ is the exceptional divisor from $S \to S'$. But there is a section $C$, and so $G_i=0$ for all $i$. Then $S'$ has trivial canonical class, and so it is a K3 surface.
\end{proof}

\begin{center}
\captionof{table}{List for $p_g=1$}
\begin{longtable}{|c|c|c|}
\hline
$X(a_1,a_2,a_3,a_4)$ & $w^*$ & Total transform of key configuration \\
\hline
$X(7,7,15,15)$ & $4$ & \psscalebox{.6 .6} 
{
\begin{pspicture}(0,-1.495)(8.54,1.795)
\psline[linecolor=black, linewidth=0.02](0.3,-0.495)(8.3,-0.495)
\psline[linecolor=black, linewidth=0.02](0.3,0.505)(8.3,0.505)
\psdots[linecolor=black, dotstyle=o, dotsize=0.2, fillcolor=white](1.1,0.505)
\psdots[linecolor=black, dotstyle=o, dotsize=0.2, fillcolor=white](1.9,0.505)
\psdots[linecolor=black, dotstyle=o, dotsize=0.2, fillcolor=white](3.5,0.505)
\psdots[linecolor=black, dotstyle=o, dotsize=0.2, fillcolor=white](4.3,0.505)
\psdots[linecolor=black, dotstyle=o, dotsize=0.2, fillcolor=white](5.1,0.505)
\psdots[linecolor=black, dotstyle=o, dotsize=0.2, fillcolor=white](6.7,0.505)
\psdots[linecolor=black, dotstyle=o, dotsize=0.2, fillcolor=white](7.5,0.505)
\rput[bl](4.1,0.705){$-2$}
\rput[bl](3.3,0.705){$-2$}
\rput[bl](2.5,0.705){$-2$}
\rput[bl](1.7,0.705){$-2$}
\rput[bl](0.9,0.705){$-2$}
\rput[bl](0.1,0.705){$-2$}
\rput[bl](4.9,0.705){$-2$}
\rput[bl](2.5,1.105){$L_1$}
\rput[bl](5.7,0.705){$-2$}
\rput[bl](6.5,0.705){$-2$}
\rput[bl](7.3,0.705){$-2$}
\rput[bl](5.7,1.105){$L_2$}
\rput[bl](0.0,-0.995){$-2$}
\rput[bl](8.1,0.705){$-2$}
\psdots[linecolor=black, dotstyle=o, dotsize=0.2, fillcolor=white](1.1,-0.495)
\psdots[linecolor=black, dotstyle=o, dotsize=0.2, fillcolor=white](1.9,-0.495)
\psdots[linecolor=black, dotstyle=o, dotsize=0.2, fillcolor=white](3.5,-0.495)
\psdots[linecolor=black, dotstyle=o, dotsize=0.2, fillcolor=white](4.3,-0.495)
\psdots[linecolor=black, dotstyle=o, dotsize=0.2, fillcolor=white](5.1,-0.495)
\psdots[linecolor=black, dotstyle=o, dotsize=0.2, fillcolor=white](6.7,-0.495)
\psdots[linecolor=black, dotstyle=o, dotsize=0.2, fillcolor=white](7.5,-0.495)
\rput[bl](0.8,-0.995){$-2$}
\rput[bl](1.6,-0.995){$-2$}
\rput[bl](2.4,-0.995){$-2$}
\rput[bl](3.2,-0.995){$-2$}
\rput[bl](4.0,-0.995){$-2$}
\rput[bl](4.8,-0.995){$-2$}
\rput[bl](5.6,-0.995){$-2$}
\rput[bl](6.4,-0.995){$-2$}
\rput[bl](7.2,-0.995){$-2$}
\rput[bl](8.0,-0.995){$-2$}
\rput[bl](2.4,-1.395){$L_3$}
\rput[bl](5.6,-1.395){$L_4$}
\rput[bl](0.1,1.105){$F_1$}
\rput[bl](0.9,1.105){$F_2$}
\rput[bl](1.7,1.105){$F_3$}
\rput[bl](3.3,1.105){$F_4$}
\rput[bl](4.1,1.105){$F_5$}
\rput[bl](4.9,1.105){$F_6$}
\rput[bl](6.5,1.105){$F_7$}
\rput[bl](7.3,1.105){$F_8$}
\rput[bl](8.1,1.105){$F_9$}
\psline[linecolor=black, linewidth=0.02](5.9,0.505)(8.3,-0.495)
\rput[bl](0.0,-1.395){$F_{10}$}
\rput[bl](0.8,-1.395){$F_{11}$}
\rput[bl](1.6,-1.395){$F_{12}$}
\rput[bl](3.2,-1.395){$F_{13}$}
\rput[bl](4.0,-1.395){$F_{14}$}
\rput[bl](4.8,-1.395){$F_{15}$}
\rput[bl](6.4,-1.395){$F_{16}$}
\rput[bl](7.2,-1.395){$F_{17}$}
\rput[bl](8.0,-1.395){$F_{18}$}
\psframe[linecolor=white, linewidth=0.02, fillstyle=solid, dimen=outer](6.9,0.305)(6.5,0.105)
\psline[linecolor=black, linewidth=0.02](0.3,-0.495)(2.7,0.505)
\psframe[linecolor=white, linewidth=0.02, fillstyle=solid, dimen=outer](2.1,0.305)(1.7,0.105)
\psline[linecolor=black, linewidth=0.02](5.9,-0.495)(0.3,0.505)
\psframe[linecolor=white, linewidth=0.02, fillstyle=solid, dimen=outer](4.5,-0.095)(4.1,-0.295)
\psline[linecolor=black, linewidth=0.02](2.7,-0.495)(8.3,0.505)
\psdots[linecolor=black, dotstyle=o, dotsize=0.2, fillcolor=white](8.3,0.505)
\psdots[linecolor=black, dotstyle=o, dotsize=0.2, fillcolor=white](0.3,0.505)
\psdots[linecolor=black, dotstyle=o, dotsize=0.2, fillcolor=white](2.7,0.505)
\psdots[linecolor=black, dotstyle=o, dotsize=0.2, fillcolor=white](5.9,0.505)
\psdots[linecolor=black, dotstyle=o, dotsize=0.2, fillcolor=white](0.3,-0.495)
\psdots[linecolor=black, dotstyle=o, dotsize=0.2, fillcolor=white](2.7,-0.495)
\psdots[linecolor=black, dotstyle=o, dotsize=0.2, fillcolor=white](5.9,-0.495)
\psdots[linecolor=black, dotstyle=o, dotsize=0.2, fillcolor=white](8.3,-0.495)
\end{pspicture}
} \\
\hline
$X(8,9,14,22)$ & $5$ & \psscalebox{.6 .6} 
{
\begin{pspicture}(0,-1.495)(9.32,1.795)
\psline[linecolor=black, linewidth=0.02](6.6,-0.495)(0.2,0.505)
\psframe[linecolor=white, linewidth=0.02, fillstyle=solid, dimen=outer](1.8,0.505)(1.4,0.105)
\psline[linecolor=black, linewidth=0.02](1.0,-0.495)(8.2,-0.495)
\psline[linecolor=black, linewidth=0.02](0.2,0.505)(9.0,0.505)
\psdots[linecolor=black, dotstyle=o, dotsize=0.2, fillcolor=white](1.0,0.505)
\psdots[linecolor=black, dotstyle=o, dotsize=0.2, fillcolor=white](3.4,0.505)
\psdots[linecolor=black, dotstyle=o, dotsize=0.2, fillcolor=white](4.2,0.505)
\psdots[linecolor=black, dotstyle=o, dotsize=0.2, fillcolor=white](5.0,0.505)
\psdots[linecolor=black, dotstyle=o, dotsize=0.2, fillcolor=white](6.6,0.505)
\psdots[linecolor=black, dotstyle=o, dotsize=0.2, fillcolor=white](7.4,0.505)
\rput[bl](4.0,0.705){$-2$}
\rput[bl](3.2,0.705){$-2$}
\rput[bl](2.4,0.705){$-2$}
\rput[bl](1.6,0.705){$-2$}
\rput[bl](0.8,0.705){$-2$}
\rput[bl](0.0,0.705){$-3$}
\rput[bl](4.8,0.705){$-2$}
\rput[bl](2.4,1.105){$F_3$}
\rput[bl](5.6,0.705){$-2$}
\rput[bl](6.4,0.705){$-2$}
\rput[bl](7.2,0.705){$-2$}
\rput[bl](5.6,1.105){$L_2$}
\rput[bl](8.0,0.705){$-2$}
\psdots[linecolor=black, dotstyle=o, dotsize=0.2, fillcolor=white](1.8,-0.495)
\psdots[linecolor=black, dotstyle=o, dotsize=0.2, fillcolor=white](3.4,-0.495)
\psdots[linecolor=black, dotstyle=o, dotsize=0.2, fillcolor=white](7.4,-0.495)
\rput[bl](0.7,-0.995){$-2$}
\rput[bl](1.5,-0.995){$-2$}
\rput[bl](2.3,-0.995){$-2$}
\rput[bl](3.1,-0.995){$-2$}
\rput[bl](3.9,-0.995){$-2$}
\rput[bl](4.7,-0.995){$-2$}
\rput[bl](5.5,-0.995){$-3$}
\rput[bl](6.3,-0.995){$-1$}
\rput[bl](7.1,-0.995){$-3$}
\rput[bl](7.9,-0.995){$-2$}
\rput[bl](3.9,-1.395){$L_3$}
\rput[bl](0.0,1.105){$F_1$}
\rput[bl](0.8,1.105){$F_2$}
\rput[bl](1.6,1.105){$L_1$}
\rput[bl](3.2,1.105){$F_4$}
\rput[bl](4.0,1.105){$F_5$}
\rput[bl](4.8,1.105){$F_6$}
\rput[bl](6.4,1.105){$F_7$}
\rput[bl](7.2,1.105){$F_8$}
\rput[bl](8.0,1.105){$F_9$}
\psline[linecolor=black, linewidth=0.02](5.8,0.505)(8.2,-0.495)
\rput[bl](0.7,-1.395){$F_{11}$}
\rput[bl](1.5,-1.395){$F_{12}$}
\rput[bl](2.3,-1.395){$F_{13}$}
\rput[bl](3.1,-1.395){$F_{14}$}
\rput[bl](4.7,-1.395){$F_{15}$}
\rput[bl](5.5,-1.395){$F_{16}$}
\rput[bl](7.1,-1.395){$F_{17}$}
\rput[bl](7.9,-1.395){$F_{18}$}
\psframe[linecolor=white, linewidth=0.02, fillstyle=solid, dimen=outer](7.0,0.105)(6.6,-0.095)
\psframe[linecolor=white, linewidth=0.02, fillstyle=solid, dimen=outer](5.4,-0.095)(5.0,-0.495)
\psline[linecolor=black, linewidth=0.02](4.2,-0.495)(9.0,0.505)
\psdots[linecolor=black, dotstyle=o, dotsize=0.2, fillcolor=white](8.2,0.505)
\psdots[linecolor=black, dotstyle=o, dotsize=0.2, fillcolor=white](0.2,0.505)
\psdots[linecolor=black, dotstyle=o, dotsize=0.2, fillcolor=white](2.6,0.505)
\psdots[linecolor=black, dotstyle=o, dotsize=0.2, fillcolor=white](5.8,0.505)
\psdots[linecolor=black, dotstyle=o, dotsize=0.2, fillcolor=white](2.6,-0.495)
\psdots[linecolor=black, dotstyle=o, dotsize=0.2, fillcolor=white](5.8,-0.495)
\psdots[linecolor=black, dotstyle=o, dotsize=0.2, fillcolor=white](8.2,-0.495)
\rput[bl](8.8,0.705){$-2$}
\rput[bl](8.8,1.105){$F_{10}$}
\psdots[linecolor=black, dotstyle=o, dotsize=0.2, fillcolor=white](5.0,-0.495)
\psdots[linecolor=black, dotsize=0.2](6.6,-0.495)
\psline[linecolor=black, linewidth=0.02](1.0,-0.495)(1.8,0.505)
\psdots[linecolor=black, dotstyle=o, dotsize=0.2, fillcolor=white](1.0,-0.495)
\psdots[linecolor=black, dotstyle=o, dotsize=0.2, fillcolor=white](4.2,-0.495)
\rput[bl](6.3,-1.395){$L_4$}
\psdots[linecolor=black, dotstyle=o, dotsize=0.2, fillcolor=white](1.8,0.505)
\psdots[linecolor=black, dotstyle=o, dotsize=0.2, fillcolor=white](9.0,0.505)
\end{pspicture}
} \\
\hline
$X(11,27,10,18)$ & $7$ & \psscalebox{.6 .6} 
{
\begin{pspicture}(0,-1.495)(10.12,1.795)
\psline[linecolor=black, linewidth=0.02](6.6,-0.495)(0.2,0.505)
\psframe[linecolor=white, linewidth=0.02, fillstyle=solid, dimen=outer](2.0,0.505)(1.6,0.105)
\psline[linecolor=black, linewidth=0.02](1.8,-0.495)(8.2,-0.495)
\psline[linecolor=black, linewidth=0.02](0.2,0.505)(9.8,0.505)
\psdots[linecolor=black, dotstyle=o, dotsize=0.2, fillcolor=white](1.0,0.505)
\psdots[linecolor=black, dotstyle=o, dotsize=0.2, fillcolor=white](3.4,0.505)
\psdots[linecolor=black, dotstyle=o, dotsize=0.2, fillcolor=white](4.2,0.505)
\psdots[linecolor=black, dotstyle=o, dotsize=0.2, fillcolor=white](6.6,0.505)
\psdots[linecolor=black, dotstyle=o, dotsize=0.2, fillcolor=white](7.4,0.505)
\rput[bl](4.0,0.705){$-2$}
\rput[bl](3.2,0.705){$-2$}
\rput[bl](2.4,0.705){$-3$}
\rput[bl](1.6,0.705){$-1$}
\rput[bl](0.8,0.705){$-4$}
\rput[bl](0.0,0.705){$-2$}
\rput[bl](4.8,0.705){$-2$}
\rput[bl](2.4,1.105){$F_3$}
\rput[bl](5.6,0.705){$-2$}
\rput[bl](6.4,0.705){$-2$}
\rput[bl](7.2,0.705){$-2$}
\rput[bl](4.8,1.105){$L_2$}
\rput[bl](8.0,0.705){$-2$}
\psdots[linecolor=black, dotstyle=o, dotsize=0.2, fillcolor=white](3.4,-0.495)
\psdots[linecolor=black, dotstyle=o, dotsize=0.2, fillcolor=white](7.4,-0.495)
\rput[bl](1.5,-0.995){$-3$}
\rput[bl](2.3,-0.995){$-2$}
\rput[bl](3.1,-0.995){$-2$}
\rput[bl](3.9,-0.995){$-2$}
\rput[bl](4.7,-0.995){$-2$}
\rput[bl](5.5,-0.995){$-4$}
\rput[bl](6.3,-0.995){$-1$}
\rput[bl](7.1,-0.995){$-4$}
\rput[bl](7.9,-0.995){$-2$}
\rput[bl](3.9,-1.395){$L_3$}
\rput[bl](0.0,1.105){$F_1$}
\rput[bl](0.8,1.105){$F_2$}
\rput[bl](1.6,1.105){$L_1$}
\rput[bl](3.2,1.105){$F_4$}
\rput[bl](4.0,1.105){$F_5$}
\rput[bl](5.6,1.105){$F_6$}
\rput[bl](6.4,1.105){$F_7$}
\rput[bl](7.2,1.105){$F_8$}
\rput[bl](8.0,1.105){$F_9$}
\psline[linecolor=black, linewidth=0.02](5.0,0.505)(8.2,-0.495)
\rput[bl](1.5,-1.395){$F_{12}$}
\rput[bl](2.3,-1.395){$F_{13}$}
\rput[bl](3.1,-1.395){$F_{14}$}
\rput[bl](4.7,-1.395){$F_{15}$}
\rput[bl](5.5,-1.395){$F_{16}$}
\rput[bl](7.1,-1.395){$F_{17}$}
\rput[bl](7.9,-1.395){$F_{18}$}
\psframe[linecolor=white, linewidth=0.02, fillstyle=solid, dimen=outer](7.0,0.105)(6.6,-0.295)
\psframe[linecolor=white, linewidth=0.02, fillstyle=solid, dimen=outer](5.6,-0.095)(5.2,-0.495)
\psline[linecolor=black, linewidth=0.02](4.2,-0.495)(9.8,0.505)
\psdots[linecolor=black, dotstyle=o, dotsize=0.2, fillcolor=white](8.2,0.505)
\psdots[linecolor=black, dotstyle=o, dotsize=0.2, fillcolor=white](0.2,0.505)
\psdots[linecolor=black, dotstyle=o, dotsize=0.2, fillcolor=white](2.6,0.505)
\psdots[linecolor=black, dotstyle=o, dotsize=0.2, fillcolor=white](5.8,0.505)
\psdots[linecolor=black, dotstyle=o, dotsize=0.2, fillcolor=white](2.6,-0.495)
\psdots[linecolor=black, dotstyle=o, dotsize=0.2, fillcolor=white](5.8,-0.495)
\psdots[linecolor=black, dotstyle=o, dotsize=0.2, fillcolor=white](8.2,-0.495)
\rput[bl](8.8,0.705){$-2$}
\rput[bl](8.8,1.105){$F_{10}$}
\psdots[linecolor=black, dotstyle=o, dotsize=0.2, fillcolor=white](5.0,-0.495)
\psdots[linecolor=black, dotsize=0.2](6.6,-0.495)
\psline[linecolor=black, linewidth=0.02](1.8,-0.495)(1.8,0.505)
\psdots[linecolor=black, dotstyle=o, dotsize=0.2, fillcolor=white](4.2,-0.495)
\rput[bl](6.3,-1.395){$L_4$}
\psdots[linecolor=black, dotstyle=o, dotsize=0.2, fillcolor=white](9.0,0.505)
\psdots[linecolor=black, dotstyle=o, dotsize=0.2, fillcolor=white](9.8,0.505)
\rput[bl](9.6,0.705){$-2$}
\rput[bl](9.6,1.105){$F_{11}$}
\psdots[linecolor=black, dotstyle=o, dotsize=0.2, fillcolor=white](5.0,0.505)
\psdots[linecolor=black, dotstyle=o, dotsize=0.2, fillcolor=white](1.8,-0.495)
\psdots[linecolor=black, dotsize=0.2](1.8,0.505)
\end{pspicture}
} \\
\hline
$X(17,14,42,18)$ & $11$ & \psscalebox{.6 .6} 
{
\begin{pspicture}(0,-1.495)(8.54,1.795)
\psline[linecolor=black, linewidth=0.02](6.7,-0.495)(1.0,0.505)
\psframe[linecolor=white, linewidth=0.02, fillstyle=solid, dimen=outer](3.1,0.305)(2.7,-0.095)
\psline[linecolor=black, linewidth=0.02](0.3,-0.495)(8.3,-0.495)
\psline[linecolor=black, linewidth=0.02](1.0,0.505)(8.3,0.505)
\psdots[linecolor=black, dotstyle=o, dotsize=0.2, fillcolor=white](1.1,0.505)
\psdots[linecolor=black, dotstyle=o, dotsize=0.2, fillcolor=white](3.5,0.505)
\psdots[linecolor=black, dotstyle=o, dotsize=0.2, fillcolor=white](7.5,0.505)
\rput[bl](4.1,0.705){$-2$}
\rput[bl](3.3,0.705){$-2$}
\rput[bl](2.5,0.705){$-2$}
\rput[bl](1.7,0.705){$-3$}
\rput[bl](0.9,0.705){$-2$}
\rput[bl](4.9,0.705){$-2$}
\rput[bl](2.5,1.105){$F_3$}
\rput[bl](5.7,0.705){$-6$}
\rput[bl](6.5,0.705){$-1$}
\rput[bl](7.3,0.705){$-3$}
\rput[bl](6.5,1.105){$L_2$}
\rput[bl](8.1,0.705){$-4$}
\psdots[linecolor=black, dotstyle=o, dotsize=0.2, fillcolor=white](3.5,-0.495)
\psdots[linecolor=black, dotstyle=o, dotsize=0.2, fillcolor=white](7.5,-0.495)
\rput[bl](1.6,-0.995){$-2$}
\rput[bl](2.4,-0.995){$-2$}
\rput[bl](3.2,-0.995){$-3$}
\rput[bl](4.0,-0.995){$-1$}
\rput[bl](4.8,-0.995){$-3$}
\rput[bl](5.6,-0.995){$-4$}
\rput[bl](6.4,-0.995){$-1$}
\rput[bl](7.2,-0.995){$-6$}
\rput[bl](8.0,-0.995){$-2$}
\rput[bl](4.0,-1.395){$L_3$}
\rput[bl](0.9,1.105){$F_1$}
\rput[bl](1.7,1.105){$F_2$}
\rput[bl](4.1,1.105){$L_1$}
\rput[bl](3.3,1.105){$F_4$}
\rput[bl](4.9,1.105){$F_5$}
\rput[bl](5.7,1.105){$F_6$}
\rput[bl](7.3,1.105){$F_7$}
\rput[bl](8.1,1.105){$F_8$}
\psline[linecolor=black, linewidth=0.02](6.7,0.505)(8.3,-0.495)
\rput[bl](2.4,-1.395){$F_{12}$}
\rput[bl](3.2,-1.395){$F_{13}$}
\rput[bl](4.8,-1.395){$F_{14}$}
\rput[bl](5.6,-1.395){$F_{15}$}
\rput[bl](7.2,-1.395){$F_{16}$}
\rput[bl](8.0,-1.395){$F_{17}$}
\psframe[linecolor=white, linewidth=0.02, fillstyle=solid, dimen=outer](7.3,0.305)(7.1,-0.095)
\psframe[linecolor=white, linewidth=0.02, fillstyle=solid, dimen=outer](5.5,-0.095)(5.1,-0.495)
\psline[linecolor=black, linewidth=0.02](4.3,-0.495)(8.3,0.505)
\psdots[linecolor=black, dotstyle=o, dotsize=0.2, fillcolor=white](8.3,0.505)
\psdots[linecolor=black, dotstyle=o, dotsize=0.2, fillcolor=white](2.7,0.505)
\psdots[linecolor=black, dotstyle=o, dotsize=0.2, fillcolor=white](5.9,0.505)
\psdots[linecolor=black, dotstyle=o, dotsize=0.2, fillcolor=white](2.7,-0.495)
\psdots[linecolor=black, dotstyle=o, dotsize=0.2, fillcolor=white](5.9,-0.495)
\psdots[linecolor=black, dotstyle=o, dotsize=0.2, fillcolor=white](8.3,-0.495)
\psdots[linecolor=black, dotstyle=o, dotsize=0.2, fillcolor=white](5.1,-0.495)
\psdots[linecolor=black, dotsize=0.2](6.7,-0.495)
\psline[linecolor=black, linewidth=0.02](0.3,-0.495)(4.3,0.505)
\psdots[linecolor=black, dotsize=0.2](4.3,-0.495)
\rput[bl](6.4,-1.395){$L_4$}
\psdots[linecolor=black, dotstyle=o, dotsize=0.2, fillcolor=white](5.1,0.505)
\psdots[linecolor=black, dotstyle=o, dotsize=0.2, fillcolor=white](1.9,-0.495)
\psdots[linecolor=black, dotstyle=o, dotsize=0.2, fillcolor=white](1.9,0.505)
\psdots[linecolor=black, dotstyle=o, dotsize=0.2, fillcolor=white](1.1,-0.495)
\psdots[linecolor=black, dotstyle=o, dotsize=0.2, fillcolor=white](0.3,-0.495)
\rput[bl](0.0,-0.995){$-2$}
\rput[bl](0.8,-0.995){$-2$}
\rput[bl](0.8,-1.395){$F_{10}$}
\rput[bl](1.6,-1.395){$F_{11}$}
\psdots[linecolor=black, dotstyle=o, dotsize=0.2, fillcolor=white](4.3,0.505)
\psdots[linecolor=black, dotsize=0.2](6.7,0.505)
\rput[bl](0.2,-1.395){$F_{9}$}
\end{pspicture}
} \\
\hline
$X(20,21,43,22)$ & $13$ & \psscalebox{.6 .6} 
{
\begin{pspicture}(0,-1.495)(8.72,1.795)
\psline[linecolor=black, linewidth=0.02](6.7,-0.495)(1.0,0.505)
\psframe[linecolor=white, linewidth=0.02, fillstyle=solid, dimen=outer](2.7,0.305)(2.5,-0.095)
\psline[linecolor=black, linewidth=0.02](0.3,-0.495)(8.3,-0.495)
\psline[linecolor=black, linewidth=0.02](1.0,0.505)(8.3,0.505)
\psdots[linecolor=black, dotstyle=o, dotsize=0.2, fillcolor=white](1.1,0.505)
\psdots[linecolor=black, dotstyle=o, dotsize=0.2, fillcolor=white](7.5,0.505)
\rput[bl](4.1,0.705){$-2$}
\rput[bl](3.3,0.705){$-1$}
\rput[bl](2.5,0.705){$-5$}
\rput[bl](1.7,0.705){$-2$}
\rput[bl](0.9,0.705){$-2$}
\rput[bl](4.9,0.705){$-7$}
\rput[bl](2.5,1.105){$F_3$}
\rput[bl](5.7,0.705){$-1$}
\rput[bl](6.5,0.705){$-3$}
\rput[bl](7.3,0.705){$-3$}
\rput[bl](5.7,1.105){$L_2$}
\rput[bl](8.1,0.705){$-2$}
\psdots[linecolor=black, dotstyle=o, dotsize=0.2, fillcolor=white](7.5,-0.495)
\rput[bl](1.6,-0.995){$-2$}
\rput[bl](2.4,-0.995){$-2$}
\rput[bl](3.2,-0.995){$-2$}
\rput[bl](4.0,-0.995){$-2$}
\rput[bl](4.8,-0.995){$-2$}
\rput[bl](5.6,-0.995){$-5$}
\rput[bl](6.4,-0.995){$-1$}
\rput[bl](7.2,-0.995){$-7$}
\rput[bl](8.0,-0.995){$-2$}
\rput[bl](3.4,-1.395){$L_3$}
\rput[bl](0.9,1.105){$F_1$}
\rput[bl](1.7,1.105){$F_2$}
\rput[bl](3.3,1.105){$L_1$}
\rput[bl](4.1,1.105){$F_4$}
\rput[bl](4.9,1.105){$F_5$}
\rput[bl](6.5,1.105){$F_6$}
\rput[bl](7.3,1.105){$F_7$}
\rput[bl](8.1,1.105){$F_8$}
\psline[linecolor=black, linewidth=0.02](5.9,0.505)(8.3,-0.495)
\rput[bl](2.6,-1.395){$F_{12}$}
\rput[bl](4.2,-1.395){$F_{13}$}
\rput[bl](5.0,-1.395){$F_{14}$}
\rput[bl](5.8,-1.395){$F_{15}$}
\rput[bl](7.4,-1.395){$F_{16}$}
\rput[bl](8.2,-1.395){$F_{17}$}
\psframe[linecolor=white, linewidth=0.02, fillstyle=solid, dimen=outer](6.9,0.305)(6.5,-0.095)
\psframe[linecolor=white, linewidth=0.02, fillstyle=solid, dimen=outer](5.1,-0.095)(4.7,-0.495)
\psline[linecolor=black, linewidth=0.02](3.5,-0.495)(8.3,0.505)
\psdots[linecolor=black, dotstyle=o, dotsize=0.2, fillcolor=white](8.3,0.505)
\psdots[linecolor=black, dotstyle=o, dotsize=0.2, fillcolor=white](2.7,0.505)
\psdots[linecolor=black, dotsize=0.2](5.9,0.505)
\psdots[linecolor=black, dotstyle=o, dotsize=0.2, fillcolor=white](2.7,-0.495)
\psdots[linecolor=black, dotstyle=o, dotsize=0.2, fillcolor=white](5.9,-0.495)
\psdots[linecolor=black, dotstyle=o, dotsize=0.2, fillcolor=white](8.3,-0.495)
\psdots[linecolor=black, dotstyle=o, dotsize=0.2, fillcolor=white](5.1,-0.495)
\psdots[linecolor=black, dotsize=0.2](6.7,-0.495)
\psline[linecolor=black, linewidth=0.02](0.3,-0.495)(3.5,0.505)
\psdots[linecolor=black, dotstyle=o, dotsize=0.2, fillcolor=white](4.3,-0.495)
\rput[bl](6.6,-1.395){$L_4$}
\psdots[linecolor=black, dotstyle=o, dotsize=0.2, fillcolor=white](5.1,0.505)
\psdots[linecolor=black, dotstyle=o, dotsize=0.2, fillcolor=white](1.9,-0.495)
\psdots[linecolor=black, dotstyle=o, dotsize=0.2, fillcolor=white](1.9,0.505)
\psdots[linecolor=black, dotstyle=o, dotsize=0.2, fillcolor=white](1.1,-0.495)
\psdots[linecolor=black, dotstyle=o, dotsize=0.2, fillcolor=white](0.3,-0.495)
\rput[bl](0.0,-0.995){$-4$}
\rput[bl](0.8,-0.995){$-2$}
\rput[bl](1.0,-1.395){$F_{10}$}
\rput[bl](1.8,-1.395){$F_{11}$}
\psdots[linecolor=black, dotstyle=o, dotsize=0.2, fillcolor=white](4.3,0.505)
\psdots[linecolor=black, dotstyle=o, dotsize=0.2, fillcolor=white](6.7,0.505)
\psdots[linecolor=black, dotsize=0.2](3.5,0.505)
\psdots[linecolor=black, dotstyle=o, dotsize=0.2, fillcolor=white](3.5,-0.495)
\rput[bl](0.2,-1.395){$F_{9}$}
\end{pspicture}
} \\
\hline
$X(26,56,39,64)$ & $17$ & \psscalebox{.6 .6} 
{
\begin{pspicture}(0,-1.495)(8.44,1.795)
\psline[linecolor=black, linewidth=0.02](5.8,-0.495)(0.2,0.505)
\psframe[linecolor=white, linewidth=0.02, fillstyle=solid, dimen=outer](2.6,0.305)(2.2,-0.095)
\psline[linecolor=black, linewidth=0.02](1.0,-0.495)(7.4,-0.495)
\psline[linecolor=black, linewidth=0.02](0.2,0.505)(8.2,0.505)
\psdots[linecolor=black, dotstyle=o, dotsize=0.2, fillcolor=white](1.0,0.505)
\psdots[linecolor=black, dotstyle=o, dotsize=0.2, fillcolor=white](7.4,0.505)
\rput[bl](4.0,0.705){$-2$}
\rput[bl](3.2,0.705){$-1$}
\rput[bl](2.4,0.705){$-5$}
\rput[bl](1.6,0.705){$-2$}
\rput[bl](0.8,0.705){$-2$}
\rput[bl](0.0,0.705){$-2$}
\rput[bl](4.8,0.705){$-9$}
\rput[bl](1.6,1.105){$F_3$}
\rput[bl](5.6,0.705){$-1$}
\rput[bl](6.4,0.705){$-3$}
\rput[bl](7.2,0.705){$-2$}
\rput[bl](5.6,1.105){$L_2$}
\rput[bl](8.0,0.705){$-4$}
\rput[bl](1.5,-0.995){$-2$}
\rput[bl](2.3,-0.995){$-3$}
\rput[bl](3.1,-0.995){$-1$}
\rput[bl](3.9,-0.995){$-3$}
\rput[bl](4.7,-0.995){$-6$}
\rput[bl](5.5,-0.995){$-1$}
\rput[bl](6.3,-0.995){$-9$}
\rput[bl](7.1,-0.995){$-2$}
\rput[bl](3.1,-1.395){$L_3$}
\rput[bl](0.0,1.105){$F_1$}
\rput[bl](0.8,1.105){$F_2$}
\rput[bl](3.2,1.105){$L_1$}
\rput[bl](2.4,1.105){$F_4$}
\rput[bl](4.0,1.105){$F_5$}
\rput[bl](4.8,1.105){$F_6$}
\rput[bl](6.4,1.105){$F_7$}
\rput[bl](7.2,1.105){$F_8$}
\rput[bl](8.0,1.105){$F_9$}
\psline[linecolor=black, linewidth=0.02](5.8,0.505)(7.4,-0.495)
\rput[bl](2.3,-1.395){$F_{12}$}
\rput[bl](3.9,-1.395){$F_{13}$}
\rput[bl](4.7,-1.395){$F_{14}$}
\rput[bl](6.3,-1.395){$F_{15}$}
\rput[bl](7.1,-1.395){$F_{16}$}
\psframe[linecolor=white, linewidth=0.02, fillstyle=solid, dimen=outer](6.6,0.305)(6.2,-0.095)
\psframe[linecolor=white, linewidth=0.02, fillstyle=solid, dimen=outer](4.8,-0.095)(4.4,-0.495)
\psline[linecolor=black, linewidth=0.02](3.4,-0.495)(8.2,0.505)
\psdots[linecolor=black, dotstyle=o, dotsize=0.2, fillcolor=white](8.2,0.505)
\psdots[linecolor=black, dotstyle=o, dotsize=0.2, fillcolor=white](0.2,0.505)
\psdots[linecolor=black, dotstyle=o, dotsize=0.2, fillcolor=white](2.6,0.505)
\psdots[linecolor=black, dotsize=0.2](5.8,0.505)
\psdots[linecolor=black, dotstyle=o, dotsize=0.2, fillcolor=white](2.6,-0.495)
\psdots[linecolor=black, dotsize=0.2](5.8,-0.495)
\psdots[linecolor=black, dotstyle=o, dotsize=0.2, fillcolor=white](5.0,-0.495)
\psdots[linecolor=black, dotstyle=o, dotsize=0.2, fillcolor=white](6.6,-0.495)
\psline[linecolor=black, linewidth=0.02](1.0,-0.495)(3.4,0.505)
\psdots[linecolor=black, dotstyle=o, dotsize=0.2, fillcolor=white](4.2,-0.495)
\rput[bl](5.5,-1.395){$L_4$}
\psdots[linecolor=black, dotstyle=o, dotsize=0.2, fillcolor=white](5.0,0.505)
\psdots[linecolor=black, dotstyle=o, dotsize=0.2, fillcolor=white](1.8,-0.495)
\psdots[linecolor=black, dotstyle=o, dotsize=0.2, fillcolor=white](1.8,0.505)
\psdots[linecolor=black, dotstyle=o, dotsize=0.2, fillcolor=white](1.0,-0.495)
\rput[bl](0.7,-0.995){$-4$}
\rput[bl](0.7,-1.395){$F_{10}$}
\rput[bl](1.5,-1.395){$F_{11}$}
\psdots[linecolor=black, dotstyle=o, dotsize=0.2, fillcolor=white](4.2,0.505)
\psdots[linecolor=black, dotstyle=o, dotsize=0.2, fillcolor=white](6.6,0.505)
\psdots[linecolor=black, dotsize=0.2](3.4,0.505)
\psdots[linecolor=black, dotsize=0.2](3.4,-0.495)
\psdots[linecolor=black, dotstyle=o, dotsize=0.2, fillcolor=white](7.4,-0.495)
\end{pspicture}
} \\
\hline
$X(29,30,42,32)$ & $19$ & \psscalebox{.6 .6} 
{
\begin{pspicture}(0,-1.495)(8.44,1.795)
\psline[linecolor=black, linewidth=0.02](5.0,-0.495)(0.2,0.505)
\psframe[linecolor=white, linewidth=0.02, fillstyle=solid, dimen=outer](2.2,0.305)(1.8,-0.095)
\psline[linecolor=black, linewidth=0.02](1.0,-0.495)(6.6,-0.495)
\psline[linecolor=black, linewidth=0.02](0.2,0.505)(8.2,0.505)
\psdots[linecolor=black, dotstyle=o, dotsize=0.2, fillcolor=white](0.2,0.505)
\psdots[linecolor=black, dotstyle=o, dotsize=0.2, fillcolor=white](6.6,0.505)
\rput[bl](3.2,0.705){$-2$}
\rput[bl](2.4,0.705){$-1$}
\rput[bl](1.6,0.705){$-7$}
\rput[bl](0.8,0.705){$-2$}
\rput[bl](0.0,0.705){$-2$}
\rput[bl](4.0,0.705){$-2$}
\rput[bl](1.6,1.105){$F_3$}
\rput[bl](4.8,0.705){$-7$}
\rput[bl](5.6,0.705){$-1$}
\rput[bl](6.4,0.705){$-3$}
\rput[bl](6.4,1.105){$F_7$}
\rput[bl](7.2,0.705){$-4$}
\rput[bl](0.7,-0.995){$-5$}
\rput[bl](1.5,-0.995){$-4$}
\rput[bl](2.3,-0.995){$-1$}
\rput[bl](3.1,-0.995){$-5$}
\rput[bl](3.9,-0.995){$-4$}
\rput[bl](4.7,-0.995){$-1$}
\rput[bl](5.5,-0.995){$-10$}
\rput[bl](6.3,-0.995){$-2$}
\rput[bl](2.3,-1.395){$L_3$}
\rput[bl](0.0,1.105){$F_1$}
\rput[bl](0.8,1.105){$F_2$}
\rput[bl](3.2,1.105){$L_1$}
\rput[bl](2.4,1.105){$F_4$}
\rput[bl](4.0,1.105){$F_5$}
\rput[bl](4.8,1.105){$F_6$}
\rput[bl](5.6,1.105){$L_2$}
\rput[bl](7.2,1.105){$F_8$}
\rput[bl](8.0,1.105){$F_9$}
\psline[linecolor=black, linewidth=0.02](5.8,0.505)(6.6,-0.495)
\rput[bl](3.1,-1.395){$F_{12}$}
\rput[bl](3.9,-1.395){$F_{13}$}
\rput[bl](5.5,-1.395){$F_{14}$}
\rput[bl](6.3,-1.395){$F_{15}$}
\psframe[linecolor=white, linewidth=0.02, fillstyle=solid, dimen=outer](6.2,0.305)(6.0,-0.095)
\psframe[linecolor=white, linewidth=0.02, fillstyle=solid, dimen=outer](4.0,-0.095)(3.6,-0.495)
\psline[linecolor=black, linewidth=0.02](2.6,-0.495)(8.2,0.505)
\psdots[linecolor=black, dotstyle=o, dotsize=0.2, fillcolor=white](7.4,0.505)
\psdots[linecolor=black, dotstyle=o, dotsize=0.2, fillcolor=white](1.8,0.505)
\psdots[linecolor=black, dotstyle=o, dotsize=0.2, fillcolor=white](5.0,0.505)
\psdots[linecolor=black, dotstyle=o, dotsize=0.2, fillcolor=white](1.8,-0.495)
\psdots[linecolor=black, dotsize=0.2](5.0,-0.495)
\psdots[linecolor=black, dotstyle=o, dotsize=0.2, fillcolor=white](4.2,-0.495)
\psdots[linecolor=black, dotstyle=o, dotsize=0.2, fillcolor=white](5.8,-0.495)
\psline[linecolor=black, linewidth=0.02](1.0,-0.495)(2.6,0.505)
\psdots[linecolor=black, dotstyle=o, dotsize=0.2, fillcolor=white](3.4,-0.495)
\rput[bl](4.7,-1.395){$L_4$}
\psdots[linecolor=black, dotstyle=o, dotsize=0.2, fillcolor=white](4.2,0.505)
\psdots[linecolor=black, dotstyle=o, dotsize=0.2, fillcolor=white](1.0,-0.495)
\psdots[linecolor=black, dotstyle=o, dotsize=0.2, fillcolor=white](1.0,0.505)
\rput[bl](1.5,-1.395){$F_{11}$}
\psdots[linecolor=black, dotstyle=o, dotsize=0.2, fillcolor=white](3.4,0.505)
\psdots[linecolor=black, dotsize=0.2](5.8,0.505)
\psdots[linecolor=black, dotsize=0.2](2.6,0.505)
\psdots[linecolor=black, dotsize=0.2](2.6,-0.495)
\psdots[linecolor=black, dotstyle=o, dotsize=0.2, fillcolor=white](6.6,-0.495)
\rput[bl](8.0,0.705){$-2$}
\psdots[linecolor=black, dotstyle=o, dotsize=0.2, fillcolor=white](8.2,0.505)
\rput[bl](0.7,-1.395){$F_{10}$}
\end{pspicture}
} \\
\hline
$X(47,51,63,91)$ & $20$ & \psscalebox{.6 .6} 
{
\begin{pspicture}(0,-1.495)(7.64,1.795)
\psline[linecolor=black, linewidth=0.02](5.0,-0.495)(0.2,0.505)
\psframe[linecolor=white, linewidth=0.02, fillstyle=solid, dimen=outer](2.2,0.305)(1.8,-0.095)
\psline[linecolor=black, linewidth=0.02](1.0,-0.495)(6.6,-0.495)
\psline[linecolor=black, linewidth=0.02](0.2,0.505)(7.4,0.505)
\psdots[linecolor=black, dotstyle=o, dotsize=0.2, fillcolor=white](0.2,0.505)
\psdots[linecolor=black, dotstyle=o, dotsize=0.2, fillcolor=white](6.6,0.505)
\rput[bl](3.2,0.705){$-3$}
\rput[bl](2.4,0.705){$-1$}
\rput[bl](1.6,0.705){$-2$}
\rput[bl](0.8,0.705){$-6$}
\rput[bl](0.0,0.705){$-2$}
\rput[bl](4.0,0.705){$-7$}
\rput[bl](1.6,1.105){$F_3$}
\rput[bl](4.8,0.705){$-1$}
\rput[bl](5.6,0.705){$-2$}
\rput[bl](6.4,0.705){$-6$}
\rput[bl](4.8,1.105){$L_2$}
\rput[bl](7.2,0.705){$-2$}
\rput[bl](0.7,-0.995){$-7$}
\rput[bl](1.5,-0.995){$-3$}
\rput[bl](2.3,-0.995){$-1$}
\rput[bl](3.1,-0.995){$-7$}
\rput[bl](3.9,-0.995){$-3$}
\rput[bl](4.7,-0.995){$-1$}
\rput[bl](5.5,-0.995){$-7$}
\rput[bl](6.3,-0.995){$-3$}
\rput[bl](2.3,-1.395){$L_3$}
\rput[bl](0.0,1.105){$F_1$}
\rput[bl](0.8,1.105){$F_2$}
\rput[bl](2.4,1.105){$L_1$}
\rput[bl](3.2,1.105){$F_4$}
\rput[bl](4.0,1.105){$F_5$}
\rput[bl](5.6,1.105){$F_6$}
\rput[bl](6.4,1.105){$F_7$}
\rput[bl](7.2,1.105){$F_8$}
\psline[linecolor=black, linewidth=0.02](5.0,0.505)(6.6,-0.495)
\rput[bl](3.9,-1.395){$F_{12}$}
\rput[bl](5.5,-1.395){$F_{13}$}
\rput[bl](6.3,-1.395){$F_{14}$}
\psframe[linecolor=white, linewidth=0.02, fillstyle=solid, dimen=outer](5.8,0.305)(5.4,-0.095)
\psframe[linecolor=white, linewidth=0.02, fillstyle=solid, dimen=outer](4.0,-0.095)(3.6,-0.495)
\psline[linecolor=black, linewidth=0.02](2.6,-0.495)(7.4,0.505)
\psdots[linecolor=black, dotstyle=o, dotsize=0.2, fillcolor=white](7.4,0.505)
\psdots[linecolor=black, dotstyle=o, dotsize=0.2, fillcolor=white](1.8,0.505)
\psdots[linecolor=black, dotsize=0.2](5.0,0.505)
\psdots[linecolor=black, dotstyle=o, dotsize=0.2, fillcolor=white](1.8,-0.495)
\psdots[linecolor=black, dotsize=0.2](5.0,-0.495)
\psdots[linecolor=black, dotstyle=o, dotsize=0.2, fillcolor=white](4.2,-0.495)
\psdots[linecolor=black, dotstyle=o, dotsize=0.2, fillcolor=white](5.8,-0.495)
\psline[linecolor=black, linewidth=0.02](1.0,-0.495)(2.6,0.505)
\psdots[linecolor=black, dotstyle=o, dotsize=0.2, fillcolor=white](3.4,-0.495)
\rput[bl](4.7,-1.395){$L_4$}
\psdots[linecolor=black, dotstyle=o, dotsize=0.2, fillcolor=white](4.2,0.505)
\psdots[linecolor=black, dotstyle=o, dotsize=0.2, fillcolor=white](1.0,-0.495)
\psdots[linecolor=black, dotstyle=o, dotsize=0.2, fillcolor=white](1.0,0.505)
\rput[bl](3.1,-1.395){$F_{11}$}
\psdots[linecolor=black, dotstyle=o, dotsize=0.2, fillcolor=white](3.4,0.505)
\psdots[linecolor=black, dotstyle=o, dotsize=0.2, fillcolor=white](5.8,0.505)
\psdots[linecolor=black, dotsize=0.2](2.6,0.505)
\psdots[linecolor=black, dotsize=0.2](2.6,-0.495)
\psdots[linecolor=black, dotstyle=o, dotsize=0.2, fillcolor=white](6.6,-0.495)
\rput[bl](1.5,-1.395){$F_{10}$}
\rput[bl](0.7,-1.395){$F_9$}
\end{pspicture}
} \\ \hline

\end{longtable}

\end{center}

We now go case by case, showing what the support supp$(F)$ of $F$ is and its type (using Kodaira's notation), and showing $C$. Here we are choosing $F$ and $C$, there are other choices in general.

\begin{itemize}
\item[4)] supp$(F)=\sum_{i=1}^6 F_i + L_1 +L_2+ L_4 +F_{16}+F_{17}+F_{18}$, type $I_{12}$, $C=F_7$.

\item[5)] supp$(F)=F_1+F_{16}+F_{17}+L_4$, type $IV$, $C=F_2$.

\item[7)] supp$(F)=F_1+F_{16}+F_{17}+L_4$, type $III$, $C=F_{15}$.

\item[11)] supp$(F)=F_6+L_2+F_{17}+F_7$, type $II$, $C=F_{5}$.

\item[13)] supp$(F)=F_1+F_2+L_4+ L_3+F_8+ \sum_{i=10}^{15} F_i$, type $III^*$, $C=F_3$.

\item[17)] supp$(F)=L_2+  \sum_{i=7}^{9} F_i + F_{12}+L_3+F_{13}+F_{16}$, type $IV$, $C=F_{11}$.

\item[19)] supp$(F)=F_4+L_1+F_5+F_6+F_7+L_2+F_{15}$, type $II$, $C=F_{3}$.

\item[20)] supp$(F)=F_3+L_1+F_4+F_5+F_6+L_2+F_{14}$, type $II$, $C=F_{2}$.

\end{itemize}

\end{proof}

\subsection{$p_g \geq 2$ generic surfaces are of general type}

In this sub-section, we assume that $p_g\geq 2$. We recall that Koll\'ar surfaces are simply-connected. By classification of algebraic surfaces, the Kodaira dimension of the associate surface $Y$ is either $1$ or $2$. We first present families of explicit examples for each of the two possible Kodaira dimensions, and then we show the general picture for $w^*>>0$.

Let $g \colon Y' \to \P^2$ be the normal $w^*$-th root cover branch on $(L_1^{\mu_1} L_2^{\mu_2} L_3^{\mu_3} L_4^{\mu_4}=0)$, and let $f \colon Y \to \P^2$ be $g$ composed with the minimal resolution of singularities of $Y'$. Let $p_{i,j}=L_i \cap L_j$ for $i<j$. Let $E_{i,j,k}$ be the $k$-th exceptional curve over $p_{i,j}$. Then $$K_Y \equiv f^*\big(-3H + \frac{w^*-1}{w^*} (L_1+L_2+L_3+L_4)\big) - \sum_{i<j} \sum_k \left(1- \frac{\alpha_{i,j,k}+\beta_{i,j,k}}{w^*} \right)E_{i,j,k}$$ where $H$ is a line in $\P^2$, and so $$K_Y \equiv \frac{w^*-4}{4} \big(L'_1+L'_2+L'_3+L'_4\big) + \sum_{i<j} \sum_k \left(\frac{\alpha_{i,j,k}+\beta_{i,j,k}-4}{4}\right) E_{i,j,k},$$ where we are using notation and facts from the beginning of Section \ref{s2}, and $L'_i \simeq \P^1$ is the (reduced, irreducible) pre-image of $L_i$.

\begin{example}
Let $b \geq 2$. Consider $w^*=4(b-1)$, $\mu_1=\mu_2=1$, and $\mu_3=\mu_4=2b-3$. Then, over $p_{1,2}$ and $p_{3,4}$ we have $A_{w^*-1}$ singularities in $Y'$, and over the rest of the $p_{i,j}$ we have $\frac{1}{w^*}(1,2b-1)$. Notice that $\frac{w^*}{2b-1}=[2,b,2]$. We have that ${L'}_i^2=-2$, and
$$K_Y \equiv \frac{b-2}{2} \Big( 2 \sum_i L'_i + \sum_k 2(E_{1,2,k} + E_{3,4,k}) + (E_{1,3,k} + E_{1,4,k}+ E_{2,3,k} + E_{2,4,k}) \Big).$$ Therefore $Y$ is a minimal surface with $K_Y^2=0$ and $e(Y)=3 w^*+12$, and so $p_g(Y)=b-1$. The surface $Y$ is K3 when $b=2$, and Kodaira dimension $1$ when $b>2$. In fact, one can show that $E_{1,3,2}, E_{1,4,2}, E_{2,3,2}, E_{2,4,2}$ are sections (and $(-b)$-curves) for an elliptic fibration $Y \to \P^1$, and the complement of them in the support above of $K_Y$ give two $I_{w^*}^*$ singular fibers (using Kodaira notation).
\label{ex1}
\end{example}

\begin{example}
Let $b\geq 1$. Consider $w^*=28b+1$, $\mu_1=1$, $\mu_2=2$, $\mu_3=4$, and $\mu_4=28b-6$. Then, over $p_{i,j}$ we have:

\begin{itemize}
\item[$p_{1,2}$]: $\frac{1}{w^*}(1,w^*-2)$, $[2,\ldots,2,3]$ with $(14b-1)$ $2$'s

\item[$p_{1,3}$]: $\frac{1}{w^*}(1,7b)$, $[5,2,\ldots,2]$ with $(7b-1)$ $2$'s

\item[$p_{1,4}$]: $\frac{1}{w^*}(1,7)$, $[4b+1,2,2,2,2,2,2]$

\item[$p_{2,3}$]: $\frac{1}{w^*}(1,w^*-2)$, $[2,\ldots,2,3]$ with $(14b-1)$ $2$'s

\item[$p_{2,4}$]: $\frac{1}{w^*}(1,14b+4)$, $[2,2b+1,3,2,2]$

\item[$p_{3,4}$]: $\frac{1}{w^*}(1,7b+2)$, $[4,b+1,2,2,3]$
\end{itemize}

One can also compute that ${L'}_1^2={L'}_2^2={L'}_4^2=-2$ and ${L'}_3^2=-1$. The configuration of all these curves is shown in Figure \ref{curveConfGeneralType}.

\begin{figure}[htb]
\psscalebox{.7 .7} 
{
\begin{pspicture}(0,-2.085)(13.74,2.085)
\psline[linecolor=black, linewidth=0.02](6.7,0.815)(0.3,-0.785)
\psline[linecolor=black, linewidth=0.02](3.3,-0.785)(13.5,0.815)
\psframe[linecolor=white, linewidth=0.02, fillstyle=solid, dimen=outer](11.5,0.615)(10.9,0.215)
\psframe[linecolor=white, linewidth=0.02, fillstyle=solid, dimen=outer](6.5,-0.185)(6.1,-0.585)
\psframe[linecolor=white, linewidth=0.02, fillstyle=solid, dimen=outer](4.1,0.415)(3.7,0.015)
\psline[linecolor=black, linewidth=0.02](10.5,0.815)(13.3,-0.785)
\psline[linecolor=black, linewidth=0.02](8.7,-0.785)(0.3,0.815)
\psline[linecolor=black, linewidth=0.02](12.7,0.815)(13.5,0.815)
\psline[linecolor=black, linewidth=0.04, linestyle=dotted, dotsep=0.10583334cm](11.3,0.815)(12.7,0.815)
\psline[linecolor=black, linewidth=0.02](8.9,0.815)(11.3,0.815)
\psline[linecolor=black, linewidth=0.04, linestyle=dotted, dotsep=0.10583334cm](7.5,0.815)(8.9,0.815)
\psline[linecolor=black, linewidth=0.02](0.3,0.815)(7.5,0.815)
\psdots[linecolor=black, dotstyle=o, dotsize=0.2, fillcolor=white](0.3,0.815)
\psdots[linecolor=black, dotstyle=o, dotsize=0.2, fillcolor=white](1.1,0.815)
\psdots[linecolor=black, dotstyle=o, dotsize=0.2, fillcolor=white](1.9,0.815)
\psdots[linecolor=black, dotstyle=o, dotsize=0.2, fillcolor=white](2.7,0.815)
\psdots[linecolor=black, dotstyle=o, dotsize=0.2, fillcolor=white](3.5,0.815)
\psdots[linecolor=black, dotstyle=o, dotsize=0.2, fillcolor=white](4.3,0.815)
\psdots[linecolor=black, dotstyle=o, dotsize=0.2, fillcolor=white](5.5,0.815)
\psdots[linecolor=black, dotstyle=o, dotsize=0.2, fillcolor=white](6.7,0.815)
\psdots[linecolor=black, dotstyle=o, dotsize=0.2, fillcolor=white](7.5,0.815)
\psdots[linecolor=black, dotstyle=o, dotsize=0.2, fillcolor=white](8.9,0.815)
\rput[bl](4.1,1.015){$-2$}
\rput[bl](4.9,1.015){$-4b-1$}
\rput[bl](3.3,1.015){$-2$}
\rput[bl](2.5,1.015){$-2$}
\rput[bl](1.7,1.015){$-2$}
\rput[bl](0.9,1.015){$-2$}
\rput[bl](0.1,1.015){$-2$}
\rput[bl](6.5,1.015){$-2$}
\rput[bl](6.5,1.415){$L_1$}
\rput[bl](7.3,1.015){$-2$}
\rput[bl](8.7,1.015){$-2$}
\psdots[linecolor=black, dotstyle=o, dotsize=0.2, fillcolor=white](9.7,0.815)
\rput[bl](9.5,1.015){$-3$}
\psdots[linecolor=black, dotstyle=o, dotsize=0.2, fillcolor=white](10.5,0.815)
\rput[bl](10.3,1.015){$-2$}
\rput[bl](10.3,1.415){$L_2$}
\psdots[linecolor=black, dotstyle=o, dotsize=0.2, fillcolor=white](11.3,0.815)
\psdots[linecolor=black, dotstyle=o, dotsize=0.2, fillcolor=white](12.7,0.815)
\rput[bl](11.1,1.015){$-2$}
\rput[bl](12.5,1.015){$-2$}
\psdots[linecolor=black, dotstyle=o, dotsize=0.2, fillcolor=white](13.5,0.815)
\rput[bl](13.3,1.015){$-3$}
\psline[linecolor=black, linewidth=0.02](7.5,1.415)(7.5,1.615)(9.1,1.615)(9.1,1.415)
\psline[linecolor=black, linewidth=0.02](11.3,1.415)(11.3,1.615)(12.9,1.615)(12.9,1.415)
\rput[bl](7.7,1.815){$14b-1$}
\rput[bl](11.5,1.815){$14b-1$}
\psdots[linecolor=black, dotsize=0.2](3.3,-0.785)
\psline[linecolor=black, linewidth=0.02](1.7,-0.785)(13.3,-0.785)
\psline[linecolor=black, linewidth=0.04, linestyle=dotted, dotsep=0.10583334cm](0.3,-0.785)(1.7,-0.785)
\rput[bl](0.0,-1.285){$-2$}
\rput[bl](1.4,-1.285){$-2$}
\rput[bl](2.2,-1.285){$-5$}
\rput[bl](3.0,-1.285){$-1$}
\rput[bl](3.8,-1.285){$-4$}
\rput[bl](4.6,-1.285){$-b-1$}
\psdots[linecolor=black, dotstyle=o, dotsize=0.2, fillcolor=white](5.1,-0.785)
\rput[bl](6.0,-1.285){$-2$}
\psdots[linecolor=black, dotstyle=o, dotsize=0.2, fillcolor=white](6.3,-0.785)
\rput[bl](6.8,-1.285){$-2$}
\psdots[linecolor=black, dotstyle=o, dotsize=0.2, fillcolor=white](7.1,-0.785)
\rput[bl](7.6,-1.285){$-3$}
\psdots[linecolor=black, dotstyle=o, dotsize=0.2, fillcolor=white](7.9,-0.785)
\rput[bl](8.4,-1.285){$-2$}
\psdots[linecolor=black, dotstyle=o, dotsize=0.2, fillcolor=white](8.7,-0.785)
\rput[bl](3.0,-1.685){$L_3$}
\rput[bl](8.4,-1.685){$L_4$}
\rput[bl](9.2,-1.285){$-2$}
\psdots[linecolor=black, dotstyle=o, dotsize=0.2, fillcolor=white](9.5,-0.785)
\rput[bl](10.0,-1.285){$-2$}
\psdots[linecolor=black, dotstyle=o, dotsize=0.2, fillcolor=white](10.3,-0.785)
\rput[bl](10.8,-1.285){$-3$}
\psdots[linecolor=black, dotstyle=o, dotsize=0.2, fillcolor=white](11.1,-0.785)
\rput[bl](11.6,-1.285){$-2b-1$}
\psdots[linecolor=black, dotstyle=o, dotsize=0.2, fillcolor=white](12.1,-0.785)
\rput[bl](13.0,-1.285){$-2$}
\psdots[linecolor=black, dotstyle=o, dotsize=0.2, fillcolor=white](13.3,-0.785)
\psdots[linecolor=black, dotstyle=o, dotsize=0.2, fillcolor=white](4.1,-0.785)
\psdots[linecolor=black, dotstyle=o, dotsize=0.2, fillcolor=white](2.5,-0.785)
\psdots[linecolor=black, dotstyle=o, dotsize=0.2, fillcolor=white](1.7,-0.785)
\psdots[linecolor=black, dotstyle=o, dotsize=0.2, fillcolor=white](0.3,-0.785)
\psline[linecolor=black, linewidth=0.02](0.3,-1.385)(0.3,-1.585)(1.7,-1.585)(1.7,-1.385)
\rput[bl](0.6,-2.085){$7b-1$}
\end{pspicture}
}
\caption{Curve configuration of a general type example.}
\label{curveConfGeneralType}
\end{figure}
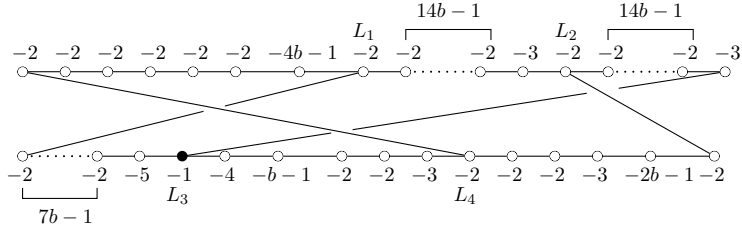

One can verify that $\alpha_{i,j,k}+\beta_{i,j,k}>4$ for all $i,j,k$. Therefore, by the formula above, $K_Y$ can be written with positive coefficients supported in the configuration of curves, so that to obtain the minimal model $Y''$ of $Y$ we only need to contract $L'_3$ since $\frac{w^*-4}{4}>1$ (and see the figure). We compute using the formulas above: $K_{Y''}^2=7(3b-1)$, $e(Y'')=63b+19$, and $p_g(Y'')=7b$. In this way, $Y''$ is of general type for any $b$.
\label{ex2}
\end{example}

We now consider prime numbers $w^*>>0$ and partitions $$\mu_1+\mu_2+\mu_3+\mu_4=w^*$$ with $0<\mu_i<w^*$. Let $\SS$ be the set of all partitions. Then, as we did before, there are smooth projective surfaces $Y$ constructed as $w^*$-th root covers $Y \to Y' \to \P^2$, and there are infinitely many Koll\'ar surfaces $X(a_1,a_2,a_3,a_4)$ birational to each $Y$. Let $X_{\text{min}}$ be a minimal (smooth) model for $Y$ (and so for all $X(a_1,a_2,a_3,a_4)$). The following is based on \cite{Urz10,Urz15}.

\begin{theorem}
There is $\SS' \subset \SS$ with $\SS'/{w^*} \to 0$ as $w^*>>0$ such that if $\{\mu_1,\mu_2,\mu_3,\mu_4\} \in \SS \setminus \SS'$, then $X_{\text{min}}$ is a simply-connected surface of general type with $K_{X_{\text{min}}}^2/ e(X_{\text{min}}) \to 1$ as $w^*>>0$.
\label{generic}
\end{theorem}

\begin{proof}
By Proposition \ref{ek^2}, we have $e(Y)=w^*+2+\sum_{i<j} l(\mu_i,\mu_j;w^*)$, and $$K_Y^2=w^*+\frac{4}{w^*}+4 + \sum_{i<j} 12s(\mu_i,\mu_j;w^*)-l(\mu_i,\mu_j;w^*).$$
Notice that by Theorem 4.1 in \cite{Urz15}, both $e(Y)>>0$ and $K_Y^2>>0$. In particular $Y$ is of general type by classification of algebraic surfaces. We also note that $K_{Y'}$ is ample since it is numerically $(1-4/w^*)$ times the pull-back of the class of a line. Thus, by Theorem 4.3 in \cite{Urz15}, the number of potential $(-1)$-curves to be contracted over $w^*$ tends to zero as $w^*$ approaches infinity, and so $X_{\text{min}}$ satisfies $K_{X_{\text{min}}}^2/ e(X_{\text{min}}) \to 1$ as $w^*>>0$.
\end{proof}



\begin{thebibliography}{99}

\bibitem[Ba77]{Ba77}
    P. Barkan,
    \emph{Sur les sommes de {D}edekind et les fractions continues finies},
    C. R. Acad. Sci. Paris S\'er. A-B, v.284, 1977, no.16, A923--A926.

\bibitem[BHPV]{BHPV04}
    W. P. Barth, K. Hulek, C. A. M. Peters, A. Van de Ven,
    \emph{Compact complex surfaces},
    Ergebnisse der Mathematik und ihrer Grenzgebiete. 3. Folge., second edition, vol. 4, Springer-Verlag, Berlin, 2004.

\bibitem[Dolg82]{Dolg82}
    I. Dolgachev,
    \emph{Weighted projective varieties},
    Group actions and vector fields, Springer, 34--71, 1982.

\bibitem[EV92]{EV92}
    H. Esnault, and E. Viehweg,
    \emph{Lectures on vanishing theorems},
    DMV Seminar, vol. 20, Birkh\"auser Verlag, Basel, 1992.

\bibitem[Girs16]{Girs16}
    K. Girstmair,
    \emph{The largest values of Dedekind sums},
    arXiv:1607.07682, pre-print 2016.

\bibitem[GL97]{GL97}
    G. Gonzalez-Sprinberg, and M. Lejeune-Jalabert,
    \emph{Families of smooth curves on singularities and wedges}
    Ann. Polon. Math. 67(1997), no. 2, 179--190.

\bibitem[Ian00]{Ian00}
    A. R. Iano-Fletcher,
    \emph{Working with weighted complete intersections},
    Explicit birational geometry of 3-folds, 3, 101--173, 2000.

\bibitem[Hart77]{Hart77},
    R. Hartshorne,
    \emph{Algebraic geometry},
    Graduate Texts in Mathematics, No. 52, Springer-Verlag, New York-Heidelberg, 1977.

\bibitem[HiZa74]{HiZa74}
    F. Hirzebruch, and D. Zagier,
    \emph{The {A}tiyah-{S}inger theorem and elementary number theory},
    Mathematics Lecture Series, No.3, Publish or Perish Inc., Boston, Mass., 1974.

\bibitem[HK12]{HK12}
    D. Hwang, J. Keum,
    \emph{Construction of singular rational surfaces of Picard number one with ample canonical divisor},
    Proc. Amer. Math. Soc. 140(2012), no. 6, 1865--1879.

\bibitem[Ko08]{Ko08}
    J. Koll\'ar,
    \emph{Is there a topological Bogomolov-Miyaoka-Yau inequality?},
    Pure Appl. Math. Q. 4(2008), no. 2, part 1, 203--236.

\bibitem[R03]{R}
	M. Reid,
	\emph{Surface cyclic quotient singularities and Hirzebruch - Jung resolution},
	(available from \verb|homepages.warwick.ac.uk/|$\sim$\verb|masda/surf/more/cyclic.pdf|).

\bibitem[Urz10]{Urz10}
    G. Urz\'ua,
    \emph{Arrangements of curves and algebraic surfaces},
    J. of Algebraic Geom. 19(2010), 335--365.

\bibitem[Urz15]{Urz15}
    G. Urz\'ua (with an appendix by R. Codorniu),
    \emph{Chern slopes of surfaces of general type in positive characteristic},
    arXiv:1509.05260, to appear in Duke Mathematical Journal.

\end{thebibliography}
\end{document}